\documentclass[12pt,oneside]{amsbook}

\usepackage{fullpage}

\usepackage{amsmath, amscd, amsthm, amssymb}

\usepackage{hyperref}

\usepackage{epic}
\usepackage{amscd}
\usepackage{amsfonts}
\usepackage[normalem]{ulem}
\usepackage[margin=1in]{geometry}

\usepackage[all,arc]{xy}
\usepackage{enumerate}
\usepackage{mathrsfs}
\usepackage{color}   
\usepackage{hyperref}
\hypersetup{
    colorlinks=true, 
    linktoc=all,     
    linkcolor=blue,  
}
\usepackage{amsthm}
\usepackage{amsmath}
\usepackage{graphicx}
\usepackage{wasysym}
\usepackage{pgf,tikz}
\usetikzlibrary{positioning}
\usepackage{ marvosym }
\usepackage{tikz-cd}
\usepackage{ textcomp }
\usepackage{bbm}
\usepackage{ stmaryrd }

\def\C{{\mathbf C}}
\def\R{{\mathbf R}}

\def\Z{{\mathbf Z}}
\def\Q{{\mathbf Q}}
\def\A{{\mathbf A}}

\def\ZZ{\widehat{\mathbf Z}}

\newtheorem{theorem}{Theorem}[subsection]

\newtheorem{lemma}[theorem]{Lemma}

\newtheorem{proposition}[theorem]{Proposition}

\newtheorem{corollary}[theorem]{Corollary}

\theoremstyle{definition}
\newtheorem{definition}[theorem]{Definition}

\theoremstyle{remark}
\newtheorem{remark}[theorem]{Remark}

\newcommand{\mm}[4]{\left(\begin{smallmatrix} #1 & #2\\ #3 & #4\end{smallmatrix}\right)}

\DeclareMathOperator{\val}{val}
\DeclareMathOperator{\tr}{tr}

\DeclareMathOperator{\SO}{SO}

\DeclareMathOperator{\Sp}{Sp}
\DeclareMathOperator{\GSp}{GSp}

\DeclareMathOperator{\SU}{SU}

\DeclareMathOperator{\SL}{SL}
\DeclareMathOperator{\GL}{GL}

\DeclareMathOperator{\diag}{diag}

\DeclareMathOperator{\Span}{Span}

\def\g{{\mathfrak g}}

\def\t{{\mathfrak t}}
\def\h{{\mathfrak h}}
\def\k{{\mathfrak k}}
\def\p{{\mathfrak p}}

\def\m{{\mathfrak m}}

\def\sl{{\mathfrak {sl}}}

\def\G{\widetilde{G}}

\def\Vm{{\mathbb{V}}}

\def\W{{\mathbf W}}

\newcommand{\la}{\langle}
\newcommand{\ra}{\rangle}
\newcommand{\lra}{\longrightarrow}

\newcommand{\al}{\alpha}
\newcommand{\be}{\beta}
\newcommand{\ga}{\gamma}
\newcommand{\de}{\delta}
\newcommand{\De}{\Delta}
\newcommand{\Ga}{\Gamma}
\newcommand{\ep}{\epsilon}
\newcommand{\lam}{\lambda}

\DeclareMathOperator{\Ind}{Ind}

\DeclareMathOperator{\Sym}{Sym}

\DeclareMathOperator{\Spec}{Spec}

\begin{document}
\title{Modular forms of half-integral weight on exceptional groups}
\author{Spencer Leslie}
\address{Department of Mathematics\\ Boston College\\ Chestnut Hill, MA USA}
\email{spencer.leslie@bc.edu}
\author{Aaron Pollack}
\address{Department of Mathematics\\ The University of California San Diego\\ La Jolla, CA USA}
\email{apollack@ucsd.edu}

\thanks{SL has been supported by an AMS-Simons Travel Award and by NSF grant DMS-1902865. AP has been supported by the Simons Foundation via Collaboration Grant number 585147, by the NSF via grant numbers 2101888 and 2144021, and by an AMS Centennial Research Fellowship.}

\begin{abstract} We define a notion of modular forms of half-integral weight on the quaternionic exceptional groups.  We prove that they have a well-behaved notion of Fourier coefficients, which are complex numbers defined up to multiplication by $\pm 1$.  We analyze the minimal modular form $\Theta_{F_4}$ on the double cover of $F_4$, following Loke--Savin and Ginzburg.  Using $\Theta_{F_4}$, we define a modular form of weight $\frac{1}{2}$ on (the double cover of) $G_2$.  We prove that the Fourier coefficients of this modular form on $G_2$ see the $2$-torsion in the narrow class groups of totally real cubic fields.
\end{abstract}
\maketitle

\setcounter{tocdepth}{1}
\tableofcontents



\chapter{Introduction}

\section{Main result}
We introduce our main result by way of an analogy.  Let $\Theta(z) = \sum_{n \in \Z}{q^{n^2}}$, where $q = e^{2\pi i z}$.  As is well-known, $\Theta(z)$ is a classical holomorphic modular form of weight $\frac{1}{2}$ and level $\Ga_1(4) \subseteq \SL_2(\Z)$.  Consider the weight $\frac{3}{2}$ modular form
\[E_{CZ}(z) := \Theta(z)^3 = \sum_{n \geq 0}{r_3(n)q^n};\] here $r_3(n):= \#\{(n_1,n_2,n_3)\in \Z^3: n = n_1^2+n_2^2+n_3^2\}$ is the number of ways $n$ can be written as the sum of three squares. We have named this modular form after Cohen and Zagier, in light of their papers \cite{cohen}, \cite{zagier}.

Recall now the following theorem of Gauss:
\begin{theorem}[Gauss] Suppose $n$ is squarefree, $n \equiv 1,2 \pmod{4}$ and $n \geq 4$.  Then $r_3(n) = 12 \cdot |\mathrm{Cl}(\Q(\sqrt{-n}))|$, $12$ times the class number of the associated quadratic imaginary field.
\end{theorem}
Thus the Fourier coefficients of $E_{CZ}(z)$ see the class numbers of imaginary quadratic fields.  Our main result is the construction of an analogous modular form $\Theta_{G_2}$ of weight $\frac{1}{2}$ on $G_2$, whose Fourier coefficients see the $2$-torsion in the narrow class groups of totally real cubic fields. In particular, we define a notion of modular forms of half-integral weight on certain exceptional groups, very similar to the integral weight theory \cite{ganGrossSavin}. We prove that these modular forms, which are now automorphic forms on certain non-linear double covers of these exceptional groups, have a robust notion of Fourier coefficients. We then construct a particular interesting example $\Theta_{G_2}$ on $G_2$ and partially calculate its Fourier expansion.

To motivate our construction of $\Theta_{G_2}$, observe that one has a commuting pair $\SL_2 \times \SO(3) \subseteq \Sp_6$.  One can also think of $E_{CZ}(z)$ as the restriction to $\SL_2$ of a weight $\frac{1}{2}$ Siegel modular theta-function: $E_{CZ}(z) = \Theta_{Sp_6}(\diag(z,z,z))$, where
\[\Theta_{Sp_6}(Z) = \sum_{v = (n_1,n_2,n_3) \in \Z^3}{e^{2\pi i v Z v^t}}\]
and $Z$ is in the Siegel upper half space of degree three.  Now, there is a commutative diagram of inclusions
\[\begin{array}{ccc}
\Sp_{6} & & F_4 \\
\bigcup & \subseteq & \bigcup \\
\SL_2 \times \SO(3) & & G_2 \times \SO(3) \end{array}.\]
Following Loke--Savin \cite{lokeSavin} and Ginzburg \cite{ginzburgF4} we consider the automorphic minimal representation on the double cover of $F_4$.  We show that the minimal representation can be used to define a weight $\frac{1}{2}$ modular form $\Theta_{F_4}$ on $F_4$, and define $\Theta_{G_2}$ as the pullback to $G_2$ of $\Theta_{F_4}$.

The Fourier coefficients of modular forms $\varphi$ on $G_2$ are parametrized by integral binary cubic forms $f(u,v) = au^3 + bu^2v+cuv^2 + dv^3$, $a,b,c,d \in \Z$, for which $f(u,v)$ splits into three linear factors over the real numbers.  So, for each such binary cubic $f$, there is an associated Fourier coefficient $a_{\varphi}(f)$, which is a complex number well-defined up to multiplication by $\pm 1$.  Our main result is the explicit description of the Fourier coefficients of the weight $\frac{1}{2}$ modular form $\Theta_{G_2}$.  More precisely, we can explicitly compute these Fourier coefficients $a_{\Theta_{G_2}}(f)$ when the binary cubic $f(u,v)$ has $d=1$.  We explicate the special case of this result when the cubic ring $\Z[y]/(f(1,y))$ is a maximal order in a totally real cubic field.

\begin{theorem}\label{thm:introMain} There is a modular form $\Theta_{G_2}$ of weight $\frac{1}{2}$ on $G_2$ whose Fourier coefficients satisfy the following: Suppose $f(u,v) =  au^3 + bu^2v+cuv^2 + dv^3$ is an integral binary cubic form with $d=1$, and that the cubic ring $R=\Z[y]/(f(1,y))$ is a maximal order in a totally real cubic field $E = R \otimes \Q$.
\begin{enumerate}
	\item If the inverse different $\mathfrak{d}_R^{-1}$ is not a square in the narrow class group of $E$, then the Fourier coefficient $a_{\Theta_{G_2}}(f) = 0$.
	\item If the inverse different $\mathfrak{d}_R^{-1}$ is a square in the narrow class group of $E$, then the Fourier coefficient $a_{\Theta_{G_2}}(f) = \pm 24 |\mathrm{Cl}_{E}^+[2]|$, plus or minus $24$ times the size of the two-torsion in the narrow class group of $E$.
\end{enumerate}
\end{theorem}
Thus, in both cases of Theorem \ref{thm:introMain}, the Fourier coefficient of $\Theta_{G_2}$ corresponding to the binary cubic $f$ is $\pm 24$ times the number of square roots of the inverse different $\mathfrak{d}_R^{-1}$ in the narrow class group $\mathrm{Cl}_E^{+}$ of $E$.

\section{Extended introduction}
In this section we outline the contents of the paper.

\subsection{Quaternionic modular forms} As our main results concern modular forms of \emph{half-integral} weight on the quaternionic exceptional groups, we begin by reviewing the integral weight theory. To set the stage for these quaternionic modular forms, we first recall holomorphic modular forms.

Thus suppose $G$ is a semisimple algebraic $\Q$-group whose associated symmetric space is a Hermitian tube domain.  Then $G$ has a notion of \emph{holomorphic modular forms}.  These can be thought of as very special automorphic forms for $G$, which are closely connected to arithmetic.  They have a classical Fourier expansion and Fourier coefficients, and these Fourier coefficients often encode arithmetic data.

Among the exceptional Dynkin types, only $E_6$ and $E_7$ have a real form with a Hermitian symmetric space, and only $E_7$ has a real form with an Hermitian tube domain.  So, if one is interested in studying a class of special automorphic forms on, say, $G_2, F_4$ or $E_8$, there is not an obvious place to look for such objects.  Nevertheless, beginning with work of Gross and Wallach \cite{grossWallachI,grossWallachII} and developed in work of Wallach \cite{wallach} and Gan--Gross--Savin \cite{ganGrossSavin}, a theory of special automorphic forms on the exceptional algebraic groups began to emerge.

These special automorphic forms have been dubbed \emph{quaternionic modular forms}.  For each exceptional Dynkin type, there is a so-called \emph{quaternionic} real form: for $G_2$ and $F_4$, this is the split real form, while for $E_6$, $E_7$ and $E_8$ this is the real form with real rank equal to four. The quaternionic modular forms are special automorphic forms on reductive groups $G$ over $\Q$ for which $G(\R)$ is a quaternionic real group.

The real quaternionic exceptional groups never have a symmetric space with complex structure.  However, these groups share similar structures, and the quaternionic modular forms on these groups share similar properties.  To be more specific, suppose $G$ is an adjoint exceptional group with $G(\R)$ quaternionic.  Then the maximal compact subgroup $K_G$ of $G(\R)$ is of the form $(\SU(2) \times L)/\mu_2(\R)$, for a compact group $L$ that depends upon $G$.  Let $\Vm_2$ denote the standard representation of $\SU(2)$ and for a positive integer $\ell$ let $\mathbf{V}_{\ell}$ denote the representation of $K_G$ that is the representation $Sym^{2\ell}(\Vm_2)$ of the $\SU(2)$ factor and the trivial representation of the $L$-factor.  If $\Gamma \subseteq G(\R)$ is a congruence subgroup, a quaternionic modular form on $G$ of weight $\ell$ and level $\Gamma$ is an automorphic function $\varphi: \Gamma \backslash G(\R) \rightarrow \mathbf{V}_{\ell}$ satisfying
\begin{enumerate}
	\item $\varphi(gk) = k^{-1} \cdot \varphi(g)$ for all $k \in K_G$ and $g \in G(\R)$
	\item $D_{\ell} \varphi \equiv 0$ for a certain specific differential operator $D_{\ell}$.
\end{enumerate}
This is the definition from \cite{pollackQDS}, which is a slight generalization and paraphrase of the definition from \cite{ganGrossSavin}, where quaternionic modular forms are defined in terms of the quaternionic discrete series representations of the group $G(\R)$.

To make this definition precise, of course we must specify the differential operator $D_{\ell}$.  We do this now.  Let the notation be as above.  Write $\g_0 = \k_0 \oplus \p_0$ for the Cartan decomposition of the Lie algebra $\g_0$ of $G(\R)$.  Then, as a representation of $K_G$, one has $\p := \p_0 \otimes \C \simeq \Vm_2 \otimes W$ for a certain symplectic representation $W$ of $L$.  Let $\{X_\alpha\}_\alpha$ be a basis of $\p$ and $\{X_\alpha^\vee\}_\alpha$ be the dual basis of $\p^\vee$.  For $\varphi$ satisfying $\varphi(gk) = k^{-1} \cdot \varphi(g)$, define $\widetilde{D}_{\ell} \varphi = \sum_{\alpha}{X_\alpha \varphi \otimes X_\alpha^\vee}$.  Here $X_\alpha \varphi$ denotes the right regular action, and $\widetilde{D}_{\ell} \varphi$ is valued in
\[\mathbf{V}_{\ell} \otimes \p^\vee \simeq Sym^{2\ell+1}(\Vm_2) \boxtimes W \oplus Sym^{2\ell-1}(\Vm_2) \boxtimes W.\]
We let $\mathrm{pr}: \mathbf{V}_{\ell} \otimes \p^\vee \rightarrow Sym^{2\ell-1}(\Vm_2) \boxtimes W$ be the $K_G$-equivariant projection and define $D_{\ell} = \mathrm{pr}\circ \widetilde{D}_{\ell}$.

The relationship of the definition of quaternionic modular forms with representation theory is as follows.  Suppose $\pi$ is an irreducible $(\g_0,K_G)$-module embedded in the space of automorphic forms on $\Gamma\backslash G(\R)$ via a map $\alpha$.  Suppose moreover that $\pi$ has minimal $K_G$-type $\mathbf{V}_{\ell}$.  Then out of $\mathbf{V}_{\ell}$ and $\alpha$ one can construct a quaternionic modular form of weight $\ell$: One defines $\varphi(g) = \sum_{j = -\ell}^{\ell}{\alpha(x_j)(g) \otimes x_j^\vee}$, where $\{x_j\}$ is a basis of $\mathbf{V}_\ell \subseteq \pi_{\ell}$ and $x_j^\vee$ is the dual basis of $\mathbf{V}_{\ell}^\vee \simeq \mathbf{V}_{\ell}$.  Using the fact that $\mathbf{V}_{\ell}$ is the minimal $K$-type of $\pi$, it is easy to show that $\varphi$ is a quaternionic modular form of weight $\ell$.  If $\ell$ is sufficiently large depending on $G$, then there is a discrete series representation $\pi_\ell$ of $G(\R)$ whose minimal $K_G$-type is $\mathbf{V}_{\ell}$, so embeddings of these discrete series representations into the space of automorhpic forms on $G$ give rise to quaternionic modular forms of weight $\ell$.

Modular forms of integral weight $\ell$ have been studied in \cite{ganGrossSavin}, \cite{weissmanD4}, \cite{pollackQDS, pollackE8, pollackCuspidal, pollackNTM} and \cite{dalal}.  For an introduction to what is known about them, we refer to \cite{pollackAWSNotes}.  The main result of \cite{pollackQDS} is that quaternionic modular forms have a robust, semi-classical Fourier expnansion, similar to the Fourier expansion of classical holomorphic modular forms on tube domains.  This result generalized and refined work of Wallach \cite{wallach}.

To explain this Fourier expansion, we recall another common feature of the quaternionic exceptional groups.  While none of them has a parabolic with abelian unipotent radical, they all have a Heisenberg parabolic $P = MN$ whose unipotent radical $N \supseteq Z = [N,N] \supseteq 1$ is two-step, with the center $Z$ one-dimensional.  Thus if $\varphi$ is an automorphic form on $G$, one can take the constant term $\varphi_Z$ of $\varphi$ along $Z$, and Fourier expand the result along $N/Z$: $\varphi_Z = \sum_{\chi}{\varphi_\chi}$ where $\varphi_\chi(g) = \int_{N(\Z)\backslash N(\R)}{\chi^{-1}(n)\varphi(ng)\,dn}$.  The main result of \cite{pollackQDS} is an explication of this Fourier expansion for quaternionic modular forms $\varphi$ of weight $\ell$.  Namely, it is proved in \cite{pollackQDS} that there are certain completely explicit functions $W_\chi: G(\R) \rightarrow \mathbf{V}_{\ell}$ so that if $\varphi$ is a weight $\ell$ modular form, then $\varphi_\chi(g) = a_{\varphi}(\chi) W_\chi(g)$ for some complex number $a_{\varphi}(\chi)$.  The complex numbers $a_\varphi(\chi)$ are called the Fourier coefficients of $\varphi$.

While defined in a purely transcendental way, the Fourier coefficients $a_\varphi(\chi)$ of a quaternionic modular form $\varphi$ appear to have arithmetic significance; for evidence of this claim, see \cite{pollackE8,pollackCuspidal,pollackNTM}.  One purpose of the present paper is to add to this growing evidence that quaternionic modular forms have arithmetically-interesting Fourier coefficients.

\subsection{The double cover of quaternionic exceptional groups}
In this paper, we define and study certain quaternionic modular forms of \emph{half-integral} weight and their Fourier coefficients.  To define these notions, suppose again that $G$ is an adjoint quaternionic exceptional group.  Then, because $G(\R)$ deformation retracts onto $K_G \simeq (\SU(2) \times L)/\mu_2(\R)$, and $K_G$ has a two-cover $\widetilde{K}_G \simeq \SU(2) \times L$, the group $G(\R)$ has a two cover $\widetilde{G}$.  Choosing a basepoint of $\widetilde{G}$ above $1 \in G(\R)$ makes $\widetilde{G}$ into a connected Lie group, which is a central $\mu_2(\R)$-extension of $G(\R)$
\[
1 \rightarrow \mu_2(\R) \rightarrow \widetilde{G} \rightarrow G(\R) \rightarrow 1,\] and $\widetilde{K}_G$ can be identified with a maximal compact subgroup of $\widetilde{G}$.

Suppose $\ell \geq 1$ is an odd integer.  Let $\mathbf{V}_{\ell/2} = Sym^{\ell}(\Vm_2)$ be the representation of $\widetilde{K}_G$ that is the $\ell^{th}$ symmetric power of $\Vm_2$, as a representation of $\SU(2)$, and is the trivial representation of $L$.  Suppose moreover that  $\Gamma \subseteq G(\R)$ is a congruence subgroup, equipped with a splitting $s_{\Gamma}: \Gamma \rightarrow \widetilde{G}$.  We define a quaternionic modular form $\varphi$ for $G$ of weight $\ell/2$ and level $(\Gamma,s_\Gamma)$ to be a $\mathbf{V}_{\ell/2}$-valued automorphic function  $\varphi:s_\Gamma(\Gamma)\backslash\widetilde{G} \rightarrow \mathbf{V}_{\ell/2}$ that satisfies
\begin{enumerate}
	\item $\varphi(gk) = k^{-1} \cdot \varphi(g)$ for all $g \in \widetilde{G}$ and $k \in \widetilde{K}_G$ and
	\item $D_{\ell/2} \varphi \equiv 0.$
\end{enumerate}
Here the differential operator $D_{\ell/2}$ is defined exactly as $D_{\ell}$ was above.

Our first result, which is perhaps of independent interest, is an explicit description of these Lie groups $\widetilde{G}$.  To motivate it, let $\mathfrak{h} = \SL_2(\R)/\SO(2)$ denote the upper half plane and recall that one can identify the double cover of $\SL_2(\R)$ with pairs $(g,j_g)$ where $g = \mm{a}{b}{c}{d} \in \SL_2(\R)$ and $j_g: \mathfrak{h} \rightarrow \C^\times$ is a holomorphic function that satisfies $j_g(z)^2 = cz+d$.  If now $G$ is an adjoint quaternionic exceptional group, with symmetric space $X_G = G(\R)/K_G$, we define a factor of automorphy $j_{lin}: G(\R) \times X_G \rightarrow \GL_3(\C)$; it satisfies $j_{lin}(g_1g_2,x) = j_{lin}(g_1, g_2 x) j_{lin}(g_2,x)$.  We then consider the set of pairs $(g,j_g)$ with $g \in G(\R)$ and $j_g: X_G \rightarrow \GL_2(\C)$ continuous that satisfy $Sym^2(j_g(x)) = j_{lin}(g,x)$. It is easy to see that this set forms a group with multiplication $(g_1,j_{g_1}(x))(g_2,j_{g_2}(x)) = (g_1 g_2, j_{g_1}(g_2 x)j_{g_2}(x))$.

\begin{theorem}\label{thm:introCover} With a certain topology on the set of pairs $(g,j_g)$ above, this set can be identified with the connected topological group $\widetilde{G}$.
\end{theorem}

To study modular forms of half-integral weight on the group $\widetilde{G}$, it helps to have explicit congruence subgroups $\Gamma \subseteq G$ together with an explicit splitting $s_\Gamma: \Gamma \rightarrow \widetilde{G}$.  This is accomplished in the following result in case $G$ is $G_2$ or $F_4$.
\begin{theorem}\label{thm:introSplit} When $G$ is $G_2$ and $F_4$, there are explicit, large congruence subgroups $\Gamma_{G_2}(4)$ and $\Gamma_{F_4}(4)$ that split into the double cover.
\end{theorem}

To define the groups $\Gamma_{F_4}(4)$ and $\Gamma_{G_2}(4)$ and to prove Theorem \ref{thm:introSplit}, we work adelically. When $G$ is a split, simply-connected algebraic group, Steinberg \cite{steinbergBook} and Matsumoto \cite{matsumoto} have defined a $2$-cover $\G^{(2)}(k)$ of $G(k)$ for every local field $k$.  This group can be constructed by generators and relations \cite{steinbergBook}, as we recall in Section \ref{Sec: double covers gen}.  The groups $\G^{(2)}(\Q_v)$ can be glued together to produce a $2$-cover $\G^{(2)}(\A)$ of $G(\A)$.  It follows from the construction of $\G^{(2)}(\A)$ and the global triviality of the Hilbert symbol that the group of rational points $G(\Q)$ splits into $\G^{(2)}(\A)$.

When $G$ is $G_2$ or $F_4$, the group $\G^{(2)}(\R)$ can be identified with the group $\widetilde{G}$ above.  Using this fact, and the global splitting $G(\Q) \rightarrow \G^{(2)}(\A)$, constructing congruence subgroups $\Gamma \subseteq G(\R)$ that admit a splitting into $\widetilde{G}$ amounts to constructing compact open subgroups $K_p$ of $G(\Q_p)$ that admit a splitting into $\G^{(2)}(\Q_p)$.  When $p > 2$, it is proved by Loke and Savin \cite{lokeSavin} that the hyperspecial maximal compact subgroup of $G(\Q_p)$ splits into $\G^{(2)}(\Q_p)$.  Thus it remains to analyze the case $p=2$, and it is here where we do detailed computations: we produce an explicit (non-maximal) compact open subgroup of $F_4(\Q_2)$ that splits into the double cover.  Our result in this direction can be considered an extension of some work of \cite{karasiewicz}, who considers the simply-laced case.  This yields the split congruence subgroup $\Gamma_{F_4}(4)$, and we define $\Gamma_{G_2}(4) = \Gamma_{F_4}(4) \cap G_2(\R)$ under an embedding $G_2 \rightarrow F_4$.

\subsection{The Fourier expansion of half-integral weight modular forms} With the groups $\widetilde{G}$ defined and split congruence subgroups $(\Gamma,s_\Gamma)$ of $G(\R)$ constructed, it makes sense to ask about examples and properties of quaternionic modular forms of half-integral weight.  The main property we prove is the existence of a robust, semi-classical Fourier expansion.  To make sense of Fourier expansions on the covering groups $\widetilde{G}$, one begins with the observation that the unipotent group $N(\R)$ splits uniquely into $\widetilde{G}$, with corresponding subgroup $N(\Z)$, so that one can ask about the Fourier expansion of $\varphi_Z(g)$ if $\varphi(g)$ is an automorphic function on $\widetilde{G}$.

To produce the desired Fourier expansion, we analyze \emph{generalized Whittaker functions} on the groups $\widetilde{G}$.  If $\chi: N(\R) \rightarrow \C^\times$ is a nontrivial unitary character, and $\ell \geq 1$ is an odd integer, a generalized Whittaker function of type $(N,\chi,\ell/2)$ is a smooth function $F: \widetilde{G} \rightarrow \mathbf{V}_{\ell/2}$ satisfying
\begin{enumerate}
	\item $F(gk) = k^{-1} \cdot F(g)$ for all $g \in \widetilde{G}$ and $k \in \widetilde{K}_G$;
	\item $F(ng) = \chi(n) F(g)$ for all $n \in N(\R)$ and $g \in \widetilde{G}$;
	\item $D_{\ell/2} F \equiv 0$.
\end{enumerate}
With regard to these generalized Whittaker functions, we prove the following theorem, which is the analogue in the half-integral weight case of the main result of \cite{pollackQDS}.  To state the result, we recall that if $G$ is a quaternionic exceptional group  then there is a notion of ``positive semi-definiteness" of nontrivial unitary characters $\chi$ of $N(\R)$. Recall also that we let $M$ denote a particular fixed Levi subgroup of the Heisenberg parabolic $P$.

\begin{theorem}\label{thm:introGWF} Let the notation be as above, with $\chi$ a non-trivial unitary character of $N(\R)$.
	\begin{enumerate}
		\item Suppose $F$ is a moderate growth generalized Whittaker function of type $(N,\chi, \ell/2)$, and $\chi$ is not positive semi-definite.  Then $F$ is identically $0$.
		\item Suppose $\chi$ is positive semi-definite and $\ell$ is fixed.  There are a pair of nonzero functions $W_{\chi}^1(g)$ and $W_\chi^2(g)$ that satisfy the following properties:
		\begin{enumerate}
			\item $W_\chi^{2}(g) = - W_\chi^1(g)$;
			\item the $W_\chi^j$ are moderate growth generalized Whittaker functions of type $(N,\chi,\ell/2)$;
			\item the set $\{W_\chi^1(g),W_\chi^2(g)\}$ depends continuously on $\chi$;
			\item if $r$ is in the derived group $[M,M](\R)$ and $\widetilde{r}$ is a preimage of $r$ in $\widetilde{G}$, then the set $\{W_\chi^1(\widetilde{r}g),W_\chi^2(\widetilde{r}g)\} = \{W_{\chi \cdot r}^1(g), W_{\chi \cdot r}^2(g)\}$.
			\item Moreover, if $F$ is moderate growth generalized Whittaker function of type $(N,\chi,\ell/2)$, then there is a pair of complex numbers $a_{\chi,2}(F) = - a_{\chi,1}(F)$ so that $F(g) = a_{\chi,j}(F)W_\chi^j(g)$ for $j = 1,2$.
		\end{enumerate}
	\end{enumerate}
\end{theorem}
Note that, if $\zeta$ is the non-identity element of the preimage of $\{1\}$ in $\widetilde{G}$, then $W_\chi^1(\zeta g) = W_\chi^1(g\zeta) = - W_\chi^1(g) = W_\chi^2(g)$, so that one really needs both $W_\chi^1$ and $W_\chi^2$ to appear in property 2(d) of Theorem \ref{thm:introGWF}.

The Fourier expansion of quaternionic modular forms on $\widetilde{G}$ of weight $\ell/2$ follows immediately from Theorem \ref{thm:introGWF}:
\begin{corollary} Suppose $\varphi$ is a quaternionic modular form on $\widetilde{G}$ of weight $\ell/2$.  Then
	\[\varphi_Z(g) = \varphi_N(g) +  \sum_{1 \neq \chi \in (N(\R)/Z(\R)N(\Z))^\vee}{a_\varphi^j(\chi) W_\chi^j(g)}\]
for certain complex numbers $a_\varphi^j(\chi)$ that satisfy $a_\varphi^2(\chi) = -a_\varphi^1(\chi)$.
\end{corollary}
The elements $a_\varphi^j(\chi) \in \C/\{\pm 1\}$ are called the Fourier coefficients of $\varphi$.

\subsection{The automorphic minimal representation}  One of the first examples of quaternionic modular forms of integral weight is given by the automorphic minimal representation on quaternionic $E_8$, which was produced by Gan \cite{ganMin}, see also \cite{pollackE8}.  The double cover of $F_4$ has an automorphic minimal representation; this representation was defined and studied by Loke-Savin \cite{lokeSavin} and further analyzed by Ginzburg \cite{ginzburgF4}.  Our first example of a modular form of half-integral weight, in fact of weight $\frac{1}{2}$, comes from this automorphic minimal representation on $\widetilde{F}_4^{(2)}(\A)$.

The following is our main result concerning the automorphic minimal representation on $\widetilde{F}_4^{(2)}(\A)$.  To state the result, let $J_0 = Sym^2(\Z^3)$ denote the $3 \times 3$ integral symmetric matrices, and let $J_0^\vee$ be the dual lattice with respect to the trace pairing, so that $J_0^\vee$ is the set of half-integral symmetric $3 \times 3$ matrices.  If $N$ denotes the Heisenberg parabolic of $F_4$, then there is an identification $(N(\R)/Z(\R)N(\Z))^\vee$ with the lattice $W(\Z)^\vee = \Z \oplus J_0^\vee \oplus J_0^\vee \oplus \Z$, so that Fourier coefficients of modular forms of $\widetilde{F}_4$ are parameterized by elements of $W(\Z)^\vee$.

\begin{theorem}\label{thm:introMin} Let $\Pi_{min} = \Pi_{min,f} \otimes \Pi_{min,\infty}$ denote the automorphic minimal representation of $\widetilde{F}_4^{(2)}(\A)$.  The minimal $K_{F_4}$-type of $\Pi_{min,\infty}$ is $\Vm_2 = \mathbf{V}_{1/2}$.  Consequently, if $v_f \in \Pi_{min,f}$, there is an associated weight $\frac{1}{2}$ modular form $\theta(v_f)$ on $\widetilde{F}_4$.  Moreover,
	\begin{enumerate}
		\item The $(a,b,c,d) \in W(\Q)^\vee$ Fourier coefficient of $\theta(v_f)$ is zero unless $(a,b,c,d)$ is ``rank one";
		\item the vector $v_f$ can be chosen so that $\theta(v_f)$ has level $\Gamma_{F_4}(4)$ and nonzero $(0,0,0,1) \in W(\Z)^\vee$ Fourier coefficient.
	\end{enumerate}
\end{theorem}

The fact that the minimal $K_{F_4}$-type of $\pi_\infty$ is $\mathbf{V}_{1/2}$ follows easily from work of \cite{ABPTV}.  As explained above, this implies the fact that there are associated weight $\frac{1}{2}$ modular forms $\theta(v_f)$ on $\widetilde{F}_4$.  The statement that the Fourier coefficients of $\theta(v_f)$ vanish unless $(a,b,c,d)$ is rank one is the result \cite[Proposition 3]{ginzburgF4} of Ginzburg, imported into our language.  Where we work hard is the last statement, that $v_f$ can be chosen so that $\theta(v_f)$ has large level and nonzero $(0,0,0,1)$-Fourier coefficient.

To prove this result about level and Fourier coefficients, we make some detailed computations of certain twisted Jacquet modules of the automorphic minimal representation $\pi$, especially at the $2$-adic place.  To do these computations, we bootstrap off of twisted Jacquet module computations in \cite{gelbartPSdistinguished}, which concerns the Weil representation of a double cover of $\GL_2(\Q_p)$.

\subsection{A modular form on $G_2$}\label{Sec: intro arith}  Let $\Theta_{F_4}$ denote a weight $\frac{1}{2}$, level $\Gamma_{F_4}(4)$-modular form on $\widetilde{F}_4$, with nonzero $(0,0,0,1)$-Fourier coefficient, as guaranteed by Theorem \ref{thm:introMin}.  We normalize $\Theta_{F_4}$ so that its $(0,0,0,1)$-Fourier coefficient is $\pm 1$.  There is an embedding $\widetilde{G}_2 \subseteq \widetilde{F}_4$, and denote by $\Theta_{G_2}$ the pullback to $\widetilde{G}_2$ of $\Theta_{F_4}$.  Then we check that $\Theta_{G_2}$ is a quaternionic modular form of weight $\frac{1}{2}$ and level $\Gamma_{G_2}(4)$ for $G_2$.  Our main result concerns the Fourier coefficients of $\Theta_{G_2}$.

To describe these Fourier coefficients, first note that if $N$ is the unipotent radical of the Heisenberg parabolic of $G_2$, then $(N(\R)/Z(\R)N(\Z))^\vee$ can be identified with the integral binary cubic forms $f(u,v) = au^3 + bu^2v + cuv^2 + dv^3$.  We can then give a formula for the Fourier coefficient $a_{\Theta_{G_2}}(f)$ for every integral binary cubic form $f$ with $d=1$.

To state (the main part) of this formula, we introduce a notation concerning cubic rings, following Swaminathan \cite{swaminathan2021average}.  Let $R$ be an order in a totally real cubic field $E = R \otimes \Q$.  Let $\mathfrak{d}_R^{-1}$ be the inverse different of $R$, i.e., the fractional $R$ ideal consisting of those $x \in E$ for which $\tr_{E}(x\lambda) \in \Z$ for all $\lambda \in R$.  Say that a pair $(I,\mu)$ of a fractional $R$ ideal $I$ and a totally positive unit $\mu \in E_{>0}^\times$ is \emph{balanced} if
\begin{enumerate}
	\item $\mu I^2 \subseteq \mathfrak{d}_R^{-1}$
	\item $N(\mu) N(I)^2 \mathrm{disc}(R/\Z) = 1$.
\end{enumerate}
Thus, if $R$ is the maximal order in $E$, $(I,\mu)$ is balanced if and only if $\mu I^2 = \mathfrak{d}_R^{-1}$.  Here $N(\mu)$ is the norm of $\mu$ and $N(I)$ (well-defined up to multiplication by $\pm 1$) is the determinant of a linear transformation of $E$ that takes a $\Z$-basis of $R$ to a $\Z$-basis of $I$.

Let $Q_R$ be the set of balanced pairs $(I,\mu)$ up to equivalence, where we say $(I,\mu)$ is equivalent to $(I',\mu')$ if there exists $\beta \in E^\times$ such that $I' = \beta I$, $\mu' = \beta^{-2} \mu$.  The set $Q_R$ is always finite and sometimes empty.  If $R$ is the maximal order and $Q_R$ is nonempty, then we show that $|Q_R| = |\mathrm{Cl}^{+}_E[2]|$ where $\mathrm{Cl}^{+}_E[2]$ is the $2$-torsion in the narrow class group of $E$.

\begin{theorem}\label{thm:introMain2} Let the notation be as above, and suppose the binary cubic form $f(u,v)$ has $d=1$.  Denote by $R = \Z[y]/(f(1,y))$, and suppose that $R \otimes \Q$ is a totally real cubic field.  The weight $\frac{1}{2}$ modular form $\Theta_{G_2}$ on $G_2$ has Fourier coefficient $a_{\Theta_{G_2}}(f) = \pm 24 |Q_R|$.
\end{theorem}

We also give an arithmetic interpretation of the Fourier coefficients of $\Theta_{G_2}$ in the case that $R \otimes \Q$ is of the form $\Q \times K$ for $K$ a real quadratic field.  See Section \ref{sec:gencase}.

\subsection{Acknowledgements} We thank Benedict Gross for his comments on a previous version of this manuscript, which have improved the exposition of this work. We also thank Gordan Savin for helpful comments.

\chapter{Group theory}
In this chapter, we work out many of the group-theoretic aspects of this paper. We prove Theorems \ref{thm:introCover} and \ref{thm:introSplit} of the introduction.

\section{Central extensions: the general picture and conventions}\label{Sec: general covers}
 Quaterionic modular forms of half-integral weight live on certain central extensions of adjoint forms of exceptional groups. We therefore begin by discussing some generalities about extensions of the group of points of algebraic groups and setting certain conventions. The theory is much more transparent in the simply connected case (which is also our setting when $G=G_2$, $F_4$, or $E_8$), so we recall this setting first. We will only work over $\Q$ and its localizations, so we restrict our discussion to this case. Let $p$ be a place of $\Q$ and let $\Q_p$ be the associated local field; we set $\Q_\infty=\R$.

Assume that $G$ is a simply-connected, simple linear algebraic group over $\Q$ and consider the topological group $G(\Q_p)$ for $p\leq \infty$. In \cite{DeligneCover}, Deligne constructs a canonical extension
\[
1\lra H^2(\Q_p, \mu_n^{\otimes2})\lra \widetilde{G}^{(n)}(\Q_p)\lra G(\Q_p)\lra 1
\]
for any $n\in \mathbb{N}$. This construction relies heavily on the cohomology of the classifying space $BG$ and on the construction of the \emph{Galois symbol} by Tate \cite{TateK2}; we will not review this construction further.

It is known \cite{DeligneCover,MSKtheory} that if $N$ is the number of roots of unity in $\Q_p$, then
 \begin{equation}\label{eqn: what is the kernel}
 H^2(\Q_p, \mu_n^{\otimes2})\cong \mathbb{K}_2(\Q_p)/(n,N)\mathbb{K}_2(\Q_p)\cong \mu_{(n,N)}(\Q_p)
 \end{equation}
 where $\mathbb{K}_2(\Q_p)$ is the Milnor $K$-theory of $\Q_p$. In particular, for any $p\leq \infty$, we obtain a canonical double cover
 \begin{equation}\label{eqn: delignedouble}
 1\lra \mu_2(\Q_p)\lra \G(\Q_p):=\G^{(2)}(\Q_p)\lra G(\Q_p)\lra 1
 \end{equation}
which satisfies the following properties:
 \begin{enumerate}
     \item when $p=\infty$ {and $G(\R)$ is not topologically simply connected}, then $\widetilde{G}$ is the unique connected topological double cover of $G(\R)$ (note that $\pi_1(G(\R))$ is either $\Z$ and $\Z/2\Z$, so this is well defined);
     \item when $G$ is $\Q$-split, then for all $p$ the group $\widetilde{G}(\Q_p)$ agrees with the topological double cover  constructed by  Steinberg and Matsumoto via generators and relations.
 \end{enumerate}
 Both of these facts are relevant to us:  in Section \ref{sec:dcQEG} we give an explicit construction for $\widetilde{G}(\R)$ for quaternionic exceptional groups that is amenable to the definition of generalized Whittaker functions. On the other hand, our main applications to modular forms involve only the split groups $F_4$ and $G_2$. In order to make certain local calculations, we recall the Steinberg--Matsumoto presentation of $\widetilde{G}(\Q_p)$ in Section \ref{Sec: double covers gen}.

If $\A=\A_\Q$ is the adele ring, Deligne similarly constructs a canonical central extension of $G(\A)$ by $\mu_2(\Q)$, so that we have a short exact sequence of locally-compact topological groups
\begin{equation}\label{eqn: global group}
    1\lra \mu_2(\Q)\lra \widetilde{{G}}(\A)\lra {G}(\A)\lra 1.
\end{equation}
This central extension splits canonically over $G(\Q)$, which allows for the definition of automorphic forms on this group. There is a decomposition $\widetilde{{G}}(\A) = \prod_p\widetilde{{G}}(\Q_p)/\mu_2^{+}$, where
$\widetilde{{G}}(\Q_p)$
is the local cover \eqref{eqn: delignedouble} and $\mu_2^{+}$ denotes the subgroup of $\bigoplus_{p}{\mu_2(\Q_p)}$ with product of terms being $1$. When $G$ is a simply-connected, semi-simple group over $\Q$ or $\Q_p$ for $p\leq \infty$ (in particular, when $G$ is of type $G_2$, $F_4$, or $E_8$), we always consider this canonical double cover of Deligne.

 When our reductive group $G$ is no longer semi-simple and simply connected, such as the adjoint forms of $E_6$ and $E_7$ or for Levi subgroups, there is no canonical central extension of $\G(\Q_p)$ by $\mu_2(\Q_2)$; indeed, we will deal with two distinct double covers of $\GL_2(\Q_p)$ in Section \ref{Section: GL2}. The classification of a large class of central extensions (known as Brylinski--Deligne covers) is given in \cite{brylinskiDeligne}, where the authors classify extensions of $G$ by the Milnor $K$-theory sheaf $\mathbb{K}_2$, viewed as sheaves of groups on the big Zariski site over $\Q_p$. Given such a central extension of sheaves of groups over $\Spec(\Q_p)$
 \[
 \mathbb{K}_2\lra \overline{G}\lra G,
 \]
one obtains a topological double cover by taking $\Q_p$-points and pushing out by the Hilbert symbol
 \[
\begin{tikzcd}
\mathbb{K}_2(\Q_p)\ar[d,"{(\cdot,\cdot)_2}"]\ar[r]&\overline{G}(\Q_p)\ar[d]\ar[r]&G(\Q_p)\ar[d,"="]\\
\mu_2(\Q_p)\ar[r]&\G(\Q_p)\ar[r]&G(\Q_p).
\end{tikzcd}
 \]
 Working globally, Brylinski and Deligne also extend the adelic formulation \eqref{eqn: global group} to this more general setting.

 The connection between Deligne's cover and extensions by $\mathbb{K}_2$ may be seen in the identification \eqref{eqn: what is the kernel}. Indeed,  when $G$ is semi-simple and simply connected, it is shown in \cite[Section 4]{brylinskiDeligne} that for any $p$, there exists a central extension of sheaves of groups over $\Spec(\Q_p)$ such that the bottom row of the above diagram recovers the sequence \eqref{eqn: delignedouble}.

 Suppose now that $G$ is an adjoint exceptional group over $\Q$ of type $E_6$ or $E_7$ such that $G(\R)$ is quaternionic (recalled in the next section). In this setting, we construct a double cover $\G$ of $G(\R)$ in Section \ref{sec:dcQEG}. Our convention is that we assume that $\overline{G}$ is a given Brylinski--Deligne cover of $G$ satisfying that the induced double cover of $G(\R)$ agrees with our construction up to isomorphism. This is automatic if the pushout $\widetilde{{G}}(\R)$ is connected and non-linear.

Finally, suppose that $k$ is either a localization of $\Q$ or $k=\A$ and let $\widetilde{G}(k)$ be a given topological double cover of $G(k)$. If $S$ is a subset of $G(k)$, we denote by $\widetilde{S}$ its inverse image in $\G(k)$. If $U\subset G$ is a unipotent subgroup, then it is known that $\G(k)$ splits canonically over $U(k)$; we use a standard abuse of notation and simply denote by $U(k)\subset \widetilde{U}(k)$ the corresponding subgroup of $\G(k)$.

\section{Review of quaternionic exceptional groups}\label{Sec: quatgroups}
In this section, we review notation and constructions from \cite{pollackQDS} concerning quaternionic exceptional groups.  For more details, we refer the reader to \cite[Sections 2,3,4]{pollackQDS}.

First recall the notion of a cubic norm structure $J$.  This is a finite dimensional vector space $J$ over a field $k$ that comes equipped with a homogeneous degree three norm map $N_J: J \rightarrow k$, a non-degenerate trace pairing $(\,,\,): J \otimes J \rightarrow k$, a distinguished element $1_J \in J$, and a quadratic map $\#: J \rightarrow J^\vee \simeq J$.  The relevant examples of cubic norm structures for this paper are $J = k$ and $J = H_3(C)$, the $3 \times 3$ hermitian matrices over a composition $k$-algebra $C$.

Out of a cubic norm structure $J$, one can create various algebraic groups.  First, denote by $M_J$ the identity component of the algebraic group of linear transformations of $J$ that preserve the norm $N_J$ up to scaling.  Let $M_J^1$ denote the subgroup of $M_J$ with scaling factor equal to $1$, and let $A_J$ be the subgroup of $M_J^1$ that fixes the element $1_J$ of $J$.

We next discuss the so-called Freudenthal construction. If $J$ is defined over the field $k$ of characteristic $0$, define $W_J = k \oplus J \oplus J^\vee \oplus k$, another vector space over $k$.  One puts on $W_J$ a certain non-degenerate symplectic form $\langle \,,\,\rangle$ and a quartic form $q: W_J \rightarrow k$.  The algebraic group $H_J$ is defined to be the identity component of the set of pairs $(g,\nu(g)) \in \GL(W_J) \times \GL_1$ that satisfy $\langle gw_1, gw_2 \rangle = \nu(g) \langle w_1, w_2 \rangle$ and $q(g w) = \nu(g)^2 q(w)$.  The map $\nu: H_J \rightarrow \GL_1$ is called the similitude, and $H_J^1$ is defined to be the kernel of $\nu$.

The next algebraic structure defined out of $J$ is a Lie algebra $\g(J)$.  There are two equivalent ways to define $\g(J)$.  In the first way, one defines \[\g(J) = \sl_3 \oplus \m_J^0 \oplus V_3 \otimes J \oplus (V_3 \otimes J)^\vee.\]  Here $\m_J^0$ is the Lie algebra of $M_J^1$ and $V_3$ is the standard three-dimensional representation of $\sl_3$.  A Lie bracket can be put on $\g(J)$; see \cite[section 4.2.1]{pollackQDS}.  We refer to this way of thinking about $\g(J)$ as the ``$\Z/3$-model".  Let $E_{ij}$ be the $3 \times 3$ matrix with a $1$ in the $(i,j)$ position and $0$'s elsewhere.  If $X = \sum_{i,j}{a_{ij} E_{ij}}$ has trace $0$, we will sometimes consider $X$ as an element of $\g(J)$ via the inclusion $\sl_3 \subseteq \g(J)$.

In the second way to define $\g(J)$, one puts \[\g(J) = \sl_2\oplus \h_J^0 \oplus V_2 \otimes W_J.\]   Here $\h_J^0$ is the Lie algebra of $H_J^1$ and $V_2$ is the standard two-dimensional representation of $\sl_2$.  We refer to this way of looking at $\g(J)$ as the $\Z/2$-model.  An explicit isomorphism between the $\Z/3$-model and the $\Z/2$-model is given in \cite[section 4.2.4]{pollackQDS}.  An algebraic group $G_J$ can now be defined as $Aut^0(\g(J))$, the identity component of the automorphisms of the Lie algebra $\g(J)$.

The algebraic groups $A_J, M_J, H_J, G_J$ fit into the Freudenthal magic square, as $J = H_3(C)$ varies with $\dim C = 1,2,4,8$. In table \ref{table:FMS}, we list the absolute Dynkin types of the above groups.  The magic square can be extended to a magic triangle, which was studied in \cite{deligneGross}.  We refer the reader to \cite{deligneGross} for properties of this triangle.
\begin{table}[h]
	\caption{The Freudenthal Magic Square, $J = H_3(C)$}\label{table:FMS}
	\begin{tabular}{|c|c|c|c|c|}
		\hline
		The group & $\dim C = 1$ & $\dim C =2$ & $\dim C = 4$ & $\dim C = 8$ \\ \hline
		$A_J$ & $A_1$ & $A_2$ & $C_3$ & $F_4$ \\ \hline
		$M_J$ & $A_2$& $A_2 \times A_2$ & $A_5$ & $E_6$ \\ \hline
		$H_J$ & $C_3$& $A_5$ & $D_6$& $E_7$ \\ \hline
		$G_J$ & $F_4$ & $E_6$ & $E_7$ & $E_8$ \\ \hline
	\end{tabular}
\end{table}

In the algebraic group $G_J$ we fix a specific parabolic subgroup $P_J$, called the Heisenberg parabolic; see \cite[section 4.3.2]{pollackQDS}.  The subgroup $P_J$ can be defined as the stabilizer of the line $k E_{13} \subseteq \g(J)$.  It has $H_J$ as a Levi subgroup and unipotent radical $N_J \supseteq Z \supseteq 1$ which is two-step.  Here $Z = [N_J,N_J]$ is the exponential of the line $k E_{13}$, and one can identify $N_J/Z$ with $W_J$, as a representation of $H_J$.

Suppose now that $k=\R$ and the trace pairing on $J$ is positive definite.  Then the associated real groups in each row of the magic square share similar properties: the groups $A_J$ are all anistropic, while the groups $M_J$ have real root system of type $A_2$, with root spaces that can be naturally identified with the composition algebra $C$.

In this setting, the groups $H_J$ all have a real root system of type $C_3$, with short root spaces identified with $C$ and long root spaces one-dimensional. Denote by $H_J^+$ the identity component of $H_J(\R)$.  The group $H_J^1$ or $H_J^+$ (which contains $H_J^1$) has a hermitian symmetric domain.  More specifically, let $\mathcal{H}_J = \{Z=X + i Y: X, Y \in J, Y > 0\}$.  Identify $\mathcal{H}_J$ with a subset of $W_J \otimes \C$ via $Z \mapsto r_0(Z) := (1, -Z, Z^\#, - N_J(Z))$.  Then one proves (see \cite[Proposition 2.3.1]{pollackQDS}) that given $g \in H_J^+$ and $Z \in \mathcal{H}_J$, there exists $j(g,Z) \in \C^\times$ so that $g \cdot r_0(Z) = j(g,Z) r_0(gZ)$, for an element $gZ \in \mathcal{H}_J$.  This simultaneously defines an action of $H_J^+$ on $\mathcal{H}_J$ and the factor of automorphy $j(g,Z)$.

Still assuming that $k=\R$ and the trace pairing on $J$ is positive-definite, the group $G_J$ is called a quaternionic group.  The groups $G_J$ in the final row of the Freudenthal magic square now all have real root system of type $F_4$, with short root spaces identified with $C$ and long root spaces one-dimensional.  When $J=\R$ instead of $H_3(C)$, the group $G_J$ is $G_2$.  We refer to these cases by saying that $G_J$ is a quaternionic adjoint exceptional group.  In these cases, the group $G_J(\R)$ is connected \cite{thang}.

Suppose $G_J$ is an adjoint quaternionic exceptional group.  Then a specific Cartan involution on its Lie algebra $\g(J)$ is defined in \cite[section 4.2.3]{pollackQDS}. We denote by $K_J$ the associated maximal compact subgroup of $G_J(\R)$.  The group $K_J$ is of the form $(\SU(2) \times L^0(J))/\mu_2(\R)$, for a certain compact group $L^0(J)$.

In \cite[section 5.1]{pollackQDS}, a specific $\sl_2$-triple $(e_\ell,h_\ell,f_\ell)$ of the complexified Lie algebra of the $\SU(2)$ factor of $K_J$ is defined.  We now recall this $\sl_2$-triple.  Let $e = (1,0)^t$ and $f = (0,1)^t$ denote the standard basis of the two-dimensional representation of $\sl_2 \subseteq \g(J) =  \sl_2\oplus \h_J^0 \oplus V_2 \otimes W_J$.  One sets $e_\ell = \frac{1}{4}(ie +f) \otimes r_0(i \cdot 1_J)$, $f_\ell = -\overline{e_{\ell}}$, and $h_{\ell} = [e_{\ell},f_{\ell}]$.  Here $1_J$ is the identity element of the cubic norm structure $J$.

For $\ell \in \frac{1}{2} \Z_{\geq 0}$, set $\mathbb{V}_2=\C^2$ and $\mathbf{V}_{\ell} = Sym^{2\ell}(\mathbb{V}_2)$, a representation of the Lie algebra of $K_J$ via the projection to the $\SU(2)$ factor.  Using the above $\sl_2$-triple, we fix a basis of $\mathbf{V}_{\ell}$, as follows.  First, let $x,y$ denote a weight basis of $\mathbb{V}_2$ for $h_{\ell}$ with $y = f_{\ell} x$.  Then we let the monomials $x^iy^j$ for $i+j = 2\ell$ be our fixed basis of $\mathbf{V}_{\ell}$.  When $\ell$ is an integer, the representation $\mathbf{V}_{\ell}$ exponentiates to a representation of $K_J$.

\section{The cover in the archimedean case}\label{sec:dcQEG}
In this section, we describe an explicit construction of a connected topological double cover of the quaternionic adjoint groups $G_J(\R)$. This gives the unique non-linear double cover of these groups.

\subsection{Preliminaries}
Now let $J$ be a cubic norm structure over the real numbers $\R$, with positive definite trace pairing.  We assume $J = \R$ or $J = H_3(C)$ with $C$ a composition algebra over $\R$ with positive-definite norm.

Fix the $\sl_2$-triple $e_\ell, h_\ell, f_\ell$ of $\g_J\otimes \C$, recalled above. Identify $Span(e_\ell,h_\ell,f_\ell)$ with $Sym^2(\mathbb{V}_2)$ by sending $e_\ell \mapsto x^2$, $h_\ell \mapsto -2xy$, $f_\ell \mapsto -y^2$.  This identification is $K_J$-equivariant; see right before Lemma 9.0.2 in \cite{pollackQDS}.

We recall an Iwasawa decomposition for the group $G_J(\R)$.  Let $P_J=H_JN_J$ be the Heisenberg parabolic of $G_J$. Let $Q_J$ be the parabolic subgroup associated to the cocharacter $t \mapsto \diag(t,t,t^{-2}) \in \SL_3 \rightarrow G_J$.  The Lie algebra of $Q_J$ contains the root spaces where $E_{11}+E_{22}-2E_{33}$ acts by the weights $0,1,2$ or $3$.  Moreover, $Q_J$ stabilizes $Span(E_{13},E_{23})$ in the $\Z/3$-model of $\g_J$, as one sees by checking this on the Lie algebra level. Define $R_J = P_J \cap Q_J$ and denote by $R_J^+$ the connected component of the identity of $R_J(\R)$.  Recall that $K_J$ denotes the maximal compact subgroup of $G_J(\R)$ associated to the Cartan involution described in \cite{pollackQDS}.

\begin{proposition}\label{prop:iwasawa} Every $g \in G_J(\R)$ can be written as $g = rk$ with $r \in R_J^+$ and $k \in K_J$.  Moreover, if $k \in R_J^+ \cap K_J$, then $k$ acts trivially on $\Span(e_\ell, h_\ell, f_\ell)$.
\end{proposition}
\begin{proof} The first part follows from the usual Iwasawa decomposition of $G_J$.
	
	For the second part, let $M(R_J)$ denote the standard Levi subgroup of $R_J$, so that $M(R_J)$ is the subgroup of $H_J$ that is the centralizer of the cocharacter defined above.  Then $R_J(\R) \cap K_J = M(R_J)(\R) \cap K_J$.  Thus $R_J(\R) \cap K_J$ stabilizes the lines $\R E_{13}$ and $\R E_{23}$ in the Lie algebra $\g(J)$.  We claim that $R_J^+ \cap K_J$ acts trivially on these lines.  To see this, observe that $R_J^+ \cap K_J = M(R_J)^+ \cap K_J$ is connected as it is a maximal compact subgroup of a real connected reductive group.  The triviality of the action of $R_J^+ \cap K_J$ on $E_{13}$ and $E_{23}$ follows.
	
	Recall that $H_J^1$ denotes the similitude equal one subgroup of the Freudenthal group $H_J$.  One has $H_J^1(\R) \cap K_J$ acts by the scalar $j(k,i\cdot 1_J)$ on $e_{\ell}$; see Lemma 9.0.1 of \cite{pollackQDS}.  Because $R_J^+ \cap K_J \subseteq H_J^1(\R) \cap K_J$, $R_J^+ \cap K_J$ acts by a scalar on $e_\ell$.  Because $R_J^+ \cap K_J$ acts trivially on $E_{23}$, this scalar is $1$.  We deduce that $R_J^+ \cap K_J$ acts trivially on $e_{\ell}$, from which it follows that it also acts trivially on $f_{\ell}$ and $h_{\ell}$.
	
	Note that for the second part, one cannot replace $R_J^+$ with $R_J(\R)$ as some elements of $R_J(\R) \cap K_J$ act nontrivially on $\Span(e_\ell, h_\ell, f_\ell)$.\end{proof}

\subsection{The double cover} For $k \in K_J$, denote by $Ad(k)$ the action of $k$ on the space $\Span(e_\ell, h_\ell, f_\ell) = Sym^2(\mathbb{V}_2)$.  Fix an $\R^\times_{>0}$-valued character $\chi$ of $R_J^+$, to be specified later.  We define
\[f_{lin}: G_J(\R) \rightarrow Aut_\C(Sym^2(\mathbb{V}_2)) \simeq \GL_3(\C)\]
as $f_{lin}(g) = \chi(r)Ad(k)$ if $g = rk$ with $r \in R_J^+$ and $k \in K_J$.  By Proposition \ref{prop:iwasawa}, $f_{lin}$ is well-defined, because $\chi(R_J^+ \cap K_J) = 1$ as the image is a compact subgroup of $\R^\times_{>0}$.

Now, consider the symmetric space $X_J = G_J(\R)/K_J$; it is connected and contractible.  Define $j_{lin}(g,x)$ for $x \in X_J$ and $g \in G_J(\R)$ as $f_{lin}(gh) f_{lin}(h)^{-1}$ if $x = h K_J$.  Note that $j_{lin}$ is well-defined.

One has the following proposition, whose proof we omit; it follows from the fact that the Iwasawa decomposition of $G_J(\R)$ is smooth.
\begin{proposition} The maps
	\[
	f_{lin}: G_J(\R) \longrightarrow Aut_\C(Sym^2(\mathbb{V}_2))\:\text{  and  }\: j_{lin}: G_J(\R) \times X_J \longrightarrow Aut_\C(Sym^2(\mathbb{V}_2))
	\]
	are smooth.
\end{proposition}

We may now define $\G_J$.
\begin{definition}\label{Def: double cover}
	Let $\G_J$ be the set of pairs $(g,j_g)$ with $g \in G_J(\R)$ and
	\[
	j_g: X_J \rightarrow Aut_\C(\mathbb{V}_2)
	\]
	continuous so that $Sym^2(j_g(x)) = j_{lin}(g,x)$.  A multiplication is defined as
	\[(g_1,j_1(x))(g_2,j_2(x)) = (g_1g_2,j_1(g_2x)j_2(x)).\]
The identity is the element $(1,e)$ where $e(x) = 1$ for all $x$.
\end{definition}  With these definitions, it is easily checked that $\G_J$ is a group.

A topology can be put on $\G_J$ as follows.  Let $x_0 = 1 K_J \in X_J$ be the basepoint determined by $K_J$.  Now, note that given $g \in G_J(\R)$, there are exactly two continuous lifts $X_J \rightarrow Aut_\C(\mathbb{V}_2)$ of $j_{lin}(g,-): X_J \rightarrow Aut_\C(Sym^2(\mathbb{V}_2))$, and that these lifts are determined by their value at $x_0$.  Thus there is an injective map of sets $\G_J\rightarrow G_J(\R) \times \GL_2(\C)$ given by $(g,j_g(x)) \mapsto (g, j_g(x_0))$.  We give $\G_J$ the subspace topology of $G_J(\R) \times \GL_2(\C)$ via this map.

For $g'=(g,j_g(x)) \in \G_J$, we write $j_{1/2}(g',x) := j_g(x)$.

\begin{proposition} With the above topology, $\G_J$ is a connected topological group.  The canonical map $\G_J \rightarrow G_J(\R)$ is a covering map with central $\mu_2$ kernel.\end{proposition}

\begin{proof} One first proves that $\G_J$ is a topological group and $\G_J \rightarrow G_J(\R)$ is a covering space.  This is an exercise in covering space theory, so we omit it.
	
	Let us explain the connectedness of $\G_J$.  We will check that $(1,e(x))$ and $(1,-e(x))$ are connected by a path.  Given the other claims, this suffices.
	
	To see that $(1,e(x))$ is connected to $(1,-e(x))$, we consider $h_0 = \mm{0}{1}{-1}{0} \in \sl_2 \subseteq \g_J = \sl_2 \oplus \h_J^0 \oplus \mathbb{V}_2 \otimes W_J$.  Now, by our formulas for the Cartan decomposition, $h_0$ is in the Lie algebra of $K_J$, so $\exp(t h_0)$ is in $K_J \subseteq G_J(\R)$.   One computes that $\exp(th_0)$ acts on $e_{\ell},h_\ell, f_\ell$ as
	\begin{itemize}
		\item $e_\ell \mapsto e^{-it} e_{\ell}$
		\item $h_\ell \mapsto h_\ell$
		\item $f_\ell \mapsto e^{it} f_{\ell}$.
	\end{itemize}
	
	Now consider the path $[0,2\pi] \rightarrow K_J \subseteq G_J(\R)$ given by $t \mapsto \exp(th_0)$.  This path is a loop, with $2\pi \mapsto 1$.  Because $\G_J \rightarrow G_J(\R)$ is a covering space, it lifts to a path $\widetilde{\gamma}:[0,2\pi] \rightarrow \G_J$ satisfying $\widetilde{\gamma}(0) = 1$.  Thus $j_{1/2}(\widetilde{\gamma}(t),x_0) \in \GL_2(\C)$ satisfies that its symmetric square is the action on $e_\ell,h_\ell,f_\ell$ given above. Because it is continuous and the identity at $t=0$, $j_{1/2}(\widetilde{\gamma}(t),x_0) = \diag(e^{-it/2},e^{it/2})$.  Consequently $j_{1/2}(\widetilde{\gamma}(2\pi),x_0) = -1$.  This proves our assertion.
\end{proof}

Note that since $K_J$ is itself connected and the path $\widetilde{\gamma}$ stays in $\widetilde{K}_J$, we see that the inverse image $\widetilde{K}_J$ of $K_J$ is a connected compact Lie group.

Because $\G_J \rightarrow G_J(\R)$ is a covering space, $\G_J$ is uniquely a Lie group.  Note also that the map $j_{1/2}(\,,x_0): \widetilde{K}_J \rightarrow Aut_\C(\mathbb{V}_2)$ is a group homomorphism.  Finally, we remark that $R_J^+$ splits into $\G_J$ as $r \mapsto (r,j_r(x))$ with $j_r(x) = \chi(r)^{1/2}$ for all $x \in X_J$.

\subsection{An application}
Define $\nu: R_J \rightarrow \GL_1$ as $r E_{13} = \nu(r) E_{13}$ and $\lambda: R_J \rightarrow \GL_1$ as $r E_{23} = \lambda(r) E_{23} + * E_{13}$.  In other words, if $\det$ is the determinant of the action of $R_J$ on $\Span(E_{13},E_{23})$ then $\lambda = \det(\cdot) \nu^{-1}$.  Define $\chi$, the character defining $f_{lin}$ as $\chi=\nu \lambda^{-1} = \nu^2 \det(\cdot)^{-1}$.  With this choice, which we will make from now on, one has the following lemma.  Let $K_H = H_J^1(\R) \cap K_J$ be a maximal compact subgroup of $H_J^1(\R)$.
\begin{lemma}\label{lem:j12} With $h \in H_J^+$, one has $j_{lin}(h,x_0) = \diag(j(h,i),1,\overline{j(h,i)})$ via the action on $x^2,xy,y^2$.  Thus if $z \in \mathcal{H}_J = H_J^1(\R)/K_H \subseteq G_J/K_J$, then $j_{lin}(h,z) = \diag(j(h,z),1,\overline{j(h,z)}).$  Consequently, the $(1,1)$-coordinate of $j_{1/2}: \widetilde{H_J^+} \rightarrow \GL_2(\C)$ defines a squareroot of $j(h,z)$.\end{lemma}
\begin{proof} Let $P_{S}$ denote the Siegel parabolic of $H_J$, so that $P_S = H_J \cap R_J$.  Suppose $h \in H_J^+$ is $h = pk$ with $p \in P_{S}(\R)^+$ and $k \in K_H \subseteq H_J^1(\R)$.  Then
	\[j(p,i) = \langle p r_0(i), E_{23}\rangle = \nu(p) \langle r_0(i), p^{-1} E_{23} \rangle = \chi(p).\]
 Moreover, essentially by definition of $j$, $Ad(k) = \diag(j(k,i),1,\overline{j(k,i)})$.  As $j(h,i) = j(pk,i) = j(p,i)j(k,i)$, one obtains $j_{lin}(h,x_0) = \diag(j(h,i),1,\overline{j(h,i)})$.
	
For the second statement, suppose $h_z \in H_J^+$ satisfies $h_z \cdot i = z$.  Then
	\begin{align*}j_{lin}(h,z) &= f_{lin}(hh_z) f_{lin}(h_z)^{-1} = \diag(j(hh_z,i),1, \overline{j(hh_z,i)}) \diag(j(h_z,i),1,\overline{j(h_z,i)})^{-1} \\ &= \diag(j(h,z),1,\overline{j(h,z)}).\end{align*}
	The proposition follows.
\end{proof}

\section{Steinberg generators and relations}\label{Sec: double covers gen}
In this section, we let $k$ be a local field of characteristic zero and assume that $G$ is a simply connected simple group over $k$. In this setting, Deligne's double cover \eqref{eqn: delignedouble}coincides with the Steinberg--Matsumoto cover. We thus recall this construction for the purposes of certain $p$-adic calculations in later sections. 

Suppose that $\Phi$ is a simple root system and $\De$ a set of simple roots. We let $(\al,\be)$ denote the pairing on $\Phi$ normalized so that $(\alpha,\alpha) = 2$ for a long root (when the root system is simply laced, we assert that all roots are long).  Suppose that $\g$ is the associated split, simple Lie algebra over $\Q$ and $G$ the associated split, simply-connected group. Steinberg \cite{steinbergBook} gives a presentation for the group $G(k)$ in terms of generators and relations.  One has generators $x_\alpha(u)$ for all roots $\alpha$ and $u \in k$, subject to the following relations:
\begin{enumerate}
	\item\label{1} $x_\alpha(u)x_\alpha(v) = x_\alpha(u+v)$.
	\item\label{2} If $\alpha,\beta$ are roots with $\alpha + \beta \neq 0$, then the commutator
	\[\{x_\alpha(u),x_\beta(v)\} = \prod_{i\alpha + j \beta \in \Phi, i,j \in \Z_{> 0}}{x_{i\alpha + j\beta}(C_{ij}u^i v^j)}\]
	for integers $C_{i,j}$ that depend upon the order in the product but are independent of $u,v$.
	\item\label{3} For $t \in k^\times$ set $w_{\alpha}(t) = x_\alpha(t) x_{-\alpha}(-t^{-1})x_\alpha(t)$ and $h_\alpha(t) = w_{\alpha}(t) w_{\alpha}(-1)$.  Then $h_\alpha(t)h_\alpha(s) = h_\alpha(ts).$
	\item \label{4} When $\Phi$ is of type $A_1$, then $w_\alpha(t) x_\alpha(u) w_\alpha(-t) = x_{-\alpha}(-t^{-2} u)$.
\end{enumerate}

 Following Steinberg \cite[Theorem 12]{steinbergBook} (see also \cite[Section 2]{lokeSavin}), a topological double cover of $G(k)$ can now be defined as follows. Recall the Hilbert symbol $(\cdot,\cdot)_2: k^\times \times k^\times \rightarrow \mu_2(k)$. One takes as generators elements $x_{\alpha}(u)$ and $\{1,\zeta\} = \mu_2$ satisfying \eqref{1}, \eqref{2}, and \eqref{4}, along with
\begin{enumerate}
	\setcounter{enumi}{4}
	\item The elements $1,\zeta$ are in the center.
	\item For $t \in k^\times$ set
	\[
	\widetilde{w}_{\alpha}(t) = x_\alpha(t) x_{-\alpha}(-t^{-1})x_\alpha(t)\quad\text{ and }\quad \widetilde{h}_\alpha(t) = \widetilde{w}_{\alpha}(t) \widetilde{w}_{\alpha}(-1).
	\]
	Then $\widetilde{h}_\alpha(t)\widetilde{h}_\alpha(s) = \widetilde{h}_\alpha(ts) (t,s)_2^{\frac{2}{(\alpha,\alpha)}}$.
\end{enumerate}
From \cite[section 3, page 4904]{lokeSavin}, who cite \cite[Lemme 5.4]{matsumoto}, one has \begin{equation}\label{eqn: commutator of tori}
\{\widetilde{h}_{\alpha}(s),\widetilde{h}_{\beta}(t)\} = (s,t)_2^{(\alpha^\vee,\beta^\vee)},
\end{equation}
where $\al^\vee = \frac{2\al}{(\al,\al)}$. We let $\G(k)$ denote the double cover of $G(k)$ here constructed, where the projection $p:\G(k)\lra G(k)$ is given by sending generators to the analogous generators in $G(k)$. As previously noted, this construction recovers Deligne's cover \eqref{eqn: delignedouble} in the split case. In particular, if $J = \R$ or $H_3(\R)$, so that $G=G_J$ is the split group of type $G_2$ or $F_4$ respectively, then $\widetilde{G}(\R)\cong \G_J$.

\section{$2$-adic subgroups of $F_4$} We now specialize to $k=\Q_2$ and $G$ the split group of type $F_4$. We enumerate the $4$ simple roots in the usual way, so that the Dynkin diagram
\[\circ---\circ=>=\circ---\circ\]
has labels $\alpha_1$ through $\alpha_4$ from left to right. In this section and the next section, we define certain compact open subgroups $K_R^*(4)$ and $K_R'(4)$ of $\widetilde{F}_4(\Q_2)$ that we prove inject into $F_4(\Q_2)$.  In Section \ref{sec:splittings}, we use the compact open subgroup $K_R'(4)$ to construct congruence subgroups $\Gamma_{F_4}(4)$ and $\Gamma_{G_2}(4)$ that split into the real groups $\widetilde{F}_4$ and $\widetilde{G}_2$, respectively.

\begin{remark} Strictly speaking, we do not use the subgroup $K_R^*(4)$ in the sequel; we use only the subgroup $K_R'(4)$.  However, our analysis of the group $K_R'(4)$ is similar to, but a bit more complicated than, our analysis of the group $K_R^*(4)$.  Thus, we handle the case of $K_R^*(4)$ first, to make clear the ideas.

In Section \ref{Sec: modular forms}, we use the group $K'_R(4)$ to construct the quaternionic modular forms of half-integral weight described in Theorem \ref{thm:introMain2}. We remark that one can also construct quaternionic modular forms of level $K_R^*(4)$. However, it is unclear if their Fourier coefficients are as interesting.  \end{remark}

\subsection{Preliminaries}
To begin, we record the following slight extension of \cite[Lemma 3.1]{karasiewicz}.
\begin{lemma}\label{Lem: commutator roots}
	Let $k$ be a local field of characteristic zero. Suppose that $\Phi$ is a simple root system and $G(k)$ is the corresponding simply-connected group. For any $\al\in \Phi$ and $s,t\in k$ such that $1+st\neq0$, in the double cover $\G(k)$ we have
	\[
	x_\al(t)x_{-\al}(s) = (1+st,\frac{t}{1+st})_2^{-\frac{2}{(\al,\al)}}x_{-\al}\left(\frac{s}{1+st}\right) \widetilde{h}_\al(1+st)x_\al\left(\frac{t}{1+st}\right).
	\]
	
\end{lemma}
\begin{proof}
	This follows from \cite[Proposition 2.7]{SteinStable}.
\end{proof}
\begin{corollary}\label{Cor: unipotent nice} With notation as above, now let $k=\Q_2$ and let $\al\in \Phi$ and $s,t\in \Q_2$.
	\begin{enumerate}
		\item\label{case short} If $\Phi$ is doubly laced and $\al$ is a short root, then
		\[
		x_\al(t)x_{-\al}(s) = x_{-\al}\left(\frac{s}{1+st}\right) \widetilde{h}_\al(1+st)x_\al\left(\frac{t}{1+st}\right);
		\]
		\item\label{case long} Let $\Phi$ be of any type. If $\val_2(s)\geq 2$ and $\val_2(t)\geq 0$, then
		\[
		x_\al(t)x_{-\al}(s) = x_{-\al}\left(\frac{s}{1+st}\right) \widetilde{h}_\al(1+st)x_\al\left(\frac{t}{1+st}\right);
		\]
	\end{enumerate}
\end{corollary}
\begin{proof}
	The proof of the first claim is immediate from the lemma and our normalization that $(\be,\be)=2$ for long roots, so that $(\al,\al)=1$ for our short root. The second claim follows precisely as in the proof of  \cite[Lemma 3.1]{karasiewicz} with $\lam=0$.
\end{proof}.

We now return to $G=F_4$. The inclusion of rational Lie algebras $\m_J^0 \rightarrow \g(J)$ discussed in Section \ref{Sec: quatgroups} gives rise to an embedding of algebraic groups $\SL_3 \rightarrow F_4$ when $J = H_3(\Q)$. In terms of roots, the image corresponds to the subroot system with simple roots $\{\al_3,\al_4\}$. When $k$ is a local field, note that this embedding lifts to a splitting $s: \SL_3(k) \rightarrow \widetilde{F}_4(k)$. Indeed, the subgroup $\SL_3(k)$ of $F_4(k)$ is generated by the elements $x_\beta(u)$ for $\beta$ lying in the sub-root system generated by $\{\al_3,\al_4\}$.  We may define this $\SL_3(k)$ via generators and relations as in Section \ref{Sec: double covers gen}, and the relations defining it continue to be satisfied in $\widetilde{F}_4(k)$ due to Corollary \ref{Cor: unipotent nice}.
\begin{lemma}\label{Lem: SL split}
  Let $\SL_3\subset F_4$ be the $\Q$-subgroup just described. For any local field $k$, the double cover $\widetilde{F}_4(k)$ splits uniquely over $\SL_3(k)$.
\end{lemma}
\subsection{The case of $K^\ast_R(4)$}  We define a compact open subgroup $K_R^*(4)$ of $\widetilde{F}_4(\Q_2)$, and prove that the projection map $K_R^*(4) \rightarrow F_4(\Q_2)$ is injective. Our construction can be considered a slight generalization of \cite[Theorem 3.3]{karasiewicz}, which only considers simply laced groups.

Recall that $\alpha_1,\alpha_2,\alpha_3,\alpha_4$ are the simple roots of $F_4$, with $\alpha_1,\alpha_2$ long and $\alpha_3,\alpha_4$ short.  Let $R = M_R U_R$ be the standard non-maximal parabolic subgroup of $F_4$ with simple roots $\alpha_3, \alpha_4$ in the Levi $M_R$. {The notation $R$ here refers to the non-maximal parabolic $R_J$ from Section \ref{sec:dcQEG} as these two parabolic subgroups agree when $G=G_J$ is of type $F_4$.} Set $$\Phi_{M_R}^+ = \{\alpha_3,\alpha_4,\alpha_3+\alpha_4\},$$ set $\Phi_{M_R}^{-} = -\Phi_{M_R}^{+}$ and $\Phi_{M_R} = \Phi_{M_R}^{+} \cup \Phi_{M_R}^{-}$.  Let $\Phi_{U_R}^+ = \Phi^{+} \setminus \Phi_{M_R}^{+}$, so that $\Phi_{U_R}^{+}$ contains the roots in the unipotent radical $U_R$ of $R$.

Set $K_{M_R}^*(4)$ to be the subgroup $\widetilde{M}_R(\Q_2)$ generated by $h_{\alpha_i}(1+4\Z_2)$ for $i=1,2$ and $x_{\beta}(\Z_2)$ for $\beta \in \Phi_{M_R}$.  Let $U_R^{+}(\Z_2)$ be the subgroup of $\widetilde{F}_4(\Q_2)$ generated by $x_{\beta}(\Z_2)$ for all $\beta \in \Phi_{U_R}^{+}$, and let $U_R^{-}(4\Z_2)$ be the subgroup of $\widetilde{F}_4(\Q_2)$ generated by $x_{-\beta}(4\Z_2)$ for all $\beta \in \Phi_{U_R}^{+}$.  Finally, let $K_R^*(4)$ be the subgroup of $\widetilde{F}_4(\Q_2)$ generated by $U_{R}^{-}(4\Z_2)$, $K_{M_R}^*(4)$ and $U_R^{+}(\Z_2)$.  We have the following theorem.

\begin{theorem}\label{thm:iwahoriSplit} Let the notation be as above.
	\begin{enumerate}
		\item One has $K_R^*(4) = U_{R}^{-}(4\Z_2) K_{M_R}^*(4) U_{R}^{+}(\Z_2)$.
		\item The map $K_R^*(4) \rightarrow F_4(\Q_2)$ is injective.
	\end{enumerate}
\end{theorem}
We will prove this theorem below.  It is easy to deduce the following corollary of Theorem \ref{thm:iwahoriSplit}.
\begin{corollary}\label{cor:iwahoriBig} The group $K_R^*(4)$ has an Iwahori decomposition with respect to any standard parabolic subgroup containing $R$.\end{corollary}

Recall the subgroup $\SL_3\subset F_4$ from the previous subsection.  The subgroup $s(\SL_3(k))$ of $\widetilde{F}_4(k)$ is that which is generated by the elements $x_{\beta}(u)$ for $\beta \in \Phi_{M_R}$.
Using Lemma \ref{Lem: SL split}, we now observe:
\begin{lemma}\label{lem:inj1} The map $K_{M_R}^*(4) \rightarrow F_4(\Q_2)$ is injective.
\end{lemma}
\begin{proof} If $g \in K_{M_R}^*(4)$, it is easy to see that one can express $g$ as a product $g = t_1 t_2 s(g')$ with $t_j \in h_{\alpha_j}(1+4 \Z_2)$ and $g' \in \SL_3(\Q_2)$.  Consequently, if $g \mapsto 1$ in $F_4(\Q_2)$, then $t_1 = t_2 = 1$ and $g' =1$, proving that $g=1$.
\end{proof}

We will prove part (1) of Theorem \ref{thm:iwahoriSplit} in Paragraph \ref{Sec: iwahori}.  Let us observe now that part (1) implies part (2).  Indeed, suppose $g = n_1 m n_2$ is in $K_R^*(4)$ with $n_1 \in U_R^{-}(4\Z_2)$, $m \in K_{M_R}^*(4)$ and $n_2 \in U_{R}^{+}(\Z_2)$.  If $g \mapsto 1$ in $F_4(\Q_2)$, then we see easily that $n_1 = 1$ and $n_2 =1$.  Thus $m \mapsto 1$, hence $m=1$ by Lemma \ref{lem:inj1}.

\subsubsection{Iwahori decomposition} \label{Sec: iwahori}
For a non-negative integer $m$, let $U_R^{+}(2^m\Z_2)$ be the subgroup of $\widetilde{F}_4(\Q_2)$ generated by $x_{\beta}(2^m\Z_2)$ for all $\beta \in \Phi_{U_R}^{+}$, and let $U_R^{-}(2^m\Z_2)$ be the subgroup of $\widetilde{F}_4(\Q_2)$ generated by $x_{-\beta}(2^m\Z_2)$ for all $\beta \in \Phi_{U_R}^{+}$.

We begin with the following lemma.  Let $U_B$ be the unipotent radical of the standard Borel of $\widetilde{F}_4(\Q_2)$.
\begin{lemma}\label{lem:U+N} Recall that $\Delta = \{\alpha_1,\alpha_2,\alpha_3, \alpha_4\}$ are the simple roots.
	\begin{enumerate}
		\item The unipotent group $U_{B}(\Q_2)$ is generated by the $x_{\alpha_i}(\Q_2)$;
		\item Let $U_s$ be the subgroup of $U_{B}(\Q_2)$ generated by the $x_{\alpha_i}(\Z_2)$.  Then $U_s$ contains $U_{R}^{+}(2^A)$ for some $A\gg 0$.
	\end{enumerate}
\end{lemma}
\begin{proof}
	The first part of the lemma is standard.  For the second part, suppose $\alpha \in \Phi_{U_R}^{+}$.  By the first part, there exists a finite word $u$ in elements of the form $x_{\alpha_i}(r_i)$ with $r_i\in\Q_2$, so that $u=x_{\alpha}(1)$.  Let $T^{++}$ denote the subgroup of $t \in T$ with $|\alpha_i(t)| < 1 $ for all $i$.  Conjugating by a sufficiently deep $t \in T^{++}$, one finds that there exists a nonzero $r_{\alpha} \in \Z_2$ so that $x_{\alpha}(r_{\alpha}) \in U_s$.  Now, for $t \in \Z_2^\times$, consider the commutator $\{h_{\alpha}(t),x_{\alpha}(r_\alpha)\}$.  On the one hand, because $t \in \Z_2^\times$, this commutator is in $U_s$.  On the other hand, this commutator is $x_{\alpha}((t^2-1)r_{\alpha})$.  As $t$ varies in $\Z_2^\times$, $t^2-1$ fills out $8\Z_2$.  Thus there is $N_{\alpha} \gg0$ so that $x_{\alpha}(2^{N_\alpha}\Z_2) \subseteq U_s$.  The lemma follows.
\end{proof}

Let $U$ be the set of products of the form $U_{R}^{-}(4\Z_2) K_{M_R}^*(4) U_R^{+}(\Z_2)$.  Let $K_R^*(4,2^m)$ be the subgroup of $\widetilde{F}_4(\Q_2)$ generated by $U_R^{-}(4\Z_2)$, $K_{M_R}^*(4)$ and $U_R^{+}(2^m \Z_2)$, so that $K_R^*(4) = K_R^*(4,1)$.  In order to prove Theorem \ref{thm:iwahoriSplit}, we need to check that $K_{R}^*(4)\cdot U = U$.  We will do this by proving $K_{R}^*(4,2^A) \cdot U = U$ for $A \gg 0$, then inducting down on $A$ to obtain $K_R^*(4,1) \cdot U = U$.

We start with the following lemma.
\begin{lemma}\label{lem:easyU} One has $U_{R}^{-}(4\Z_2) \cdot U = U$, and $K_{M_R}^*(4) \cdot U = U$.
\end{lemma}
\begin{proof}  That $U_{R}^{-}(4\Z_2) U = U$ is trivial.  For the multiplication by $K_{M_R}^*(4)$, one uses that if $\beta \in \Phi_{U_R}^{-}$, $\alpha \in \Phi_{M_R}$, and $a,b$ are positive integers, then if $\gamma =a \alpha + b \beta$ is a root, then $\gamma \in \Phi_{U_R}^{-}$.  The lemma then follows easily by applying the commutator formula.
\end{proof}

Now we have:
\begin{proposition}\label{prop:U+N} There is $A\gg 0$ such that $U_{R}^{+}(2^A) \cdot U \subseteq U$.
\end{proposition}
\begin{proof} By Lemma \ref{lem:U+N}, it suffices to show that $x_{\alpha_i}(\Z_2) U \subseteq U$ for all simple roots $\alpha_i$. By Lemma \ref{lem:easyU}, we must only check this for $i = 1,2$.
	
	Thus suppose that $\alpha_i$ is a simple root, $i=1,2$, and $\alpha \in \Phi_{U_R}^{+}$.  Note that if $a,b$ are positive integers, and $\alpha \neq \alpha_i$, then if $\gamma = a\alpha_i - b \alpha$ is a root, then $\gamma \in \Phi^{-}$.  Indeed, $a \alpha_i = \gamma + b \alpha$, so that if $\gamma$ were positive, we would have that both $\gamma$ and $\alpha$ are proportional to $\alpha_i$, a contradiction.  It follows that, for such $\alpha_i$ and $\alpha$ and $u \in \Z_2$, $u' \in 4\Z_2$, the commutator $\{x_{\alpha_i}(u),x_{-\alpha}(u')\} \in U_{B}^{-}(4\Z_2)$.  Here $U_B^{-}(4\Z_2)$ is the subgroup of $\widetilde{F}_4(\Q_2)$ generated by $x_{\beta}(4\Z_2)$ for $\beta$ a negative root.
	
	Let us also note that $x_{\alpha_i}(u) x_{-\alpha_i}(u') = x_{-\alpha_i}(u'/(1+uu')) \widetilde{h_{\alpha_i}}(1+uu') x_{\alpha_i}(u/(1+uu'))$.  Combining these two facts, we obtain the following: If $g = n_1 m n_2$ is in $U$, then $x_{\alpha_i}(u) g = n_1' x_{\alpha_i}(u) m' n_2$ with $n_1' \in U_B^{-}(4\Z_2)$ and $m' \in K_{M_R}^*(4)$.
	
	Now, one verifies easily that if $m' \in K_{M_R}^*(4)$ and $u \in U_{R}^{+}(\Z_2)$, then $(m')^{-1} u m' \in U_{R}^+(\Z_2)$.  Consequently, $x_{\alpha_i}(u) g = n_1' m' n_2'$ is in $U_{B}^{-}(4\Z_2) \cdot U$.  The proposition follows from Lemma \ref{lem:easyU}.
\end{proof}

It follows from Proposition \ref{prop:U+N} and Lemma \ref{lem:easyU} that $K_{R}^*(4,2^A) \cdot U \subseteq U$ for $A \gg 0$.  As mentioned, we will now induct downward on $A$ to obtain $K_{R}^*(4) \cdot U = U$.

We require the following lemma.
\begin{lemma} Let the notation be as usual.
	\begin{enumerate}
		\item The sets $\widetilde{h_{\alpha_i}}(1+4\Z_2)$ are subgroups, and they commute with each other.
		\item Suppose $t \in 1+4\Z_2$ and $\beta \in \Phi$.  Then there are $t_1, \ldots, t_4 \in 1+4\Z_2$ so that $\widetilde{h_{\beta}}(t) = \prod_{i=1}^{4} \widetilde{h_{\alpha_i}}(t_i)$.
	\end{enumerate}	
\end{lemma}
\begin{proof} The first part of the lemma follows from the usual multiplication formulas, together with the fact that the Hilbert symbol is trivial when restricted to $1+4\Z_2$.  For the second part of the lemma, we mimic the proof of \cite[Lemma 38 (b)]{steinbergBook}.  Thus suppose $\beta = w \alpha_i$ with $\alpha_i$ a simple root.  Write $w = w_{\alpha}w'$ where $length(w') =length(w)-1$.  Set $\gamma = w_{\alpha} \beta$ so that $\beta = w_{\alpha} \gamma$.  Now \cite[Lemma 37 (c)]{steinbergBook} yields that $\widetilde{w_{\alpha}}(1) \widetilde{h_{\gamma}}(t) \widetilde{w_{\alpha}}(-1) = \widetilde{h_{w_{\alpha}\gamma}}(t) (c,t)$ for some $c = \pm 1$.  However, because $t \in 1+4\Z_2$ and $c = \pm 1$, $(c,t) = 1$.  Thus $\widetilde{h_{\beta}}(t)= \widetilde{w_{\alpha}}(1) \widetilde{h_{\gamma}}(t) \widetilde{w_{\alpha}}(-1)$, from which we obtain
	\[\widetilde{h_{\beta}}(t) = \widetilde{h_{\gamma}}(t) (\widetilde{h_{\gamma}}(t)^{-1} \widetilde{w_{\alpha}}(1) \widetilde{h_{\gamma}}(t)) \widetilde{w_{\alpha}}(-1) = \widetilde{h_{\gamma}}(t) \widetilde{w_{\alpha}}(t^{-\langle \alpha,\gamma \rangle}) \widetilde{w_{\alpha}}(-1),\]
	using \cite[Lemma 37 (e)]{steinbergBook} for the second equality.  But this is $\widetilde{h_{\gamma}}(t) \widetilde{h_{\alpha}}(t^{-\langle \alpha,\gamma \rangle})$.  The lemma follows by induction on the length of $w$.
\end{proof}

\begin{proposition}\label{prop:inductDown} For every non-negative integer $m$, one has $K^*(4,2^m) \cdot U \subseteq U$.
\end{proposition}
\begin{proof} As just noted, Proposition \ref{prop:U+N} implies $K_R^*(4,2^A) \cdot U \subseteq U$ for $A\gg0$.  We will induct downward on $N$ to obtain the proposition.
	
	Thus suppose that we have proved $K_{R}^*(4,2^{m+1}) \cdot U \subseteq U$ for a non-negative integer $m$.  We wish to verify that $K_{R}^*(4,2^m) \cdot U \subseteq U$.  To do this, it suffices to check that $U_{R}^{+}(2^m\Z_2) \cdot U \subseteq U$.  Thus suppose $u = x_{\alpha}(2^m s) \in U_{R}^{+}(2^m\Z_2)$ and $x = n_1 m n_2 \in U$.  We have $u x = (un_1u^{-1}) m (m^{-1}um) n_2$.  It is easy to see that $(m^{-1}um) n_2 \in U_{R}^+(\Z_2)$.  We claim that $un_1 u^{-1} \in K_{R}^*(4,2^{m+1})$.  Granted this claim, the proposition follows.
	
	To prove the claim, suppose $n_1 = v_1 \cdots v_{r}$ with each $v_i$ of the form $x_{-\beta_i}(4 s_i)$ with $s_i \in \Z_2$ and $\beta_i \in \Phi_{U_R}^{+}$.  The commutator formula gives $uv_j u^{-1} = k' $ with $k' \in K_R^*(4,2^{m+1})$. Indeed, if $\alpha \neq \beta_i$ this follows from the commutator formula.  If $\alpha = \beta_i$ this follows from the formula
	\[x_\alpha(t)x_{-\alpha}(s) = x_{-\alpha}(s/(1+ts)) \widetilde{h_{\alpha}}(1+st) x_{\alpha}(t/(1+st))\]
	which implies
	\[x_\alpha(t)x_{-\alpha}(s) x_{\alpha}(-t) = x_{-\alpha}(s/(1+ts)) \widetilde{h_{\alpha}}(1+st) x_{\alpha}(-st^2/(1+st)).\qedhere\]
\end{proof}

\subsection{The case of $K'_R(4)$} We now define a new subgroup, $K_R'(4) \subseteq \widetilde{F}_4(\Q_2)$, and prove that it has appropriate Iwahori factorizations and maps injectively to the linear group $F_4(\Q_2)$. This gives a useful compact open subgroup of $\widetilde{F}_4(\Q_2)$ for global constructions.

We need to introduce a bit more notation. Recall that $P_S = H_J \cap R$ is the Siegel parabolic subgroup of $H_J$; it has Levi decomposition $P_S= M_R N_S$. We set $Q=LU_Q$ denote the standard maximal parabolic of $F_4$ associated to the simple root $\al_2.$ Recalling the notation in Section \ref{sec:dcQEG}, this is the parabolic $Q_J$ of $G_J = F_4$ when $J = H_3(\Q_p)$.

Let $\Phi_{N}^{+}$ be the set of roots in the unipotent radical $N$ of the Heisenberg parabolic $P$.  Let $\Phi_{N_S}^{+}$ be the set of roots in the unipotent group $N_S$.  These are the roots $\sum_{i}{m_i \alpha_i}$ with $m_1 = 0$ and $m_2 =1$.  Note that $\Phi_{U_R}^{+} = \Phi_{N}^{+} \sqcup \Phi_{N_S}^{+} = \{\alpha_1\} \sqcup \Phi_{U_Q}^+$.

We let $N_S^{+}(2^m \Z_2)$ be the subgroup of $\widetilde{F}_4(\Q_2)$ generated by $x_{\alpha}(2^m \Z_2)$ for all $\alpha \in \Phi^{+}_{N_S}$ and let $N^{+}(2^m \Z_2)$ be the subgroup $\widetilde{F}_4(\Q_2)$ generated by $x_{\alpha}(2^m \Z_2)$ for all $\alpha \in \Phi^{+}_{N}$.  Similarly define $N^{-}(2^m \Z_2)$ and $N_S^{-}(2^m \Z_2)$

Set $U_R^+(4,1)$ to be the subgroup generated by $N_S^+(4 \Z_2)$ and $N^+(\Z_2)$.  Let $U_R^{-}(1,4)$ denote the subgroup generated by $N_S^{-}(\Z_2)$ and $N^{-}(4\Z_2)$.  Finally, we define $K_R'(4)$ to be the subgroup generated by $U_R^{-}(1,4)$, $K_{M_R}^*(4)$, and $U_R^{+}(4,1)$.

The goal of this section is to prove the following theorem.
\begin{theorem}\label{thm:KR'4} Let the notation be as above.
	\begin{enumerate}
		\item One has $K_R'(4) = U_{R}^{-}(1,4)\cdot K_{M_R}^*(4) \cdot U_R^{+}(4,1)$.
		\item The map $K_R'(4) \rightarrow F_4(\Q_2)$ is injective.
	\end{enumerate}
\end{theorem}
As in the proof of Theorem \ref{thm:iwahoriSplit}, part (1) of Theorem \ref{thm:KR'4} implies part (2).  We set $U' = U_{R}^{-}(1,4)\cdot K_{M_R}^*(4) \cdot U_R^{+}(4,1)$.  In the following subsection, we will prove $U' = K_R'(4)$, thus giving Theorem \ref{thm:KR'4}.

We now state a corollary of Theorem \ref{thm:KR'4} that we will need.  Denote by $\Phi_{1,1}^{+}$ the roots $\sum_{i}{m_i \alpha_i}$ with both $m_1, m_2 > 0$. Then $\Phi_{N}^{+}$ is a disjoint union of $\{\alpha_1\}$ and $\Phi_{1,1}^+$. Set $U_{1,1}^+(\Z_2)$ the subgroup generated by $x_{\alpha}(\Z_2)$ for all $\alpha \in \Phi_{1,1}^+$.  Define $U_{1,1}^{-}(4\Z_2)$ similarly.
\begin{corollary}\label{cor:R'IwahoriQ} The group $K_R'(4)$ has an Iwahori factorization with respect to $Q$.
\end{corollary}
\begin{proof}  Suppose $g \in K_R'(4)$.  By Theorem \ref{thm:KR'4}, we have $g = n_1 k n_2$ with $n_1 \in U_R^{-}(1,4)$, $k \in K_{M_R}^*(4)$ and $n_2 \in U_{R}^{+}(4,1)$.  Conjugating all terms of the form $x_{-\alpha_1}(4u)$ in $n_1$ to the right, one can write $n_1 = n_1' n_1''$, where $n_1'$ in the group generated by $N_S^{-}(\Z_2)$ and $U_{1,1}^{-}(4\Z_2)$ and $n_1'' \in x_{-\alpha_1}(4\Z_2)$.  Similarly, one can write $n_2 = n_2'' n_2'$, with $n_2'$ in the group generated by $N_S^{+}(4\Z_2)$ and $U_{1,1}^{+}(\Z_2)$ and $n_2'' \in x_{\alpha_1}(\Z_2)$.  This gives $g = n_1'( n_1'' k n_2'') n_2'$, which is the desired Iwahori factorization.
\end{proof}

In order to prove Theorem \ref{thm:KR'4}, we will need an Iwahori factorization of a smaller group.  Thus let $H_R'(4)$ be the subgroup of $\widetilde{F}_4(\Q_2)$ generated by $N_{S}^{-}(\Z_2), K_{M_R}^*(4)$ and $N_S^{+}(4\Z_2)$.  Then we have:
\begin{proposition}\label{prop:HR'4} One has the Iwahori factorization $H_R'(4) = N_{S}^{-}(\Z_2) \cdot K_{M_R}^*(4) \cdot N_S^{+}(4\Z_2)$.
\end{proposition}
\begin{proof} The proof is just like the proof of Theorem \ref{thm:iwahoriSplit}, except easier, and with the following change.  Set $U_H$ to be the product $N_{S}^{-}(\Z_2) \cdot K_{M_R}^*(4) \cdot N_S^{+}(4\Z_2)$.  Then one shows that $U_H \cdot H_R'(4) = U_H$, i.e., one uses right multiplication by $H_R'(4)$.  To do this, one first checks that $U_H \cdot N_S(4\Z_2) = U_H$ and $U_H \cdot K_{M_R}^*(4) = U_H$.  Then one checks $U_H \cdot x_{-\alpha_2}(\Z_2) = U_H$.  For this statement, one uses the formula $x_{\alpha_2}(t) x_{-\alpha_2}(s)  = x_{-\alpha_2}(s/(1+ts))h_{\alpha_2}(1+ts) x_{\alpha_2}(t/(1+ts))$ for $t \in 4\Z_2$ and $s \in \Z_2$, and the fact that if $\beta \in \Phi_{N_S}^+$, $\beta \neq \alpha_2$, then any root $\gamma = a\beta-b \alpha_2$ with $a,b$ positive integers is necessarily in $\Phi^{+}$.
	
	Combining the above facts, one obtains just as in Lemma \ref{lem:U+N} and Proposition \ref{prop:U+N}, that there is $A\gg0$ such that $U_H \cdot N_S^{-}(2^A) = U_H$.  Now one inducts down on $A$, as in the proof of Proposition \ref{prop:inductDown}.
\end{proof}

\subsubsection{The Iwahori factorization}
In this subsection, we prove part (1) of Theorem \ref{thm:KR'4}. We begin with the analogue of Lemma \ref{lem:easyU}.

\begin{lemma}\label{lem:easyU'} One has $U_{R}^{-}(1,4) \cdot U' = U'$ and $K_{M_R}^*(4) \cdot U' = U'$.
\end{lemma}
\begin{proof} The proof follows similarly to the proof of Lemma \ref{lem:easyU}.  The only difference is that we now must make a slightly more refined statement: if $k \in K_{M_R}^*(4)$, $n \in N_S^{-}(\Z_2)$ and $n' \in N^{-}(4\Z_2)$ then $k n k^{-1} = \{k,n\}n \in N_S^-(\Z_2)$ and $k n' k^{-1} = \{k,n'\}n' \in N^{-}(4\Z_2)$.  These statements follow from the commutator formula.
\end{proof}

\begin{lemma}\label{lem:KMRcong} If $k \in K_{M_R}^*(4)$ and $n \in N_S^+(4 \Z_2)$ then $k n k^{-1} \in N_S^+(4\Z_2)$.\end{lemma}
\begin{proof} This is straightforward, using the commutator formula.\end{proof}

\begin{lemma} One has $N_{S}^{+}(4\Z_2) \cdot U' = U'$.
\end{lemma}
\begin{proof} Suppose $g = n_1 m n_2 \in U'$.  Write $n_1 = n_1' n_1''$ with $n_1' \in N^{-}(4\Z_2)$ and $n_1'' \in N_{S}^{-}(\Z_2)$.  If $u \in N_S^{+}(4\Z_2)$, then $ug = u n_1' u^{-1} (u n_1'' m) n_2$.  Now, $u n_1' u^{-1}$ is in $N^{-}(4\Z_2)$.  Moreover, $u n_1'' m \in H_R'(4)$.  The lemma follows now by Proposition \ref{prop:HR'4}.
\end{proof}

We now observe:
\begin{lemma} One has $x_{\alpha_1}(\Z_2) \cdot U' = U'$. 
\end{lemma}
\begin{proof}  Note that if $a,b$ are positive integers and $\alpha \in \Phi_{N_S}^{+}$ then $\gamma = a \alpha_1 - b \alpha$ is never a root.  Consequently, $x_{\alpha_1}(\Z_2)$ commutes with $N_S^{-}(\Z_2)$.  Moreover, $x_{\alpha_1}(\Z_2)$ normalizes $U_Q^{-}(4\Z_2)$.  Finally, one has the identity $$x_{\alpha_1}(t) x_{-\alpha_1}(s) = x_{-\alpha_1}(s/(1+ts))h_{\alpha_1}(1+ts) x_{\alpha_1}(t/(1+ts))$$ for $t \in \Z_2$ and $s \in 4 \Z_2$.  Combining these formulas, it is easy to see that $x_{\alpha_1}(\Z_2) \cdot U' = U'$.
\end{proof}

Let $K_R'(4,2^m)$ be the subgroup of $\widetilde{F}_4(\Q_2)$ generated by:
\begin{enumerate}
	\item $N^{-}(4\Z_2)$
	\item $N_S^{-}(\Z_2)$
	\item $K_{M_R}^*(4)$
	\item $N_S^{+}(4\Z_2)$
	\item and $N^{+}(2^m \Z_2)$.
\end{enumerate}

Arguing as in Lemma \ref{lem:U+N} and Proposition \ref{prop:U+N}, one obtains the from the above the following.
\begin{proposition} There is $A\gg0$ so that $K_R'(4,2^A) \cdot U' = U'$.
\end{proposition}

We can now prove Theorem \ref{thm:KR'4} by inducting down on $A$:
\begin{proof}[Proof of Theorem \ref{thm:KR'4}] As stated, suppose that $m \geq 0$ is such that $K_R'(4,2^{m+1}) \cdot U' = U'$.  We need to check that $K_R'(4,2^{m}) \cdot U' = U'$.  Thus suppose $g = n_1' n_1'' m n_2 \in U'$ with $n_1' \in N^{-}(4\Z_2)$, $n_1'' \in N_S^{-}(\Z_2)$, $m \in K_{M_R}^*(4)$ and $n_2 \in U_R^{+}(4,1)$.  Let $u = x_{\alpha}(2^m s)$ with $\alpha \in \Phi_{N}^{+}$ and set $h = n_1'' m$. Then $u g = (un_1' u^{-1}) h (h^{-1}uh) n_2$.  Now $h^{-1} u h$ is seen to be in $N^{+}(\Z_2)$.  Moreover, as in the proof of Proposition \ref{prop:inductDown}, one checks that $u n_1' u^{-1} \in K_R'(4,2^{m+1})$.  The theorem follows.
\end{proof}

\section{Integral models}
In the previous sections, we have defined integral models of the algebraic groups $G_2$ and $F_4$ using the Chevalley--Steinberg generators and relations at each finite place. To do computations in the later sections, and to coherently relate these integral models to lattices in $\G_J(\R)$, we will need a somewhat explicit understanding of these integral models.  In this section, we give such explicit integral models for the groups $G_2$ and $F_4$.  Via the work of Steinberg, this amounts to giving a Chevalley basis of the corresponding Lie algebras, which is what we do.

\subsection{Type $G_2$}
We define $\g_{2,\Z} := M_3(\Z)^{\tr=0} \oplus V_3(\Z) \oplus V_3^\vee(\Z)$.  A Chevalley basis can be given by $X_{\alpha}$ being: $E_{ij}$ in $M_3(\Z)^{\tr=0}$, $v_1, v_2, v_3$ in $V_3(\Z)$, and $-\delta_1,-\delta_2,-\delta_3$ in $V_3^\vee(\Z)$.  Here $v_1, v_2, v_3$ is the standard basis of $V_3$ and $\delta_1, \delta_2, \delta_3$ is its dual basis.

\subsection{Type $F_4$}
First we set $J_0 = H_3(\Z)$ to be the symmetric $3 \times 3$ matrices with integer coefficients.  Let $\m_{J}(\Z)$ be the elements of $\m_{J}$ that take $J_0$ to itself.

We set
\[ f_{4,\Z} := (M_3(\Z) \oplus \m_J(\Z))^{2\tr = \mu}/\Z ( \mathbf{1}_3,2\mathbf{1}_{J_0}) \oplus V_3(\Z) \otimes J_0 \oplus V_3(\Z)^\vee \otimes J_0,\]
where the notation is as follows.  Here $\mu: \m_J\rightarrow \Q$ is the map satisfying
\[(\phi X_1,X_2, X_3) + (X_1, \phi X_2, X_3) + (X_1, X_2, \phi X_3) = \mu(\phi)(X_1,X_2,X_3).\]
A pair $(\phi_1,\phi_2) \in M_3(\Z) \oplus \m_J(\Z)$ is in $(M_3(\Z) \oplus \m_J(\Z))^{2\tr = \mu}$ if $2\tr(\phi_1) = \mu(\phi_2)$.  Note that the pair $(\mathbf{1}_3,2\mathbf{1}_{J_0})$ is in $(M_3(\Z) \oplus \m_J(\Z))^{2\tr = \mu}$ and we quotient out by the integer multiples of this pair.

We identify the quotient $(M_3(\Z) \oplus \m_J(\Z))^{2\tr = \mu}/\Z ( \mathbf{1}_3,2\mathbf{1}_{J_0})$ with a subset of $\sl_3 \oplus \m_J^0$ via
\[(\phi_1,\phi_2) \mapsto \phi_1 + \phi_2 - tr(\phi_1)\mathbf{1} := \left(\phi_1 - \frac{\tr(\phi_1)}{3} \mathbf{1}_3\right) + \left(\phi_2 - \frac{\mu(\phi_2)}{3} \mathbf{1}_{J_0}\right) \in \sl_3 + \m_J^0.\]
It is easy to see that this element acts on $V_3(\Z) \otimes J_0$ and $V_3(\Z)^\vee \otimes J_0$ preserving these integral structures.

One has the following proposition.
\begin{proposition} The lattice $f_{4,\Z}$ is closed under the bracket.
\end{proposition}

Now, we observe that because $J_0 = H_3(\Z)$, $\m_J = M_3(\Q)$ with $X \in M_3(\Z)$ acting on $Y \in H_3(\Q)$ as $XY + YX^t$.  Moreover, one can check by easy explicit calculation, $M_3(\Z) = \{X \in m_J(\Z): \mu(X) \in 2 \Z\}.$

Consequently, we have
\[f_{4,\Z} = (M_3(\Z) + M_3(\Z))^{\tr_1 = \tr_2}/\Z(\mathbf{1},\mathbf{1}) + V_3(\Z) \otimes J_0 + V_3(\Z)^\vee \otimes J_0.\]

For the Chevalley basis, we take the usual bases of $X_{\alpha} = E_{ij}$ of the two copies of $M_3(\Z)$.  Now $J_0$ is the $\Z$-span of
\[\{e_{11},e_{22},e_{33}, x_1 = e_{23}+e_{32},x_2 = e_{31}+e_{13}, x_3 = e_{12} + e_{21}\},\]
where $e_{ij}$ denotes the element of $M_3(\Z)$ with a $1$ in the $(i,j)$ location and zeros elsewhere.  For the rest of the Chevalley basis, we take the elements $v_j \otimes x_k, v_j \otimes e_{kk}, -\delta_j \otimes x_k$ and $-\delta_j \otimes e_{kk}$.

\section{Splittings}\label{sec:splittings} We may now combine our local results to construct splittings of certain congruence subgroups of $G_2(\R)$ and $F_4(\R)$ into the double cover.

Recall that when $p > 2$ is odd, we have the hyperspecial maximal compact subgroup $K_p= G(\Z_p)$ of $G(\Q_p)$ induced by our integral model.
\begin{lemma}\label{Lem: splitting away from 2}\cite[Proposition 2.1]{lokeSavin}
	The central extension $\G(\Q_p)$ splits over $K_p$. The splitting homomorphism $s_p: K_p\lra \G(\Q_p)$ is unique, and we denote its image by $K_p^\ast$.
\end{lemma}
We now define a congruence subgroup $\Ga_{F_4}(4) \subseteq F_4(\R)$ and explain that this subgroup splits into $\widetilde{F}_4(\R)$.  Let $K_R(4)$ be the image in $F_4(\Q_2)$ of the subgroup $K_R'(4)$, and let $s_2:K_R(4) \rightarrow \widetilde{F}_4(\Q_2)$ be the induced splitting.   Define now
\begin{equation}\label{eqn: good level}
\Ga_{F_4}(4) := F_4(\Q) \cap K_R(4)\prod_{p >2}{K_p}\subset F_4(\Z).
\end{equation}

To construct a splitting of $\Ga_{F_4}(4)$ into $\widetilde{F}_4$, we will use the following lemma.
\begin{lemma}\label{Lem: unique splitting} Suppose $A,B$ are two groups containing a central $\mu_2$, and $\Gamma \subseteq A/\mu_2, B/\mu_2$.  Let $s: \Gamma \rightarrow (A \times B)/\mu_2^{\Delta}$, and $s_A: \Gamma \rightarrow A$ be given splittings. Then there exists a unique splitting $s_B: \Gamma \rightarrow B$ so that $s(\gamma) = (s_A(\gamma),s_B(\gamma))\mu_2^{\Delta}$ for all $\gamma \in \Gamma$. \end{lemma}
\begin{proof} By assumption, for each $\gamma \in \Gamma$ one has $s(\gamma) = \pm (s_A(\gamma),s_B(\gamma))$ for a unique $s_B(\gamma) \in B$.  This uniquely determines the map $s_B: \Gamma \rightarrow B$, and one checks that it is a group homomorphism.
\end{proof}

Using the inclusion $\Ga_{F_4}(4)\subset F_4(\Q)\subset F_4(\R)$, we obtain a splitting $s_{\Gamma}: \Ga_{F_4}(4) \rightarrow \widetilde{F}_4(\R)$ by applying Lemma \ref{Lem: unique splitting} with $\Gamma = \Ga_{F_4}(4)$, $A = \widetilde{F}_4(\A_f)$, and $B = \widetilde{F}_4(\R)$. Let $s_{f}:\Ga_{F_4}(4) \rightarrow \widetilde{F}_4(\A_f)$ be the section induced from the local sections $s_p$ from Lemma \ref{Lem: splitting away from 2} and Theorem \ref{thm:iwahoriSplit}.  With this notation, we have obtained

\begin{proposition} There is a unique splitting $s_\Gamma: \Ga_{F_4}(4) \rightarrow \widetilde{F}_4(\R)$ characterized by the fact that $s_\Q(\gamma) = \pm (s_f(\gamma),s_{\Gamma}(\gamma))$ for all $\gamma \in \Ga_{F_4}(4)$.\end{proposition}

Below we will need the following proposition.
\begin{proposition}\label{prop:sGammaUnip} For all integers $u$, the splitting $s_\Gamma$ satisfies $s_\Gamma(x_\alpha(u)) = x_\alpha(u)$ for all $\alpha \in \Phi_{N}^{+} \cup \Phi_{N_S}^{-} \cup \Phi_{M_R}$ and $s_\Gamma(x_{\alpha}(4u)) = x_{\alpha}(4u)$ for all $\alpha \in \Phi_{N_S}^{+} \cup \Phi_{N}^{-}$.
\end{proposition}
\begin{proof} Indeed, this compatibility occurs for $s_\Q$ and $s_p$ for all $p=2,3,\ldots$.  The proposition thus follows from the definition of $s_\Gamma$.
\end{proof}

In the next section, we recall the inclusion of algebraic $\Q$-groups $G_2 \subseteq F_4$ and prove an inclusion $\widetilde{G}_2(\R) \subseteq \widetilde{F}_4(\R)$.  Assuming these inclusions for the moment, we set $\Gamma_{G_2}(4) = G_2(\R) \cap \Gamma_{F_4}(4)$ and obtain a splitting $\Gamma_{G_2}(4) \rightarrow \widetilde{G}_2(\R)$.

\section{Group embeddings}\label{sec:grpEmbeddings}We conclude this chapter with some remarks about the inclusion of $G_2$ in $F_4$.

\subsection{Algebraic groups over $\Q$}
We recall the following proposition from the theory of algebraic groups; see \cite[Theorem 25.4(c)]{milneBook}.
\begin{proposition}\label{Prop: milne} Suppose $k$ is a field of characteristic $0$, $H, G$ are algebraic groups over $k$, with $H$ semisimple, connected and simply connected.  Suppose $L:\mathfrak{h} \rightarrow \mathfrak{g}$ is an embedding of Lie algebras.  Then there exists a unique map $H \rightarrow G$ of algebraic groups whose differential is $L$.\end{proposition}

We first work with algebraic groups over $\Q$.  Either from the proposition or directly, one can see easily that there is a map $\SL_3 \rightarrow F_4$, lifting the Lie algebra embedding $\m_J^0 \rightarrow \mathfrak{f}_4$ in the notation of \cite{pollackQDS}.  Let $\SO(3)$ denote the group of $g \in \SL_3$ with $g^t g = 1$. Composing with the map $\SO(3)\rightarrow \SL_3$, we get an embedding of $\SO(3)$ into $F_4$.
\begin{lemma}
	The connected component of the identity of the centralizer of $\SO(3)$ in $F_4$ is a split form of type $G_2$.
\end{lemma}
\begin{proof}
	Denote by $G'$ this identity component of the centralizer of $\SO(3)$ in $F_4$.  We first observe that on the level of Lie algebras, we have $\g_2 \rightarrow \mathfrak{f}_4$, and this $\g_2$ is the exactly $\mathfrak{f}_4^{\SO(3)}$.  Consequently, the action of $G'$ on $\mathfrak{f}_4$ induces an action of $G'$ on $\g_2$, so we obtain a map $\alpha: G' \rightarrow G_2$, because $G_2$ is defined as the group of automorphisms of its Lie algebra.
	
	In the reverse direction, Proposition \ref{Prop: milne} implies the existence of a map $\beta: G_2 \rightarrow F_4$ lifting the inclusion of Lie algebras $\g_2 \rightarrow \mathfrak{f}_4$.  The image of this $G_2$ centralizes $\SO(3)$ by uniqueness of the lift: if $g \in \SO(3)$ then $g\beta(h)g^{-1}$ is another lift, so is equal to $\beta$.  Consequently, $\beta$ gives a map $G_2 \rightarrow G'$.  The map $\alpha \circ \beta: G_2 \rightarrow G_2$ induces the identity on Lie algebras by construction, so is the identity.  Similarly, the map $\beta \circ \alpha: G' \rightarrow G'$ induces the identity of Lie algebras, so is the identity.
\end{proof}

\subsection{Real Lie groups} We now work with real Lie groups.  We will explain the fact that the (identity component of the) centralizer of $\SO(3)$ in $\widetilde{F}_4$ is the group $\widetilde{G}_2$; see also \cite{HPS}.

First consider the case of the linear group $F_4$.
\begin{lemma}
	The identity component of the centralizer of $\SO(3)$ in $F_4(\R)$ is $G_2(\R)$.
\end{lemma}
\begin{proof}
	Via its action on $\mathfrak{f}_4^{\SO(3)} = \mathfrak{g}_2$, this centralizer maps to $Aut(\g_2)$.  Thus the identity component of this centralizer maps to the connected Lie group $G_2(\R)$.  Moreover, the identity component of the centralizer of $\SO(3)$ in $F_4$ has Lie algebra exactly $\mathfrak{f}_4^{\SO(3)} = \mathfrak{g}_2$ (it is easy to see that the Lie algebra is contained in this set, and it is surjective by considering the exponential map.) It thus remains to determine which Lie group of type $G_2$ this is.
	
	Because we already know $G_2 \rightarrow F_4$ as real algebraic groups, we obtain $G_2(\R) \rightarrow C_{F_4(\R)}(\SO(3)(\R))$.  Because the connected double cover of $G_2(\R)$ does not split over $G_2(\R)$, the centralizer of $\SO(3)$ must be the linear group $G_2(\R)$.
\end{proof}

Now, for the case of covering groups. First observe that $\SO(3) \subseteq \SL_3(\R) \subseteq R_J^+$, so the $\SO(3)$ splits into $\widetilde{F}_4$ by Lemma \ref{Lem: SL split}; the splitting is unique because $\SO(3)$ is equal to its derived group.
\begin{lemma}
	The identity component of the centralizer $C_{\widetilde{F}_4}(\SO(3))^0$ of $\SO(3)$ in $\widetilde{F}_4$ is identified with $\widetilde{G}_2(\R)$.
\end{lemma}
\begin{proof}
	 Let $G'$ be the identity component of the centralizer of this $\SO(3)$ in $\widetilde{F}_4$.  Then $G'$ consists of elements $(g, j_{1/2}(g))$ where $j_{1/2}(g): X_{F_4} \rightarrow \GL_2(\C)$ is a continuous map whose symmetric square is $j_{lin}(g): X_{F_4} \rightarrow \GL_3(\C)$.  Every element $g\in F_4(\R)$ occurring in such a pair commutes with $\SO(3)$, so that $g\in G_2(\R)$.  We thus obtain a map $G' \rightarrow G_2(\R)$.  An argument with the exponential map and Lie algebras proves that this map is surjective, because $G_2(\R)$ is generated by the image of the exponential map.
	
	We now construct a map $G' \rightarrow \widetilde{G}_2$.  We claim that $G_2(\R)/K_{G_2} = X_{G_2}$ embeds into $F_4(\R)/K_{F_4} = X_{F_4}$; this follows from the claim that the maximal compact subgroups $K_{G_2}$ and $K_{F_4}$ satisfy $K_{G_2} = G_2(\R) \cap K_{F_4}$.  Granting this for a moment, if $(g,j_{1/2}(g))$ is in $G'$, restricting $j_{1/2}(g)$ to $X_{G_2}$ gives an element of $\widetilde{G}_2$.  We therefore obtain $G'\rightarrow \widetilde{G}_2$, which covers the identity map on $G_2(\R)$.  Because $G'$ is a connected Lie group with Lie algebra $\mathfrak{g}_2$, and $G_2(\R)$ doesn't split into $\widetilde{G}_2$, the map $G' \rightarrow \widetilde{G}_2$ is an isomorphism.
	
	To see that  $K_{G_2} = G_2(\R) \cap K_{F_4}$, first recall that $K_{G_2}$ and $K_{F_4}$ are the subgroups of $G_2(\R)$, respectively $F_4(\R)$, that also preserve the bilinear form $B_{\theta}(X,Y) := - B(X,\theta Y)$ on $\g_2$, respectively, $\mathfrak{f}_4$, where $\theta$ is the Cartan involution on these Lie algebras.  Because the Cartan involution $\theta$ on $\mathfrak{f}_4$ described in \cite{pollackQDS} restricts to the one on $\mathfrak{g}_2$, it is clear that $G_2(\R) \cap K_{F_4}$ is contained in $K_{G_2}$.  For the reverse inclusion, one notes that $K_{G_2}$ can be generated by the exponentials of elements of $\mathfrak{f}_4^{\SO(3),\theta = 1} \subseteq \mathfrak{f}_4^{\theta=1}$, which are in $K_{F_4}$.
\end{proof}

\begin{remark}
	The fact that the Cartan involution on $\mathfrak{f}_4$ restricts to the one on $\mathfrak{g}_2$ plays a useful role in verifying that the pullback of a modular form on $\widetilde{F}_4$ to $\widetilde{G}_2(\R)$ remains a modular form.
\end{remark}

\chapter{Modular forms} In this chapter, we define  quaternionic modular forms of half-integral weight, generalizing the integral weight theory of \cite{pollackQDS} and prove the main results about their Fourier expansions and Fourier coefficients. We then assert the existence of a certain modular form $\Theta_{F_4}$ of weight $\frac{1}{2}$ on $\widetilde{F}_4(\R)$, the proof of which we defer to Chapter \ref{Chapt: minimal aut}. Finally, we consider the pull back of $\Theta_{F_4}$ to $\widetilde{G}_2(\R)$, proving Theorems \ref{thm:introGWF} and \ref{thm:introMain2} of the introduction. Along the way, we also do arithmetic invariant theory related to cubic rings and their inverse differents.

\section{Quaternionic modular forms}\label{Sec: quatMF} We begin by studying quaternionic modular forms of half-integral weight.  Suppose $\ell \geq 1$ is an odd integer and recall that $\mathbf{V}_{\ell/2} := Sym^{\ell}(\mathbb{V}_2)$.  We consider $\mathbf{V}_{\ell/2}$ as a representation of $\widetilde{K}_J$ via the map $j_{1/2}(\cdot, x_0): \widetilde{K}_J \rightarrow \GL_2(\mathbb{V}_2)$.  A modular form on $G_J(\R)$ (or, more accurately, $\G_J$) of weight $\ell/2$ will be a certain $\mathbf{V}_{\ell/2}$-valued automorphic function.

To define the appropriate sorts of functions on $\G_J$ that we will be considering, we require a certain differential operator. Let $\g(J) \otimes \C = \k \oplus \p$ be the Cartan decomposition of the Lie algebra $\g(J) \otimes \C$, which we identify with the complexified Lie algebra of $\G_J$. In \cite[section 5]{pollackQDS}, an identification is given between $\p$ and $\mathbb{V}_2 \otimes W_J$ over $\C$. Let $\{X_{\alpha}\}$ be a basis of $\p$ and $\{X_{\alpha}^\vee\}$ the dual basis of the dual space $\p^\vee$.  Suppose now that $\varphi$ is a smooth $\mathbf{V}_{\ell/2}$-valued function on $\G_J$ satisfying $\varphi(gk) = k^{-1} \cdot \varphi(g)$ for all $g \in \G_J$ and $k \in \widetilde{K}_J$.  For such a function, we define $D_{\ell/2}' \varphi(g) = \sum_{\alpha}{X_{\alpha}\varphi(g) \otimes X_{\alpha}^\vee}$, which is valued in
\[\mathbf{V}_{\ell/2} \otimes \p^\vee \simeq \Sym^{\ell-1}(\mathbb{V}_2) \otimes W_J \oplus \Sym^{\ell+1}(\mathbb{V}_2) \otimes W_J.\]
Let $pr: \mathbf{V}_{\ell/2} \otimes \p^\vee \rightarrow \Sym^{\ell-1}(\mathbb{V}_2) \otimes W_J$ be the $\widetilde{K}_J$-equivariant projection and define the operator $D_{\ell/2} = pr \circ D_{\ell/2}'$.

Suppose $\Gamma \subseteq G_J$ is a congruence subgroup such that there exists a splitting $s_\Gamma: \Gamma \rightarrow \G_J$.  We fix this splitting, which need not be unique.

\begin{definition}\label{def:QMF} Suppose $\ell \geq 1$ is an odd integer.  A quaternionic modular form on $\G_J$ of weight $\ell/2$ and level $(\Gamma,s_\Gamma)$ is a smooth function $\varphi: s_{\Gamma}(\Gamma)\backslash \G_J \rightarrow \mathbf{V}_{\ell/2}$ of moderate growth satisfying
	\begin{enumerate}
		\item\label{equivar} $\varphi(gk) = k^{-1} \cdot \varphi(g)$ for all $g \in \G_J$ and $k \in \widetilde{K}_J$
		\item $D_{\ell/2}\varphi \equiv 0$.
	\end{enumerate}
\end{definition}

Our first main result will be to show that such a definition of quaterionic modular form of half-integral weight has a robust theory of Fourier coefficients, generalizing the integral weight theory of \cite{pollackQDS} and its antecedents.

\subsubsection{Adelic formulation}\label{Sec: adelic perspective}
Suppose that $\mathbf{G}_J$ is a reductive group over $\Q$ such that $\mathbf{G}_J(\R)$ is an adjoint quaternionic exceptional group. Following our conventions from Section \ref{Sec: general covers}, we further assume we are given a metaplectic double cover $\widetilde{\mathbf{G}}_J(\A)$ of $\mathbf{G}_J(\A)$ coming from the appropriate Brylinski--Deligne extension. We thus have a short exact sequence of locally-compact topological groups
\[
1\lra \mu_2(\Q)\lra \widetilde{\mathbf{G}}_J(\A)\lra \mathbf{G}_J(\A)\lra 1,
\]
which splits canonically over $G_J(\Q)$; let $s_{\Q}$ denote this splitting. There is a decomposition $\widetilde{\mathbf{G}}_J(\A) = \prod_p\widetilde{\mathbf{G}}_J(\Q_p)/\mu_2^{+}$. Our convention implies that $\widetilde{\mathbf{G}}_J(\R)\cong \G_J$.

Then for all but finitely many odd primes $p$, $\mathbf{G}_J$ is unramified and contains a hyperspecial subgroup $K_p:=\mathbf{G}_J(\Z_p)$ over which the cover $\widetilde{\mathbf{G}}_J(\Q_p)\to \mathbf{G}_J(\Q_p)$ splits \cite[Section 7]{Weissman}. Let $T$ be a finite number of primes containing $2$ such that for $p\notin T$, the above statement holds. Let $K^T\subset \mathbf{G}_J(\A_T):=\prod_{p\in T}\mathbf{G}_J(\Q_p)$ be a given compact subgroup equipped with a splitting
\[
\begin{tikzcd}
&\widetilde{\mathbf{G}}_J(\A_T)\ar[d]\\
K^T\ar[ur,"s^T"]\ar[r]&\mathbf{G}_J(\A_T),
\end{tikzcd}
\]
where $\widetilde{\mathbf{G}}_J(\A_T):=\prod_{p\in T}\widetilde{\mathbf{G}}_J(\Q_p)/\mu^+_2$.

Setting $K_T: = K^T\prod_{p\notin T} K_p$, we have a splitting $s_T: K_T\to \widetilde{\mathbf{G}}_J(\A_f)$; let $K^\ast_T$ denote its image.
\begin{definition}\label{def:QMF adelic} Suppose $\ell \geq 1$ is an odd integer.  An adelic quaternionic modular form on $\G_J(\A)$ of weight $\ell/2$ and level $(K_T,s_T)$ is a smooth function
	\[
	\varphi: G_J(\Q)\backslash \widetilde{\mathbf{G}}_J(\A) \rightarrow \mathbf{V}_{\ell/2}
	\]
	of moderate growth satisfying
	\begin{enumerate}
		\item $\varphi(gk_\infty) = k_\infty^{-1} \cdot \varphi(g)$ for all $g \in \widetilde{\mathbf{G}}_J(\A)$ and $k \in \widetilde{K}_\infty$,
		\item $\varphi(gk) = \varphi(g)$ for all $g \in \widetilde{\mathbf{G}}_J(\A)$ and $k \in {K}_T^\ast,$
		\item $D_{\ell/2}\varphi \equiv 0$, where $D_{\ell/2}$ is the Schmid operator from Definition \ref{def:QMF}.
	\end{enumerate}
\end{definition}

The compatibility of the definitions may be seen as follows. Set $\Gamma(K_T) := \mathbf{G}_J(\Q) \cap K_T$. Using the inclusion $\Gamma(K_T) \subset \mathbf{G}_J(\Q)\subset G_J(\R)$, we obtain a splitting $s_{\Gamma}: \Gamma(K_T) \rightarrow \G_J$ via Lemma \ref{Lem: unique splitting}.  Applying this lemma to $A= \widetilde{\mathbf{G}}_J(\A_f)$, $B=\G_J$, and $\Ga = \Ga(K_T)$, we see that the pullback of $\varphi$ to $\G_J \subset \widetilde{\mathbf{G}}_J(\A)$ gives a quaternionic modular form of level $(\Ga(K_T),s_\Gamma)$ as in Definition \ref{def:QMF}.

\section{Generalized Whittaker functions}
 We now investigate the so-called generalized Whittaker functions associated to quaternionic modular forms.  In other words, we reproduce the main result of \cite{pollackQDS} except now in the half-integral weight case.  Because almost all of the proof in \cite{pollackQDS} carries over, we are quite brief.

We begin with the following crucial proposition. Recall that an $\omega = (a,b,c,d) \in W_J(\R)$ is said to be positive semi-definite if the function $p_{\omega}(Z) = aN(Z) + (b,Z^\#) + (c,Z) + d$ is never $0$ on the upper-half space $\mathcal{H}_J=\{X+iY: X,Y \in J, Y > 0\}$.
\begin{proposition} Consider the function $g \mapsto \langle \omega, g r_0(i) \rangle$ on $H_J(\R)^{+}$, and suppose $\omega$ is positive semi-definite.  Then there exists a smooth genuine function $\alpha_\omega(g): \widetilde{H_J(\R)^+} \rightarrow \C$ satisfying $\alpha_\omega(g)^2 = \langle \omega, g r_0(i)\rangle$.
\end{proposition}
\begin{proof} Recall from \cite{pollackQDS} that $\langle \omega, gr_0(i)\rangle = - j(g,i) p_{\omega}(g\cdot i)$.  Because $\mathcal{H}_J$ is contractible and $p_{\omega}(Z)$ is never $0$ on $\mathcal{H}_J$, $p_{\omega}(Z)$ has a smooth square root on $\mathcal{H}_J$. This follows from covering space theory: the map $\C^\times \rightarrow \C^\times$ via $z \mapsto z^2$ is a cover, so the map $Z \mapsto p_{\omega}(Z)$ from $\mathcal{H}_J \rightarrow \C^\times$ lifts to the first copy of $\C^\times$.  Let us pick, arbitrarily, one of the two square roots and call it $p_{\omega}(Z)^{1/2}$.
	
	Now, the function $g \mapsto j(g,i)$ on $H_J(\R)^+$ has a genuine square root  $j_{1/2}$ on $\widetilde{H_J(\R)^+}$; such a function was constructed in the Lemma \ref{lem:j12}.  Thus $\alpha_\omega(g) = \sqrt{-1} j_{1/2}(g,x_0) p_\omega(gi)^{1/2}$ is the desired function.\end{proof}

We can now state the main theorem of this section.  To do so, we make some notations. First, let $n = \frac{\ell}{2} \in \frac{1}{2} + \Z_{\geq 0}$. Suppose $\omega \in W_J(\R)$ is positive semi-definite.  Let $\alpha_\omega(g)$ be one of the two square roots of $\langle \omega, g r_0(i)\rangle$ to $\widetilde{H_J(\R)^+}.$  For $g \in \widetilde{H_J(\R)^+}$, define

\begin{equation}\label{eqn: whittaker function}
\W_{\omega,\alpha_\omega}(g) = \nu(g)^{n+1}\sum_{-n\leq v\leq n}{\left(\frac{|\alpha_\omega(g)|}{\alpha_\omega(g)}\right)^{2v} K_v(|\alpha_\omega(g)|^2) \frac{x^{n+v}y^{n-v}}{(n+v)!(n-v)!}}.
\end{equation}
Here the sum is over half-integers $v \in \frac{1}{2} + \Z$ with $-n \leq v \leq n$. Note that
\begin{enumerate}
	\item $n,v$ are half-integers, i.e., in $\frac{1}{2} + \Z$, so that $n+v$ and $n-v$ are integers;
	\item $\nu(g) > 0$ so $\nu(g)^{n+1}$ makes sense;
	\item $2v$ is an odd integer;
	\item one has $\W_{\omega,-\alpha_\omega}(g) = -\W_{\omega,\alpha_\omega}(g)$;
	\item for $\ep\in \mu_2(\Q)$, one has $\W_{\omega,\alpha_\omega}(\ep g) = \ep\W_{\omega,\alpha_\omega}(g)$.
\end{enumerate}

Let $N_J$ be the unipotent radical of the Heisenberg parabolic of $G_J$.  This subgroup of $G_J(\R)$ splits uniquely into $\G_J$ so we also write $N_J(\R)$ for its image in $\G_J$.  One can extend $\W_{\omega,\alpha_\omega}$ to a function on all of $\G_J$ as
\[\W_{\omega,\alpha_\omega}(nmk) = e^{i \langle \omega, \overline{n} \rangle} k^{-1} \W_{\omega,\alpha_\omega}(m)\]
for $n \in N_J(\R)$, $m \in \widetilde{H_J(\R)^{+}}$, and $k \in \widetilde{K}_J$.  One checks immediately that this is well-defined.

Recall that a generalized Whittaker function of weight $n$ for $\omega$ is a function $\phi: \G_J \rightarrow Sym^{2n}(\mathbb{V}_2)$ satisfying
\begin{enumerate}
	\item $\phi(gk) = k^{-1} \phi(g)$ for all $g \in \G_J$ and $k \in \widetilde{K}_J$;
	\item $\phi(ug) = e^{i \langle \omega, \overline{u}\rangle} \phi(g)$ for all $g \in \G_J$ and $u \in N_J(\R)$. Here $\overline{u}$ is the image of $u \in W_J(\R)$;
	\item $D_{n} \phi \equiv 0$.
\end{enumerate}

\begin{theorem}\label{thm:FE} Suppose $\omega \in W_J(\R)$ is non-zero and $n \in \frac{1}{2} +\Z$ is positive.  Suppose moreover that $\phi: \G_J \rightarrow Sym^{2n}(\mathbb{V}_2)$ is a moderate growth generalized Whittaker function of weight $n$ for $\omega$.
	\begin{enumerate}
		\item If $\omega$ is not positive semi-definite, then $\phi \equiv 0$.
		\item If $\omega$ is positive semi-definite, then $\phi$ is proportional to $\W_{\omega, \alpha_\omega}(g)$.
	\end{enumerate}
\end{theorem}
\begin{proof} The work is nearly identical to \cite{pollackQDS}, so we only sketch the proof.
		
Let us first review the definition of the right regular action of the Lie algebra $\g(J)$ on smooth functions $\phi$ on $\G_J$.  Thus suppose $X \in \g(J)$.  Then for $t \in \R$ sufficiently small, $\exp(tX)$ is an element of $G_J(\R)$ near the identity.  Because $\G_J\rightarrow G_J(\R)$ is a covering space, there is a unique lift, call it $\exp'(tX)$, of $\exp(tX)$ to $\G_J$ that is near the identity of $\G_J$.  Then $(X \phi)(g) := \frac{d}{dt} \phi(g \exp'(tX))|_{t=0}$.  It is a fact that this definition gives a linear action of the real Lie algebra $\g(J)$ on smooth functions on $\G_J$.  One obtains an action of $\g(J) \otimes \C$ by complexification.

Let now $\phi = \sum_{v}{\phi_v \frac{x^{n+v} y^{n-v}}{(n+v)!(n-v)!}}$ be a generalized Whittaker function.   (To make notation consistent with \cite{pollackQDS}, $\lambda = \omega$.)  By the Iwawasa decomposition $\G_J = R_J^+ \widetilde{K}_J$, and because $\phi$ is $\widetilde{K}_J$-equivariant by definition, to determine $\phi$ it suffices to determine $\phi$ on $R_J^+$.

Now, recall that $R_J^+$ splits into $\G_J$.  Thus $\phi|_{R_J^{+}}$ can be thought of as function on the linear group $G_J(\R)$, so we may apply \cite[Corollary 7.6.1]{pollackQDS} to obtain differential equations satisfied by $\phi$: Indeed, the proof of this corollary is to write a basis of $X \in \p$ as sums $X = X_1 + X_2$, with $X_1 \in Lie(R_J^+) \otimes \C$ and $X_2 \in Lie(K_J) \otimes \C= Lie(\widetilde{K}_J) \otimes \C$, and use the given action of $Lie(K_J) = Lie(\widetilde{K}_J)$ on $\phi$ to write the differential equation $D_{n}\phi \equiv 0$ in explicit coordinates on $R_J^+$.  In \cite[Corollary 7.6.1]{pollackQDS} recall that:
	\begin{itemize}
		\item $w \in \R^\times_{>0}$ is considered as an element in the center of the group $H_J(\R)^{+}$ which acts on $E_{13}$ as the real number $w^2$ (as opposed to $w^{-2}$).  The element $w$ is in $R_J^+$ so the group of such $w$'s splits into $\G_J$.
		\item $\widetilde{Z} = M r_0(i)$ and $r_0(Z) = (1,-Z,Z^\#,-n(Z))$.
		\item for $E \in J$, $D_{Z(E)}$ denotes the action of the Lie algebra element $\frac{1}{2}M(\Phi_{1,E}) - i n_{L}(E)$, where $\Phi_{1,E}$ is the map $J \rightarrow J$ given by $Z \mapsto \{E,Z\}$ (see \cite[Subsection 3.3.2, equation (7)]{pollackQDS}; see also \cite[Subsection 3.3, equation (3)]{pollackQDS}) and $M(\Phi_{1,E})$ is defined in Subsection 3.4.1 at the top of page 1229 of \cite{pollackQDS}.
		\item $V(E)$ is defined in Subsection 5.1, equation (19) of \cite{pollackQDS}.
	\end{itemize}
	
Now, one solves these equations on a connected open subset $U$ of $\mathcal{H}_J$ where $p_\omega(Z) \neq 0$. To do this, one first argues as in section 8.1 of \cite{pollackQDS} that $\phi_v(w,X,Y)$ (see section 8.2, page 1257) is of the form $w^{2n+2}Y_v(m) K_v(|\langle \omega, \widetilde{Z} \rangle|)$ for some function $Y_v(m)$ that does not depend on $w$. Indeed, the differential equations
\begin{enumerate}
	\item $(w \partial_w -2(n+1)+k)\phi_k = - \langle \omega, \widetilde{Z}^* \rangle \phi_{k-1}$
	\item $(w \partial_w -2(n+1)-k)\phi_k = - \langle \omega, \widetilde{Z} \rangle \phi_{k+1}$
\end{enumerate}
from \cite[Corollary 7.6.1]{pollackQDS}, taken together, imply that $w^{-2n-2}\phi_v(w,X,Y)$ satisfies Bessel's differential equation.  The fact that this function must be of moderate growth as $w \rightarrow \infty$ then implies that, as a function of $w$, it is proportional to $K_v(|\langle \omega, \widetilde{Z} \rangle|)$.

To understand the functions $Y_v(m)= Y_v(X,Y)$, one argues as on the top of page 1257 to obtain that $\phi(w,X,Y) = \phi(w,m)$ is of the form $Y_{1/2}'(m) \W_{\omega,\alpha_\omega}(g)$ for some function $Y_{1/2}'(m)$ that does not depend on $w$. In other words, one uses the differential equations in $w$ above again to relate $Y_{v}(m)$ to $Y_{v+1}(m)$ for each $v$.  Note that the function $Y_{1/2}'(m)$ descends to the linear group.
	
Now one proves that the $\W_{\omega,\alpha_\omega}$ are annihilated by the operator $D_{n}$, exactly as in the proof of Proposition 8.25 of \cite{pollackQDS}. Note that in this proof, the term $|\alpha_\omega(g)| \alpha_\omega(g)^{-1}$ is rewritten as a product of $|\alpha_\omega(g)|^{-1}$ and a term that is annihilated by $D_{Z(E)}$. Moreover, the absolute value $|\alpha_\omega(g)|^{-1}$ descends to the linear group. This is why the manipulations of \cite{pollackQDS} carry over to this half-integral weight case.  In any event, it follows from this that $D_{Z(E)}(Y_{1/2}'(m)) = 0$ and $D_{Z(E)^*}(Y_{1/2}'(m)) = 0$, from which one concludes $Y_{1/2}'(m)$ is constant.
	
Thus the $\W_{\omega,\alpha_\omega}$ are annihilated by the operator $D_{n}$, and on an open subset where $p_\omega(Z) \neq 0$, any moderate growth solution agrees with the $\W_{\omega,\alpha_\omega}$ up to constant multiple.  The rest of the argument now follows as in the proof of Proposition 8.2.4 of \cite{pollackQDS}.
\end{proof}

From Theorem \ref{thm:FE} follows immediately the definition of Fourier coefficients of modular forms of weight $\frac{\ell}{2}$: let $Z = [N_J,N_J]$ denote the one-dimensional center of $N_J$. Let $\varphi$ be a modular form for $\G_J(\A)$ of weight $\frac{\ell}{2}$ and level $(K_T,s_T)$ as in Definition \ref{def:QMF adelic}.  Set $\varphi_{Z}(g) = \int_{Z(\Q)\backslash Z(\A)}{\varphi(zg)\,dz}$ and
\[\varphi_{N}(g) = \int_{N_J(\Q) \backslash N_J(\A)}{\varphi(ng)\,dn}.\]
Then we have the following generalization of \cite[Corollary 1.2.3]{pollackQDS}.
\begin{corollary}\label{Cor: fourier}
	For each positive semi-definite for $\omega\in W_J(\Q)$, there exist a constant $a_\varphi(\omega)$, well-defined up to multiplication by $\pm 1$, such that for $g \in \G_J$,
	\[
	\varphi_Z(g) = \varphi_{N}(g) + \sum_{\omega \in W_J(\Q)}{a_\varphi(\omega) \W_{2\pi\omega}(g)},
	\]
	where the sum runs over over positive semi-definite vectors. The function $\W_{2\pi\omega}(g)$ is one element of the set $\{\W_{2\pi\omega,\alpha_{2\pi\omega}},-\W_{2\pi\omega,\alpha_{2\pi\omega}}\}$.
\end{corollary}

The complex number $a_\varphi(\omega)$ is thus well-defined up to multiplication by $\pm 1$. These numbers $a_\varphi(\omega) \in \C/\{\pm 1\}$ are, by definition, the Fourier coefficients of $\varphi$.
\subsubsection{Remark on $K$-Bessel functions}
The $K$-Bessel functions $K_v(z)$ in the definition of the Whittaker functions $\W_{\omega,\al_\omega}$ only occur for half-integral values of $v$. This is especially nice as these satisfy the following classical lemma.
\begin{lemma}the $K$-Bessel function satisfies the following facts.
	\begin{enumerate}
		\item For any value of $v$, $$-z^v(\partial_z(z^{-v} K_v(z))) = K_{v+1}(z),$$
		\item For any value of $v$, $$K_{-v}(z) = K_v(z).$$
		\item We have
		\[
		K_{1/2}(z) = \sqrt{\frac{\pi}{2z}}e^{-z}.
		\]
	\end{enumerate}
\end{lemma}
Thus, the functions $\W_{\omega,\al_\omega}$ are particularly simple as functions of $\al_\omega(g)$ and $\nu(g)$. For example, when $l=1$, we have
\[
\W_{\omega,\alpha_\omega}(g) = \sqrt{\frac{\pi\nu(g)^3}{2}}\frac{e^{-|\alpha_\omega(g)|^2}}{|\al_\omega(g)|}\left[{\left(\frac{|\alpha_\omega(g)|}{\alpha_\omega(g)}\right) x+\left(\frac{\alpha_\omega(g)}{|\alpha_\omega(g)|}\right) y}\right]
\]
if $g \in \widetilde{H_J(\R)^+}.$

\section{The minimal modular form of $\widetilde{F}_4(\R)$}

Our first application is the existence of a particular modular form of weight $1/2$ on $\widetilde{F}_4(\R)$ with exceptionally few non-zero Fourier coefficients in the sense of Lemma \ref{lem:FElem} below.  This modular form is constructed from the \emph{automorphic minimal representation} $\Pi_{min}$ of $\widetilde{F}_4(\A)$, which we discuss in Chapter \ref{Chapt: minimal aut}.

\begin{theorem}\label{thm: exists theta}
    There exists a modular form $\Theta_{F_4}$ on $\widetilde{F}_4(\R)$ of weight $\frac{1}{2}$ which satisfies the following properties:
    \begin{enumerate}
        \item $\Theta_{F_4}$ is constructed from the automorphic minimal representation;
        \item the level of $\Theta_{F_4}$ is $\Gamma_{F_4}(4)$ defined in equation \eqref{eqn: good level};
        \item the $(0,0,0,1)$-Fourier coefficient of $\Theta_{F_4}$ is equal to $\pm1$.
    \end{enumerate}
\end{theorem}

The proof of this theorem is representation theoretic, relying on the analysis of the automorphic minimal representation of $\widetilde{F}_4(\A)$, and takes up all of Chapter \ref{Chapt: minimal aut}.  We defer the discussion of this representation until then. We do however need the following properties of $\Theta_{F_4}$, which follow from the minimality of $\Pi_{min}$.

To simplify notation, set $\Theta$ for the adelic modular form such that $\Theta|_{\widetilde{F}_4(\R)} = \Theta_{F_4}$. The automorphic function $\Theta$ has Fourier expansion
\begin{align*}\Theta_Z(g) &= \Theta_N(g) + \sum_{\omega\in W_J(\Q)}{\Theta_\omega(g)}.
\end{align*}

Here, for $g\in \widetilde{F}_4(\A)$, we have
\[
{\Theta_\omega(g)} = \int_{N_J(\Q)\backslash N_J(\A)} \Theta(ng) \psi^{-1}({\la\omega, \overline{n}\ra})dn.
\]
Recall the notion of \emph{rank} of an element $\omega\in W_J(\Q)$ as defined in \cite[Definition 4.2.9 and Definition 4.3.2]{pollackLL}.

\begin{lemma}\label{lem:FElem} Let the notation be as above.
	\begin{enumerate}
		\item If $\gamma \in H^1_J(\Q)$, then $\Theta_\omega(\gamma g) = \Theta_{\omega \cdot \gamma}(g)$.  If $\gamma \in \Ga_{F_4}(4) \cap H^1_J(\R)$, and $g = g_\infty$ is in the image of $\widetilde{F}_4(\R) \rightarrow \widetilde{F}_4(\A)$, then $\Theta_\omega(s_\Gamma(\gamma) g) = \Theta_{\omega \cdot \gamma}(g)$.
		\item One has $\Theta_\omega\equiv0$ unless $\mathrm{rk}(\omega)\leq 1$.
		\item Suppose $g = g_\infty$ is in the image of $\widetilde{F}_4(\R) \rightarrow \widetilde{F}_4(\A)$ and $\omega$ is of rank one. Then $\Theta_\omega(g) \equiv 0$ unless $\omega$ lies in the lattice $W_J(\Z) = \Z \oplus J_0 \oplus J_0 \oplus \Z$.
	\end{enumerate}
\end{lemma}
\begin{proof}
	The first part of the first claim follows easily from the usual change of variables in the integral defining $\Theta_\omega$.  For the second part of the first claim, we have
	\[\Theta_\omega(s_{\Gamma}(\gamma)g) = \Theta_\omega(s_{\Gamma}(\gamma)g s_f(\gamma)) = \Theta_\omega(s_{\Q}(\gamma)g) = \Theta_{\omega \cdot \gamma}(g)\]
	using that $\Theta$ is right invariant under $s_f(\Ga_{F_4}(4))$.
	
	The second claim follows from the construction of $\Theta$ from $\Pi_{min}$ in Chapter \ref{Chapt: minimal aut} and the minimality of $\Pi_{min}$. More specifically, the claim follows directly from Proposition 3 of \cite{ginzburgF4}.
	
 For the final claim, let $W_J(\Z)^\vee$ be the dual lattice to $W_J(\Z)$ under the symplectic form, so that $W_J(\Z)^\vee = \Z \oplus J_0^\vee \oplus J_0^\vee \oplus \Z$.  We first prove that $\Theta_\omega(g)$ vanishes unless $\omega$ is in $W_J(\Z)^\vee$.  To see this, suppose $n_0 \in W_J(\widehat{\Z}) = W_J(\Z) \otimes \widehat{\Z}$ and $n = \exp(n_0) \in \widetilde{F}_4(\A_f) \rightarrow \widetilde{F}_4(\A)$.  Then $n \in K_R(4) \prod_p{K_p}$, so $\Theta$ is right-invariant by $n$.  But then
	\[\Theta_\omega(g) = \Theta_\omega(gn) = \psi(\langle \omega, n_0 \rangle) \Theta_\omega(g).\]
	Consequently, if $\Theta_\omega(g) \neq 0$, then $\langle \omega, n_0 \rangle \in \widehat{\Z}$ for all $n_0 \in W_J(\widehat{\Z})$, so $\omega \in W_J(\Z)^\vee$.
	
	For the stronger claim that $\Theta_\omega(g)$ vanishes unless $\omega \in W_J(\Z) \subseteq W_J(\Z)^\vee$, we use the following lemma. \end{proof}

\begin{lemma} If $\omega \in W_J(\Z)^\vee$ is of rank one, then $\omega \in W_J(\Z)$. \end{lemma}
\begin{proof} Write $\omega = (a,b,c,d)$.  Then $b^\# = ac \in J_0^\vee$ and $c^\# = db \in J_0^\vee$ by \cite[Proposition 11.2]{ganSavin}.  But an elementary check shows that if $X \in J_0^\vee$ and $X^\# \in J_0^\vee$ then in fact $X  \in J_0$.  The lemma follows. \end{proof}

\section{Pullback to $G_2$}
We have defined an inclusion $\widetilde{G}_2(\R) \subseteq \widetilde{F}_4(\R)$ and a modular form $\Theta_{F_4}$ on the latter group.  Let $\Theta_{G_2}$ be the automorphic function that is the pullback of $\Theta_{F_4}$ to $\widetilde{G}_2(\R)$, which is evidently smooth of moderate growth and satisfies the equivariance property \eqref{equivar}. In fact, it also satisfies the requisite differential equation.

\begin{proposition} The automorphic function $\Theta_{G_2}$ is a weight $\frac{1}{2}$ quaternionic modular form on $\widetilde{G}_2(\R).$
\end{proposition}
\begin{proof} This follows just as in \cite[Corollary 4.2.3]{pollackCuspidal}.
\end{proof}

 In this section, we partially compute the Fourier expansion of $\Theta_{G_2}$. For $g\in \widetilde{F}_4(\R)$ we have
\begin{align*}\Theta_Z(g)= \Theta_N(g) + \sum_{\substack{\omega\in W_J(\Z)\\\mathrm{rk}(\omega)=1}}{a(\omega;\alpha_{2\pi\omega}) \W_{2\pi\omega;\alpha_{2\pi\omega}}(g)}\end{align*}
with $a(\omega;-\alpha_{2\pi\omega}) = -a(\omega;\alpha_{2\pi\omega}).$

Suppose $\gamma \in \Ga_{F_4}(4) \cap H_J^1(\R)$.  Define $\alpha_{2\pi\omega}^\gamma(g) = \alpha_{2\pi\omega}(\gamma g)$.  Note that
\[\alpha_{2\pi\omega}^\gamma(g)^2 = 2\pi \langle \omega, \gamma g \cdot r_0(i) \rangle = 2\pi \langle \omega \cdot \gamma, g \cdot r_0(i)\rangle,\]
so that $\alpha_{2\pi\omega}^\gamma$ is an $\alpha_{2\pi\omega \cdot \gamma}$, and $\W_{2\pi\omega;\alpha_{2\pi\omega}}(\gamma g) = \W_{2\pi\omega \cdot \gamma,\alpha_{2\pi \omega}^\gamma}(g)$.
\begin{lemma}  For $\gamma \in \Ga_{F_4}(4) \cap H_J^1(\R)$, one has an equality of Fourier coefficients $a(\omega;\alpha_{2\pi\omega}) = a(\omega \cdot \gamma; \alpha_{2\pi\omega}^\gamma)$.\end{lemma}
\begin{proof} By Lemma \ref{lem:FElem}, one has $\Theta_\omega(\gamma g) = \Theta_{\omega \cdot \gamma}(g)$.  Thus
	\begin{align*} a(\omega;\alpha_{2\pi \omega})\W_{2\pi\omega;\alpha_{2\pi \omega}}(\gamma g) &= \Theta_\omega(\gamma g) = \Theta_{\omega \cdot \gamma}(g) = a(\omega \cdot \gamma; \alpha_{2\pi\omega}^\gamma) \W_{2\pi\omega\cdot \gamma;\alpha_{2\pi\omega}^\gamma}(g) \\ &= a(\omega \cdot \gamma;\alpha_{2\pi\omega}^\gamma) \W_{2\pi\omega;\alpha_{2\pi\omega}}(\gamma g).\end{align*}
	Consequently, $a(\omega;\alpha_{2\pi\omega}) = a(\omega \cdot \gamma; \alpha_{2\pi\omega}^\gamma)$.
\end{proof}

We now consider the Fourier coefficients of $\Theta_{G_2}=\Theta_{F_4}|_{\G_2(\R)}$. We require the following two lemmas. Recall that the Fourier coefficients of a modular form on $G_2$ are parameterized by elements of $W_{\Q}(\Q)$, which may be thought of as $\mathrm{Sym}^3(\Q^2)$ by sending
\[
(r,s,t,z)\in W_{\Q}(\Q)\longmapsto r u^3 + 3s u^2v + 3t uv^2 + z v^3\in \mathrm{Sym}^3(\Q^2).
\] If $\omega = (a,b,c,d) \in W_J(\Q)$, set $\tr(\omega) = \left(a,\frac{\tr(b)}{3},\frac{\tr(c)}{3},d\right) \in \mathrm{Sym}^3(\Q^2)$, so that $\tr(\omega)$ corresponds to the binary cubic form $a u^3 + \tr(b) u^2v + \tr(c) uv^2 + d v^3$.  Now, for each $\omega' \in \mathrm{Sym}^3(\Q^2)$, fix a choice of $\alpha_{2\pi\omega'}(g)$.  Note that for $\omega\in W_J(\Q)$ the restriction of $\alpha_{2\pi\omega}(g)$ to the Heisenberg Levi in $\widetilde{G}_2(\R)\subset \widetilde{F}_4(\R)$, is of the form $\epsilon(\omega;\tr(\omega))\alpha_{2\pi\tr(\omega)}(g)$ where $\epsilon(\omega;\tr(\omega)) \in\{ \pm 1\}$.

\begin{lemma}\label{lem:FErestrict} Suppose $\varphi$ is a modular form on $\widetilde{F}_4(\R)$ of weight $\frac{\ell}{2}$, with Fourier expansion $\varphi_Z(g) = \varphi_N(g) + \sum_{\omega \in W_J(\Q)}{a(\omega;\alpha_{2\pi\omega}) \W_{2\pi\omega;\alpha_{2\pi\omega}}(g)}$.  Let $\varphi'$ be the restriction of $\varphi$ to $\widetilde{G}_2(\R)$.  Then $\varphi'$ is modular form on $\widetilde{G}_2(\R)$ of weight $\frac{\ell}{2}$, with Fourier expansion \[\varphi'_{Z'}(g) = \varphi'_{N'}(g) + \sum_{\omega' \in \mathrm{Sym}^3(\Q^2)}{b(\omega';\alpha_{2\pi\omega'}) \W_{2\pi\omega';\alpha_{2\pi\omega'}}(g)},\]
	where $N'\subset G_2$ denotes the unipotent radical of the Heisenberg parabolic. The Fourier coefficients $b(\omega';\alpha_{2\pi\omega'})$ are given as follows:
	\[b(\omega';\alpha_{2\pi\omega'}) = \sum_{\omega \in W_J(\Q): \tr(\omega) = \omega'}{\epsilon(\omega;\omega') a(\omega;\alpha_{2\pi\omega})}.\]
	The sum, a priori infinite, is in fact finite. \end{lemma}
\begin{proof} The point is that one can simply restrict the Fourier expansion of $\varphi$ to $\widetilde{G}_2(\R)$ to obtain the Fourier expansion of $\varphi'$.  In more detail, one checks that when the function $\W_{\omega,\alpha_{2\pi\omega}}$ on $\widetilde{F}_4(\R)$ is restricted to $\widetilde{G}_2(\R)$, one obtains the function $\epsilon(\omega;\tr(\omega))\W_{2\pi\tr(\omega);\alpha_{2\pi\tr(\omega)}}$ on $\widetilde{G}_2(\R)$.  We omit the proof of the finiteness claim, as we do not really need it, but we note that it follows from the vanishing of the Fourier coefficients that are not positive semidefinite, and that a similar argument can be found in \cite[Section 5.1]{pollackCuspidal}. \end{proof}

In particular, if we can control the signs $\epsilon(\omega;\omega')$, we can use our knowledge of the Fourier expansion of $\Theta_{F_4}$ to obtain information about the Fourier expansion of $\Theta_{G_2}$.  The following lemma controls the signs $\epsilon(\omega;\omega')$.

Below, for $T \in J_0$, we set $\overline{n}(T) = \exp(\delta_2 \otimes T)$, which are unipotent elements of $H_J^1 \subseteq F_4$.
\begin{lemma}\label{lem:notopp2} Suppose $\gamma_1 = \overline{n}(T_1)$ and $\gamma_2 = \overline{n}(T_2)$ are such that $\det(T_1 t +1) = \det(T_2 t + 1)$.  Then $\alpha_{2\pi(0,0,0,1)}^{\gamma_1}$ and $\alpha_{2\pi(0,0,0,1)}^{\gamma_2}$ have equal (as opposed to opposite) restrictions on $\widetilde{G}_2(\R).$
\end{lemma}
\begin{proof} We have $\alpha_{2\pi\omega}(g) = \sqrt{-1} j_{1/2}(g,x_0) p_{2\pi\omega}(gi)^{1/2}$ for a fixed squareroot of $p_{2\pi \omega}(Z)$.  Thus
	\[\alpha_{2\pi(0,0,0,1)}^{\gamma_i}(1)= \alpha_{2\pi(0,0,0,1)}(\gamma_i) = \sqrt{-1}j_{1/2}(\overline{n}(T_i),x_0) p_{2\pi(0,0,0,1)}(\gamma_i \cdot i)^{1/2}.\]
	Note that $p_{2\pi(0,0,0,1)}(Z)^{1/2}$ is constant.  We thus must analyze $j_{1/2}(\overline{n}(T_i),x_0)$.  But now note that there is a unique splitting $\overline{n}(J_3(\R))\rightarrow \widetilde{F}_4(\R)$, this splitting is continuous, and by Lemma \ref{prop:sGammaUnip}, this continuous splitting agrees with the splitting over $\Gamma_{F_4}(4)$. Consequently $j_{1/2}(\overline{n}(T),x_0)$ is a continuous function of $T \in J_3(\R)$, and thus a fixed squareroot of $\det(T i +1)$. Now, by Lemma \ref{lem:SO3trans} proved below, there is a path of $g_t \in \SO_3(\R)$ (which is connected) connecting $T_1$ to $T_2$. Thus $\det(T_1i +1)^{1/2}$ varies continuously to $\det(T_2 i +1)^{1/2}$ via $\det(g_t T_1 g_t^t i + 1)^{1/2}$.  But $\det(g_t T_1 g_t^t i + 1) = \det(T_1 i +1)$ because $g_t \in \SO_3(\R)$.  The lemma follows.
\end{proof}

To describe the Fourier coefficients of $\Theta_{G_2}$, we require the following definition.
\begin{definition} Recall that $J_0 :=S^2(\Z^3)=H_3(\Z)$ denotes the symmetric $3 \times 3$ matrices with integer entries.  If $X \in J_0$, then $\det(tI +X)$ is a monic cubic polynomial with integer coefficients.  For a cubic monic polynomial $p$ with integer coefficients, let
	\[Q_{p} := \{X \in J_0: \det(tI + X) = p(t)\}\]
be the set of $X$ in $J_0$ with $\det(tI + X) = p(t)$.
\end{definition}
The set $Q_p$ is finite, and can only be nonempty when $p(t)$ has three real roots.  In fact, it can be empty even when $p(t)$ has three real roots.

We now assume that $\Theta_{F_4}$ is normalized so that its $(0,0,0,1)$-Fourier coefficient is $\pm1$.  Putting everything together, we have the following result computing a family of Fourier coefficients of $\Theta_{G_2}$.
\begin{theorem}\label{thm:G2main} The pullback $\Theta_{G_2}$ of $\Theta_{F_4}$ to $\widetilde{G}_2(\R)$ has the following Fourier coefficients: If $a,b,c$ are integers and $p(u,v) = au^3 +bu^2v + cuv^2 + v^3$, then the $p(u,v)$ Fourier coefficient of $\Theta_{G_2}$ is $\pm |Q_{p(1,t)}|$.\end{theorem}
\begin{proof}By Lemma \ref{lem:notopp2} and Lemma \ref{lem:FErestrict}, the Fourier coefficient of $\Theta_{G_2}$ corresponding to $p(u,v)$ is the sum of the Fourier coefficients of $\Theta_{F_4}$ corresponding to elements $(\det(T),T^\#,T,1)$ in $W_J$ with $T \in J_0$ and $\det(t1 + T)=p(1,t)$.  Thus the desired Fourier coefficient of $\Theta_{G_2}$ is  given by a sign times the number of $T' \in J_0$ with $\det(tI + T') = p(1,t)$.  This is $|Q_{p(1,t)}|$, as claimed.
\end{proof}

\section{Arithmetic invariant theory}\label{sec:AIT}
The purpose of this section is to do some arithmetic invariant theory related to the set $Q_p$.  In particular, if $R = \Z[t]/(p(t))$, then we relate $Q_p$ to the sets $Q_R$ defined as follows. Set $E = R \otimes \Q$ and assume that $p(t)$ is such that $E$ is an \'etale $\Q$-algebra.  If $I$ is a fractional ideal of $R$ and $\mu \in E^\times$ is totally positive, again as before say that $(I,\mu)$ is \textbf{balanced} if
\begin{itemize}
	\item $\mu I^2 \subseteq \mathfrak{d}_R^{-1}$
	\item $N(\mu) N(I)^2 \mathrm{disc}(R/\Z) = 1$.
\end{itemize}
Note that this all makes sense, regardless of if $E$ is a field.  One puts on pairs $(I,\mu)$ an equivalence relation: $(I,\mu) \sim (\beta I, \beta^{-2} \mu)$ for $\beta \in E^\times$ and lets $Q_R$ denote the set of equivalence classes.

\subsection{The case of a field}
Let $R = \Z[t]/(p(t))$ be a monogenic order in a totally real cubic field $E = R \otimes \Q$.  Observe that the group $\SO_3(\Z)$ acts on the set $Q_p$ by $X \mapsto gXg^t$.
\begin{lemma} Suppose $T \in J_0$ has $\det(tI + T) = p(t)$. Then $\SO_3(\Z)$ acts freely on $T$, i.e., if $g \in \SO_3(\Z)$ and $gTg^t = T$, then $g =1$.
\end{lemma}
\begin{proof}  Suppose $g \in \SO_3(\Z)$, and $T = gTg^t = gTg^{-1}$.  Then $g$ commutes with $T$, so $g \in \Q[T]$.  It follows that $g$ is symmetric, so $1 = gg^t = g^2$.  Thus $g \in \mu_2(\Q[T])$.  But $\Q[T]$ is a field by assumption, so $g = \pm 1$.  Because $\det(g) =1$, $g =1$, proving the lemma.
	
	Note that the lemma is false if we do not assume $R \otimes \Q$ is a field.
\end{proof}

The following lemma is well-known.
\begin{lemma}\label{lem:onbasis} Suppose $M=\Z^3$ has a symmetric bilinear form on it $(\,,\,)$ which is integral, i.e., $(v,w) \in \Z$ for all $v,w \in M$.  Suppose moreover that the bilinear form $(\,,\,)$ is positive-definite and of determinant one, i.e. $\det((v_i,v_j))=1$ for a basis $v_1,v_2,v_3$ of $M$ over $\Z$.  Then $M$ has an orthonormal basis $v_1',v_2',v_3'$.\end{lemma}

Here is the main result of this section.
\begin{proposition}\label{prop:AITmain} Suppose $R = \Z[t]/(p(t))$ is an order in a totally real cubic field $E = R \otimes \Q$.  Then there is a bijection (to be given in the proof) between the sets $Q_R$ and $\SO_3(\Z)\backslash Q_p$.  In particular, $|Q_p| = |\SO_3(\Z)| \cdot |Q_R|=24|Q_R|$. \end{proposition}
As mentioned in the introduction, this proposition essentially follows from the work in \cite{swaminathan2021average}.  Because \cite{swaminathan2021average} is much more general, we give a direct proof of this simple case that we need.
\begin{proof} Let $\omega$ be the image of $t$ in $R = \Z[t]/(p(t))$.  Associated to a $T \in J_0$ with $\det(tI + T) = p(t)$, we obtain a module $M = \Z^3$, together with a unimodular quadratic form $(\,,\,)$ and orthonormal basis $e_1, e_2, e_3$.  The element $T$ defines an action of $R$ on $M$, via $\omega m = -Tm$.  Because $T$ is symmetric, this action is symmetric for the bilinear form: $(v,\lambda w) = (\lambda v, w)$ for all $v,w \in M$ and $\lambda \in R$.
	
	We can think of $M$ as a fractional ideal $I$ of $E := R \otimes \Q$. That is, $I = \Z e_1 + \Z e_2 + \Z e_3$ with $e_1,e_2,e_3 \in E$ such that $-\omega e_i = \sum_{j}{T_{ij} e_j}$. Moreover, because the action of $R$ is symmetric, the bilinear form on $I$ is of the form $(v,w) = \tr(\mu v w)$ for some fixed $\mu \in E^\times$.  Because the bilinear form is positive definite and because $E$ is totally real, $\mu$ must be totally positive.  We thus obtain a pair $(I,\mu)$.  The choice of $I$ is well-defined up to scalar multiple.  We claim that the pair $(I,\mu)$ is balanced.  To see this, first note that because our form $(v,w) = \tr(\mu v w)$ is integral on $I$, and $I$ is a fractional ideal, we have $\mu I^2 \subseteq \mathfrak{d}_R^{-1}$.  Now, one checks easily that $\det((\tr(\mu v_i v_j))) = N(\mu) \det((\tr(v_i v_j)))$.  Thus
	\[1 = \det((e_i,e_j)) = N(\mu)\det(\tr(e_i e_j)) = N(\mu) N(I)^2 \mathrm{disc}(R/\Z).\]
	
	Thus, out of $T \in Q_p$, we have constructed a class $[I,\mu]$ in $Q_R$.  Tracing through the maps, one sees that $[I,\mu]$ is well-defined.  Moreover, if $g \in \SO_3(\Z)$, then $g \cdot T$ maps to the same pair $[I,\mu]$, because the action of $g$ just changes the basis $e_1,e_2,e_3$ of $I$.
	
	In the reverse direction, suppose given a pair $(I,\mu)$ with $(I,\mu)$ balanced.  Then the pairing $(v,w) = \tr(\mu vw)$ on $I$ is integral.  Moreover, if $v_1,v_2,v_3$ is an integral basis of $I$, then $\det((v_i,v_j)) = \det(\tr(\mu v_iv_j)) = N(\mu)N(I)^2 \mathrm{disc}(R/\Z) =1.$ By Lemma \ref{lem:onbasis}, $I$ has an orthonormal basis $e_1,e_2,e_3$.  We thus obtain $T := -((e_i,\omega e_j))_{ij}$ with $\det(tI+T) = p(t)$.  The basis $e_1,e_2,e_3$ is well-defined up to the action of $O_3(\Z) = \{\pm 1\} \times \SO_3(\Z)$ so the element $T$ is well-defined in the orbit space $\SO_3(\Z)\backslash Q_p$.
	
	The maps described above are inverse bijections. Noting that $|\SO_3(\Z)|=24$, the proposition follows.
\end{proof}

The following lemma was used above.
\begin{lemma}\label{lem:SO3trans} The group $\SO_3(\R)$ acts transitively on the set of $T \in J_0\otimes \R$ with fixed characteristic polynomial $p(t)$.\end{lemma}
\begin{proof}  Because $O_3(\R) = \{\pm 1\} \times \SO_3(\R)$, it suffices to see that $O_3(\R)$ acts transitively.  But now, every real symmetric matrix can be diagonalized by an element of $O_3(\R)$.  Using the action of the symmetric group $S_3 \subseteq O_3(\R)$ finishes the proof. \end{proof}

We end this section by discussing the set $Q_R$ when $R$ is a maximal order in $E$.
\begin{proposition}\label{prop:AITCl}\label{prop:QRCl2} Suppose $R$ is the maximal order in $E$.  Then if $Q_R$ is non-empty, $|Q_R| = |\mathrm{Cl}_{E}^{+}[2]|$, the size of the two-torsion in the narrow class group of $E$.
\end{proposition}

To prove the proposition, we will use the following lemma.  Consider the group $A_R$ of equivalence classes of pairs $(J,\lambda)$ with $\lambda J^2 =(1)$, $J$ a fractional $E$-ideal and $\lambda$ totally positive.  That is, $(J,\lambda)$ is equivalent to $(J',\lambda')$ if there exists $\mu \in E^\times$ so that $J' = \mu J$ and $\lambda' = \mu^{-2} \lambda$.  It is clear that $Q_R$, when non-empty, is a torsor for $A_R$.  Let $A_R'$ denote the set of such pairs $(J,\lambda)$ except modulo the equivalence relation $(J,\lambda)$ is equivalent to $(J',\lambda')$ if there exists $\mu \in E_{>0}^\times$ so that $J' = \mu J$ and $\lambda' = \mu^{-2} \lambda$.
\begin{lemma}\label{lem:ESA'} One has the following exact sequences:
	\begin{equation}\label{eqn:ESA'1} 1 \rightarrow R_{>0}^\times/(R_{>0}^\times)^2 \rightarrow A_R' \rightarrow \mathrm{Cl}^{+}_E[2] \rightarrow 1,\end{equation}
	and
	\begin{equation}\label{eqn:ESA'2} 1 \rightarrow E^\times/\left(\pm E^\times_{>0}\right) \rightarrow A_R' \rightarrow A_R \rightarrow 1. \end{equation}
\end{lemma}
\begin{proof} We first consider the sequence \eqref{eqn:ESA'1}.  The map $A_R' \rightarrow \mathrm{Cl}^+_E$ is given by sending $[J,\lambda]$ to $[J] \in \mathrm{Cl}_{E}^{+}$.  Because $[J^2] = (\lambda^{-1})$ with $\lambda$ totally positive, $[J] \in \mathrm{Cl}_{E}^+[2]$.  It is clear that this map is surjective.
	
	For the kernel, if $[J] =1$ in $\mathrm{Cl}_E^+$, then $J = (\epsilon)$ with $\epsilon$ totally positive.  Consider $\lambda \epsilon^2$.  This is in $R^\times_{>0}$.  The element $\epsilon$ is well-defined up to multiplication by an $\epsilon_1 \in R^\times_{>0}$, so $\lambda \epsilon^2$ gives a well-defined class in $R^\times_{>0}/(R_{>0}^\times)^2$.  It is checked immediately that this map gives an isomorphism of the kernel of $\{A_R' \rightarrow \mathrm{Cl}_{E}^+[2]\}$ with $R^\times_{>0}/(R_{>0}^\times)^2$.
	
	Now consider the sequence \eqref{eqn:ESA'2}.  The map $A_R' \rightarrow A_R$ is dividing out by the courser equivalence relation.  The kernel of this map is the image in $A_R'$ of the set of pairs $((\mu),\mu^2)$ with $\mu \in E^\times$.  This is trivial in $A_R'$ precisely when there exists $\mu' \in E_{>0}^\times$ so that $((\mu),\mu^2) = ((\mu'),\mu'^2)$, which happens precisely if $\mu \in \pm E^\times_{>0}$.  The lemma follows.
\end{proof}

Proposition \ref{prop:QRCl2} follows from Lemma \ref{lem:ESA'} by observing that both $R_{>0}^\times/(R_{>0}^\times)^2$ and $E^\times/\left(\pm E^\times_{>0}\right)$ have size $4$.  Finally, again assuming that $R$ is the maximal order in $E$, we remark that it follows from \cite[Proposition 3.1]{grossSigns} that $Q_R$ is non-empty if and only if every quadratic extension of $E$ that is unramified at all finite primes is totally real.  Combining Proposition \ref{prop:AITmain} with Theorem \ref{thm:G2main} gives Theorem \ref{thm:introMain2}.  Combining the result with Proposition \ref{prop:AITCl} gives Theorem \ref{thm:introMain}.

\subsection{The general case}\label{sec:gencase} In the previous subsection, we discussed the arithmetic invariant theory of the sets $Q_p$ when $E = R \otimes \Q$ is a field.  We now make some remarks about the arithmetic invariant theory of the sets $Q_p$ when $E$ is just an \'etale cubic $\Q$-algebra.  We omit the proofs, as they are simple generalizations of the proofs in the previous subsection.

Recall that if $p(t) \in \Z[t]$ is cubic and monic then $Q_p$ denotes the set of $T \in J_0 = Sym^2(\Z^3)$ such that $\det(t1_3+T) = p(t)$.

One has the following bijection.
\begin{proposition}  There is a bijection between equivalence classes of balanced pairs $Q_R$ and the $O_3(\Z)$ (or, equivalently $\SO_3(\Z)$) orbits on $Q_p$.  Moreover, the stabilizer of $T \in Q_p$ under the action of $O_3(\Z)$ is $\mu_2(\mathcal{O}_I)$ where
	\[\mathcal{O}_I = \{\alpha \in E: \alpha I \subseteq I\}.\]
\end{proposition}
As a consequence of the proposition, one obtains:
\[\# Q_p = \sum_{[(I,\mu)] \text{ balanced}}{\frac{\# O_3(\Z)}{\mu_2(\mathcal{O}_I)}}.\]
In particular, if $R$ is maximal so that $\mathcal{O}_I = R$ for all $I$, then
\[\# Q_p = \frac{48}{\mu_2(R)} \times \#\{[(I,\mu)] \text{ balanced}\}.\]

In this maximal case, assuming that $E$ is \'{e}tale, one has that $(I,\mu)$ is balanced precisely if $\mu I^2 = \mathfrak{d}_R^{-1}$.  Now one can consider the exact sequences as in Lemma \ref{lem:ESA'}, which become:
\[1 \rightarrow R^\times_{>0}/(R^\times_{>0})^2 \rightarrow A_R' \rightarrow \mathrm{Cl}_E^{+}[2] \rightarrow 1\]
and
\[1 \rightarrow E^\times/(\mu_2(E) E^\times_{>0}) \rightarrow A_R' \rightarrow A_R \rightarrow 1.\]
Considering the different cases separately, one sees that in all \'etale maximal cases, $\# A_R = \# \mathrm{Cl}_{E}^{+}[2]$. Thus, if $R$ is maximal and $E$ is \'etale, one has the formula
\[\# Q_p = \frac{48}{\mu_2(R)} |\mathrm{Cl}_E^{+}[2]| \times \delta_R\]
where $\delta_R$ is $0$ if the inverse different $\mathfrak{d}_R^{-1}$ is not a square in $\mathrm{Cl}_{E}^+$ and $1$ if it is such a square.  We state this as a proposition.
\begin{proposition} Let the notation be as above, and assume that $R=\Z[t]/(p(t))$ is the maximal order in $E=R\otimes\Q$, which is assumed \'etale. Then $\# A_R = \# \mathrm{Cl}_{E}^{+}[2]$.  Consequently, $\# Q_p = \frac{48}{\mu_2(R)} |\mathrm{Cl}_E^{+}[2]| \times \delta_R$ where $\delta_R$ is $0$ if the inverse different $\mathfrak{d}_R^{-1}$ is not a square in $\mathrm{Cl}_{E}^+$ and $1$ if it is such a square.
\end{proposition}

Note that if $R = \Z \times \mathcal{O}_K$ with $K$ real quadratic, then $\mathrm{Cl}_{E}^{+} = \mathrm{Cl}_{K}^+$.  For the sake of completeness, we now answer the question of when the maximal order in such a case is monogenic.
\begin{proposition} Set $R = \Z \times \mathcal{O}_K$ with $K$ a real quadratic field.
	\begin{enumerate}
		\item If $\ell$ is squarefree and $\mathcal{O}_K = \Z[\sqrt{\ell}]$, then $R$ is monogenic if and only if $\ell = r^2 \pm 1$ for some $r$ in $\Z$.  In this case, $(r,\sqrt{\ell})$ is a generator of $R$.
		\item If $\mathcal{O}_K = \Z[\omega]$ with $\omega = \frac{1+\sqrt{4\ell+1}}{2}$, then  $R$ is monogenic if and only if the equation $r(r-1) = \ell \pm 1$ has a solution, in which case $(r,\omega)$ is a generator.
	\end{enumerate}
	\end{proposition}

\subsection{Table of data}
We now present a table of numerical data for the Fourier coefficients $|Q_p|$ of $\Theta_{G_2}$.  The rings $R$ were checked to be maximal (monogenic) orders by the L-function and Modular Form Database (LMFDB) \cite{lmfdb}.  The computer algebra system SAGE \cite{sagemath} was used to compute the narrow class groups $\mathrm{Cl}_E^{+}$. In the table, the notation $C_n$ denotes the cyclic group of order $n$.
\vskip 15pt
\begin{tabular}{|c|c|c|c|c|}
	\hline
	$p(t)$ &  structure & maximal monogenic (LMFDB) & $\# Q_p$ & $\mathrm{Cl}_E^+$ (SAGE) \\ \hline
	$t^3-t^2-2t+1$ & cubic field & yes & 24 &  1 \\ \hline
	$t^3-3t-1$ & cubic field&  yes & 24 & 1 \\ \hline
	$t^3-t^2-3t+1$ & cubic field& yes & 24& 1\\ \hline
	$t^3-t^2-9t+10$ & cubic field& yes & 48 & $C_4$ \\ \hline
	$t^3-t^2-14t+23$ & cubic field& yes &48 & $C_4$ \\ \hline
	$t^3-t^2-11t+12$ & cubic field & yes & 48 & $C_4$ \\ \hline
	$t^3-t^2-12t-1$ &cubic field& yes & 48 & $C_4$\\ \hline
	$t^3-5t-1$ & cubic field& yes & 24 & 1 \\ \hline
	$t^3-t^2-9t+8$ & cubic field & yes & 0 & $C_6$ \\ \hline
	$t^3-21t-35$ & cubic field & yes & 24 & $C_3$\\ \hline
	$(t-1)(t^2-2) $ & quadratic & yes & 12 & 1\\ \hline
	$(t-2)(t^2-3) $ & quadratic & yes & 0 & $C_2$ \\ \hline
	$(t-3)(t^2-10)$ & quadratic & yes & 24 & $C_2$\\ \hline
	$t^3-t^2-54t+169$ & cubic field & yes & 96 & $C_2 \times C_2$ \\ \hline
	$t^3-t^2-34t-57$ & cubic field & yes & 96 & $C_4 \times C_2$ \\ \hline
\end{tabular}

\chapter{The automorphic minimal representation}\label{Chapt: minimal aut}
In this chapter, we construct and study the modular form $\Theta_{F_4}$ of weight $\frac{1}{2}$ on the double cover of $F_4$ and prove Theorem \ref{thm: exists theta} via a careful analysis of the automorphic minimal representation of $\widetilde{F}_4(\A)$.

\section{Review of the construction}
 We begin by reviewing the construction of the automorphic minimal representation $\Pi_{min}$ on $\widetilde{F}_4(\A)$, following Loke--Savin \cite{lokeSavin}, and then Ginzburg \cite{ginzburgF4}.

Recall that we have ordered the simple roots of $F_4$ in the usual way, so that the Dynkin diagram
\[\circ---\circ=>=\circ---\circ\]
has labels $\alpha_1$ through $\alpha_4$ from left to right.  Define $m_{\alpha_1} = m_{\alpha_2} = 2$ and $m_{\alpha_3} = m_{\alpha_4} = 1$. Let $p$ be a place of $\Q$, allowing $p=\infty$. We begin with the following lemma.
\begin{lemma}\label{lem:Tcenter} Let $\widetilde{T}(\Q_p)$ denote the inverse image of the fixed split maximal torus of $F_4(\Q_p)$ in $\widetilde{F}_4(\Q_p)$, and $Z(\widetilde{T}(\Q_p))$ its center.  Then $t\in Z(\widetilde{T}(\Q_p))$ if and only if $t = \pm \prod_{i}\widetilde{h_{\alpha_i}}(t_i^{m_i})$.\end{lemma}
\begin{proof} One applies the commutator formula \eqref{eqn: commutator of tori} $\{\widetilde{h_{\alpha}}(s),\widetilde{h_{\beta}}(t)\} = (s,t)^{(\alpha^\vee,\beta^\vee)}$.
\end{proof}
We will also have need of a maximal abelian subgroup at every local place. This is handled uniformly by the following lemma.
\begin{lemma} For any place $p\leq\infty$, the subgroup
	\[T_{*}(\Q_p):=\pm \widetilde{h_{\alpha_1}}(\Q_p^\times)\widetilde{h_{\alpha_2}}((\Q_p^\times)^2)\widetilde{h_{\alpha_3}}(\Q_p^\times) \widetilde{h_{\alpha_4}}(\Q_p^\times)\]
	is a maximal abelian subgroup of $\widetilde{T}(\Q_p)$.\end{lemma}
\begin{proof} This is an easy check, using the commutator formula.\end{proof}
For each $p$, we let $B_\ast(\Q_p) = T_\ast(\Q_p)U_B(\Q_p)$ denote the associated subgroup of $\widetilde{B}(\Q_p)$.
\begin{definition}
	A genuine character $\chi_p$ of $Z(\widetilde{T}(\Q_p))$ is said to be \textbf{exceptional} if for each simple root $\al$, $\chi_p(\widetilde{h_{\alpha}}(t^{m_{\alpha}})) = |t|_v$. We let $\nu_{exc}:= (1/m_\al)_\al\in X^\ast(T)\otimes_\Z \R$ to be the associated exponent.
\end{definition}
Lemma \ref{lem:Tcenter} implies that there is a unique exceptional character $\chi_p$ on the center of the covering torus of $\widetilde{F}_4(\Q_p)$. Let $\chi_{exc} = \prod_p\chi_p$ be the induced character on $Z(\widetilde{T}(\A))$.  Note that $\chi$ is automatically automorphic by the product formula.

We consider the subgroup of $\widetilde{T}(\A)$ given by
\[
T_\ast(\A):=T(\Q) Z(\widetilde{T}(\A));
\]
this is a maximal abelian subgroup \cite[Theorem 4.1]{WeissmanTori}.
Abusing notation, write $\chi_{exc}$ for the automorphic extension of $\chi_{exc}$ from $Z(\widetilde{T}(\A))$ to $T_*(\A)$.  Inflating $\chi_{exc}$ to a character of $B_{*}(\A) := T_{*}(\A)U_B(\A)$, consider the induced representation
\[
V_0 :=\Ind_{B_*(\A)}^{\widetilde{F}_4(\A)}(\delta_B^{1/2} \chi_{exc}),\]
where $\delta_B$ is the modular character of $B(\A)$.

\begin{remark}
	In their construction of this representation, Loke and Savin instead define a representation $\pi(\chi_{exc})$ of $\widetilde{T}(\A)$, inflate to $\widetilde{B}(\A)$, then induce to $\widetilde{F}_4(\A)$.  It follows from \cite[Proposition 4.1 and Proposition 5.3]{lokeSavin} that their $\pi(\chi_{exc})$ is an irreducible representation of $\widetilde{T}(\A)$ with the same central character as $\Ind_{T_*(\A)}^{\widetilde{T}(\A)}(\chi_{exc})$, so they are isomorphic.  In fact, both representations are realized as spaces of functions on $T(\Q)\backslash \widetilde{T}(\A)$, and we claim that they are identical. This is because there is, in the terminology of \cite{lokeSavin}, a unique genuine representation in $A T(\Q)\backslash \widetilde{T}(\A)$ that is invariant under $M_s T_2^1 \prod_{p >2}{T_p}$; see \cite[Corollary 5.2]{lokeSavin}.  (This is true for $F_4$, but not true in general.)
\end{remark}

For $\mathbf{s} = (s_1,s_2,s_3,s_4)\in \C^4$, define $\omega_{\mathbf{s}}$ a character of $T(\A)$ as $\omega_{\mathbf{s}}(h_{\alpha_i}(t_i)) = |t_i|^{s_i}$.  Set $$V_s = \Ind_{B_*(\A)}^{\widetilde{F}_4(\A)}(\delta_B^{1/2} \chi_{exc} \omega_s).$$  Let $f(g,\mathbf{s})$ be a flat section in this induced representation, and set
\[E(g,f,\mathbf{s}) = \sum_{\gamma \in B(\Q)\backslash F_4(\Q)}{f(\gamma g, \mathbf{s})}.\]
The automorphic minimal representation on $\widetilde{F}_4(\A)$ is constructed as the residue of these Eisenstein series at a distinguished point.

\begin{theorem}\cite[Theorem 7.1]{lokeSavin}\label{thm: residue} The Eisenstein series $E(g,f,\mathbf{s})$ have at worst a simple multi-pole at $\mathbf{s} = 0$.  Let
	\[\theta(g,f) = \lim_{\mathbf{s} \rightarrow 0}{s_1 s_2 s_3 s_4 E(g,f,\mathbf{s})}\]
	and $\Pi_{min}$ be the space of these residues $\theta(g,f)$.  Then $\theta(g,f)$ is a genuine, square-integrable automorphic form on $\widetilde{F}_4(\A)$.  Moreover, the representation $\Pi_{min}$ is irreducible.
\end{theorem}

\begin{remark}
	In \cite{lokeSavin}, this theorem is proved for the associated automorphic representation on the double cover of all split simply-connected semi-simple groups over $\Q$. These are examples of \emph{generalized theta representations}, which play a fundamental role in the study of automorphic representations of non-linear covering groups; see for example \cite{Patterson,ChintaFriedHoff,BFGsmall,FriedbergGinzburg,LeslieTheta} for some conjectures and aspects of this area.
\end{remark}

Write $\Pi_{min} = \otimes'_{p} \Pi_{min,p}$.  Then Loke--Savin also identify the representations $\Pi_{min,p}$ in terms of principal series. To do this, extend the character $\chi_p$ of $Z(\widetilde{T}(\Q_p))$ to the subgroup $B_{*}(\Q_p)$, and let $I_p = Ind_{B_{*}(\Q_p)}^{\widetilde{F}_4(\Q_p)}(\delta_{B}^{1/2} \chi_p)$.
\begin{proposition}\cite[Proposition 6.3]{lokeSavin}  The representation $I_p$ has a unique irreducible quotient, which is $\Pi_{min,p}$.\end{proposition}

The notation $\Pi_{min}$ references Ginzburg's theorem \cite[Theorem 1]{ginzburgF4} that $\Pi_{min}$ is an \textbf{automorphic minimal representation} in the sense that the set of nilpotent elements associated to non-vanishing Fourier--Whittaker coefficients of $\Pi_{min}$ are contained in the minimal nilpotent orbit $\mathcal{O}_{min}\subset \mathfrak{f}_4(\overline{\Q})$; we refer the reader to \cite{ginzburgDimn} for the notions of Fourier--Whittaker coefficients associated to nilpotent orbits. This result plays a central role in our analysis of the Fourier expansion of $\Theta_{F_4}$; see Lemma \ref{lem:FElem}.

\section{Archimedean aspects}
Relating these generalized theta series to quaternionic modular forms requires information of the $\widetilde{K}_\infty$-types of the local representation $\Pi_{min,\infty}$. This representation turns out to be the same as the representation $\Pi_{GW}$ constructed by Gross--Wallach in \cite{grossWallachI}.
\begin{proposition}\label{prop: right Ktypes}
	The representation $\Pi_{min,\infty}$ is isomorphic to the minimal representation $\Pi_{GW}$ constructed by Gross--Wallach; its $\widetilde{K}_\infty=\SU(2)\times \Sp({6})$-types are
	\begin{equation}
	\bigoplus_{n=0}^\infty Sym^{1+n}(\C^2)\boxtimes \mathbf{V}(n\omega_3),
	\end{equation}
	where $\omega_3$ is the $3^{rd}$ fundamental weight of $\Sp(6)$ and $\mathbf{V}(n \omega_3)$ denotes the irreducible representation of $\Sp(6)$ with highest weight $n \omega_3$. In particular, $\Pi_{min,\infty}$ has minimal $\widetilde{K}_\infty$-type $\mathbf{V}_{1/2}$.
\end{proposition}
\begin{proof}
	Note that, from \cite[Proposition 6.3]{lokeSavin}, $\Pi_{min,\infty}$ is the Langlands quotient of the principal series representation
	\[
	\Ind_{B_*(\R)}^{\widetilde{F}_4(\R)}(\delta_B^{1/2} \chi_\infty)\cong  \Ind_{\widetilde{B}}^{G}(\pi(\chi_{\infty})),
	\]
	where $\chi_\infty$ is the exceptional character and $\pi(\chi_{\infty})\cong \widetilde{\de}\boxtimes \chi_\infty$ is the induced representation of $\widetilde{T}(\R)=\widetilde{M} \cdot T(\R)^\circ$.  Here $\widetilde{M}$ is a certain finite subgroup of $\widetilde{T}(\R)$ and $T(\R)^\circ$ is the connected component of the identity of the covering torus. Note we use the fact that
	\begin{equation}\label{eqn: exception inf}
	\nu_{exc}=\left(\frac{1}{2},\frac{1}{2},1,1\right)= \rho-\frac{1}{2}(\omega_1+\omega_2)\in \t^\ast:= X^\ast(T)\otimes_\Z\R
	\end{equation}
	lies in the dominant chamber in identifying $\Pi_{min,\infty}$ as the Langlands quotient.
	
	Referring the reader to \cite[Sections 4 and 5]{ABPTV} for the notions of pseudospherical representations and notation, in the decomposition
	\[
	\pi(\chi_{\infty})=\tilde{\de}\boxtimes \chi_{\infty},
	\]
    the two dimensional representation $\tilde{\de}$ is a \emph{pseudospherical representation} of $\widetilde{M}$. It is easy to check that there is a unique such representation for $\G=\widetilde{F}_4(\R)$, and it arises as the restriction of the $\widetilde{K}_\infty=\SU(2)\times \Sp(6)$-representation $\mathbf{V}_{{1}/{2}}$ to $\widetilde{M}\subset\widetilde{K}_\infty$. In particular, $\mathbf{V}_{{1}/{2}}$ is the unique pseudospherical $\widetilde{K}_\infty$-type for $\G$.

    In the notation of \cite{ABPTV}, we see that $\Pi_{min,\infty}$ is the Langlands quotient $J(\tilde{\de},\nu_{exc})$ of the corresponding pseudospherical principal series
	\[
	I(\tilde{\de},\nu_{exc})=\Ind_{\widetilde{B}(\R)}^{\G}(\tilde{\de}\boxtimes(\nu_{exc}+\rho)).
	\]By \cite[Def. 5.5]{ABPTV} and the subsequent discussion, we conclude that $\Pi_{min,\infty}$ has the minimal $\widetilde{K}_\infty$-type $\mathbf{V}_{1/2}$.	The key point, as noted in \cite[Section 5]{ABPTV}, is that this Langlands quotient $J(\tilde{\de},\nu_{exc})$ is the unique irreducible representation of $\G$ containing the $\widetilde{K}_\infty$-type $\mathbf{V}_{{1}/{2}}$ and having infinitesimal character $\nu_{exc}\in \t^\ast/W$. This follows from the analysis of pseudospherical $\tilde{K}_\infty$-types in \emph{loc. cit.} combined with Harish-Chandra's subquotient theorem.

	On the other hand, Gross and Wallach apply cohomological techniques to construct the minimal representation $\Pi_{GW}$ in \cite{grossWallachII}; here, minimal means the ideal of $\mathfrak{U}(\mathfrak{f}_4(\C))$ annihilating $\Pi_{GW}$ is the Joseph ideal. In particular, they compute that the  $\widetilde{K}_\infty$-types of $\Pi_{GW}$ are precisely the representations occurring in the proposition \cite[Section 12]{grossWallachII}. Furthermore, as an element of $\mathfrak{t}^\ast/W$, the infinitesimal character of $\Pi_{GW}$ is
	\[
	\nu_{GW}:=\rho-\frac{3}{2}\omega_1,
	\]
	where $\omega_1$ is the first fundamental weight of $F_4$ \cite[pg.109]{grossWallachII}. Here $W$ denotes the Weyl group of the pair $(F_4,T).$
	
	To complete the proof, it suffices to check that there exists $w\in W$ such that $w(\nu_{GW}) = \nu_\infty$. Referencing \eqref{eqn: exception inf}, this is equivalent to the existence of $w\in W$ such that
	\[
	w\bullet\left(-\frac{3}{2}\omega_1\right) = -\frac{1}{2}(\omega_1+\omega_2),
	\]
	where $\bullet$ denotes the dot action of the Weyl group of $F_4$ on $\mathfrak{t}^\ast.$ The existence of such an element may be verified via a computer calculation, using SAGE \cite{sagemath} for example. By uniqueness, this proves the proposition.
\end{proof}

\subsection{Modular forms of weight $\frac{1}{2}$}\label{Sec: modular forms}
Using Proposition \ref{prop: right Ktypes}, we can now construct modular forms of weight $1/2$ on $\widetilde{F}_4(\A)$ from $\Pi_{min}$.  Let $x,y$ be our fixed weight basis of $\mathbf{V}_{1/2}=\mathbb{V}_2 \simeq \mathbb{V}_2^\vee$. Setting $\Pi_{min,f} = \otimes'_{p<\infty}\Pi_{min,p}$, fix a vector $v_f \in \Pi_{min,f}$ and let
\[\alpha: \Pi_{min}= \Pi_{min,f} \otimes \Pi_{min,\infty} \rightarrow \mathcal{A}(\widetilde{F}_4(\A))\]
be the automorphic embedding in Theorem \ref{thm: residue}.  Define
\begin{equation}\label{eqn:theta}\theta(v_f) := \alpha(v_f \otimes x) \otimes x^\vee + \alpha(v_f \otimes y) \otimes y^\vee \in \mathcal{A}(\widetilde{F}_4(\A))\otimes \mathbb{V}_2^\vee.\end{equation}

Pulling back $\theta(v_f)$ to $\widetilde{F}_4(\R)$ via the map $\widetilde{F}_4(\R) \rightarrow \widetilde{F}_4(\A)$, one obtains a quaternionic modular form of weight $\frac{1}{2}$ on $\widetilde{F}_4(\R)$. Indeed, smoothness of $v_f$ implies that there exists some compact open subgroup $K_f\subset F_4(\ZZ)$ splitting the cover $s: K_f\to \widetilde{F}_4(\A)$ such that if $\Ga:=F_4(\Q)\cap K_f$, then $\theta(v_f)$ descends to a smooth function on $s_f(\Ga)\backslash \widetilde{F}_4(\R)$ satisfying the $\widetilde{K}_\infty$-equivariance \eqref{equivar}. Moreover, the construction of the Schmid operator $D_{{1/}{2}}$ precisely detects the fact that the automorphic function $\theta(v_f)$ corresponds to the minimal $\widetilde{K}_{\infty}$-type $\mathbb{V}_2$, so that $D_{1/2}\theta(v_f)\equiv 0$ for any $v_f$.

Our goal for the remainder of the chapter is to prove that $v_f$ can be chosen so that $\theta(v_f)$ has $\Ga_{F_4}(4)$ level and nonzero $(0,0,0,1)$-Fourier coefficient, as in Theorem \ref{thm:introMin}.

\section{Weil representations for $\GL_2$}\label{Section: GL2}   To accomplish this goal, we will calculate a certain twisted Jacquet module of $\Pi_{min}$.  For this latter calculation, we make a detour to consider the Weil representation of $\GL_2$.

The main results of this section are Corollaries \ref{cor:L02} and \ref{cor:L0p}, asserting that if certain Whittaker functionals vanish on particular subspaces of these Weil representations, then they vanish identically. For this we need to compare a certain double cover of $\GL_2(\Q_p)$ arising in our context with other constructions in the literature. Strictly speaking, we could appeal to the results of Kazhdan--Patterson \cite[Section 1]{kazhdanPatterson} to see that the representation theory of these various covers of $\GL_2(\Q_p)$ are related as described in Proposition \ref{prop:compare}. We have opted for a more-or-less self-contained presentation for the sake of the reader.

\subsection{The double cover of $\SL_2(\Q_p)$ and its Weil representation}\label{Sec: Weil rep on SL2} Now set $k = \Q_p$ for any prime $p$, though the results of this section hold for any local field. We recall various essentially well-known facts about the group $\widetilde{\SL_2}(k)$ and its Weil representation.

Let $(V,q)$ be a quadratic space over $k$, and $B(x,y) = q(x+y) - q(x) - q(y)$ the associated bilinear form.  We define a representation of $\widetilde{\SL}_2(k)$ on $S(V)$, the Schwartz space of $V$, which is genuine if $\dim(V)$ is odd.

We fix the additive character $\psi$ of $k$.  Fix the Haar measure $dv$ on $V$ that is self-dual with respect to the Fourier transform on $V$ as
\[\widehat{\Phi}(v) = \int_{V}{\psi((v,w))\Phi(w)\,dw}.\]  Define $F_{q}(v) = \psi(q(v))$, and let $\gamma(q) \in \C$ be defined as
\begin{equation}\label{eqn: weil index}
\gamma(q) = \lim_{L\subset V}\int_{L}{F_q(v)\,dv},
\end{equation}
where the limit indicates that the value stabilizes for sufficiently large lattices $L$ in $V$ and we take this value.

One defines a Weil representation of $\widetilde{\SL}_2(k)$ on $S(V)$, via:
\begin{enumerate}
	\item $\displaystyle\zeta\cdot \Phi(v) = (-1)^{dim(V)}\Phi(v)$
	\item $\displaystyle x_\al(t)\cdot \Phi(v) = \psi(tq(v))\Phi(v)$.
	\item  $\displaystyle w_1\cdot \Phi(v) = \gamma(q)  \widehat{\Phi}(v)$, where $w_1=\widetilde{w}_\al(1)$.
	\item $\displaystyle \widetilde{h}_\al(y) \cdot\Phi(v) = |y|^{d/2} \frac{\gamma(yq)}{\gamma(q)} \Phi(yv)$.
\end{enumerate}

\begin{proposition} The implied action of $\widetilde{\SL}_2(k)$ on $S(V)$ is well-defined and gives a representation, denoted by $\omega_{\psi,q}$. This representation is genuine when $\dim(V)$ is odd. \end{proposition}
\begin{proof} We omit the proof, which is well-known. \end{proof}

Consider now the special case where $V=k$ and $q(x) = x^2$. The genuine representation $\omega_{\psi,q}$ is not irreducible: if $S^+(k)$ is the subspace of even Schwartz functions (ie: $\Phi(-x) = \Phi(x)$), then $\widetilde{\SL}_2(k)$ preserves this subspace. This gives an irreducible representation, which we will denote by $\omega_\psi^+$.

In \cite{gelbartBook}, Gelbart defines a double cover of $\SL_2(k)$ via an explicit two-cocycle, as follows. For a matrix $s = \mm{a}{b}{c}{d}$ define \[
x(s) = \begin{cases}\quad c&: \text{ if }c \neq 0,\\\quad d&: \text{ if }c = 0.\end{cases}
\]
Define
\[\alpha(g_1,g_2) = (x(g_1),x(g_2))_2(-x(g_1)x(g_2),x(g_1g_2))_2\]
and $\widetilde{\SL}'_2(k)$ as the set of pairs $(g,\zeta)$ with $g \in \SL_2(k)$ and $\zeta \in \{\pm 1\}$ with multiplication
\begin{equation}\label{eqn: Gelbart's SL2}
(g_1,\zeta_1)(g_2,\zeta_2) = (g_1g_2,\alpha(g_1,g_2)\zeta_1 \zeta_2).
\end{equation}
Because of the uniqueness up-to-isomorphism of the nontrivial double cover of $\SL_2(k)$, this double cover is isomorphic to $\widetilde{\SL}_2(k)$.

\subsection{Two double covers of $\GL_2$}\label{Sec: two covers}
We now define two double covers of the group $\GL_2(k)$ and consider extensions of the genuine representation $\omega_\psi^+$ to these groups. Our motivation is to relate a cover arising in our analysis of modular forms on $\widetilde{F}_4(k)$ with one considered in \cite{gelbartPSdistinguished}.

The first construction is given via generators and relations: consider the group $\widetilde{\GL}_2^{(1)}(k)$ generated by $\widetilde{\SL}_2(k)$ and $\widetilde{h_{\al_2}}(t)$ for $t \in k^\times$, subject to the relations: if we let $\al_1$ denote the simple root of $\SL_2$, then
\begin{enumerate}
	\item\label{1here} $\zeta$ is still central;
	\item\label{2here}$\widetilde{h_{\alpha_2}}(t) x_{\pm\alpha_1}(u) \widetilde{h_{\alpha_2}}(t)^{-1} = x_{\pm\alpha_1}(t^{\langle \alpha_2^\vee, \pm\alpha_1 \rangle} u)$ where $\langle \alpha_2^\vee, \pm\alpha_1 \rangle = \mp1$.
	\item $\widetilde{h_{\alpha_2}}(s) \widetilde{h_{\alpha_2}}(t) = \widetilde{h_{\alpha_2}}(st) (s,t)_2$;
\end{enumerate}
One can prove from these relations the following additional relations:
\begin{enumerate}
	\setcounter{enumi}{3}
	\item the commutator $\{\widetilde{h_{\alpha_1}}(s),\widetilde{h_{\alpha_2}}(t)\} = (s,t)_2$.
	\item $\widetilde{w_{\alpha_1}}(t) \widetilde{h_{\alpha_2}}(u) \widetilde{w_{\alpha_1}}(-t) = (u^{-1},u^{-1}t)_2 \widetilde{h_{\alpha_1}}(u) \widetilde{h_{\alpha_2}}(u)$
\end{enumerate}

Sending $\widetilde{h_{\alpha_2}}(t)$ to $\diag(1,t)$, we obtain a surjective homomorphism $ \pi^{(1)}: \widetilde{\GL}_2^{(1)}(k) \longrightarrow \GL_2(k)$, which we claim is a double covering map extending the cover $\pi: \widetilde{\SL}_2(k)\lra \SL_2(k)$.  It is immediately checked that this map is well-defined.  Moreover, by a Bruhat decomposition argument, one sees that the kernel is exactly the image of $\mu_2(k)$ in $\widetilde{\GL}_2^{(1)}(k)$.  To see that this image is nontrivial, so that $\widetilde{\GL}_2^{(1)}(k)$ is really a double cover of $\GL_2(k)$, we note that $\widetilde{\GL}_2^{(1)}(k)$ so defined is precisely the full inverse image of the subgroup $\GL_2(k)\subset F_4(k)$ in the double cover $\widetilde{F}_4(k)$ described in Section \ref{Sec: double covers gen} where $\GL_2(k)\subset F_4(k)$ denotes the subgroup generated by the subgroup isomorphic to $\SL_2(k)$ associated to the simple root $\al_1$ and the coroot associated to the simple root $\al_2$.
\begin{remark}
	In the literature (for example, \cite{kazhdanPatterson}), one often finds this cover described in terms of the inverse image in $\widetilde{\SL}_3(k)$ of the $(2,1)$-Levi subgroup. We opt for the inclusion into $F_4$ as this better illustrates our interest in this covering group. In any case, we have
	\[
	\widetilde{\GL}_2^{(1)}(k)\subset\widetilde{\SL}_3(k)\subset\widetilde{F}_4(k),
	\]
	where the inclusion $\SL_3\subset F_4$ is the one discussed in Section \ref{sec:grpEmbeddings}.
\end{remark}

Let
\begin{equation}
G^*:=\{ g \in \widetilde{\GL}_2^{(1)}(k): \pi^{(1)}(g)\in \GL_2(k)\text{ has determinant a square in $k^\times$}\}.
\end{equation}
As is easily seen, this is the subgroup of $\widetilde{\GL}_2^{(1)}(k)$ generated by $\widetilde{\SL}_2(k)$ and $\widetilde{h_{\alpha_2}}(t^2)$, $t \in k^\times$.
\begin{lemma}
	The group $G^*$ is generated by $\widetilde{\SL}_2(k)$ and $\widetilde{h_{\alpha_2}}(t^2)$ subject only to the relations defining $\widetilde{\GL_2}(k)$, restricted to the $\widetilde{h_{\alpha_2}}(t^2)$.
\end{lemma}
\begin{proof}
	Let temporarily $G^*_1$ be the group described in the statement of the lemma. Then one has a tautological surjection $G^*_1 \rightarrow G^*$.  Now $G_1^*$ maps to $\GL_2(k)$, with kernel at most $\mu_2(k)$. Now suppose $\ep$ is in the kernel of $G^*_1 \rightarrow G^*$. Then $\ep \in \mu_2(k)$.  But the image of $\mu_2(k)$ in $G^*$ has size two, so $\ep =1$.
\end{proof}

Fix a character $\chi$ of $k^\times$, with $\chi(-1) = 1$.  Let $S^+(k)$ be the Schwartz space of even functions. We then have the genuine representation $\omega_{\psi}^+$ of $\widetilde{\SL}_2(k)$ on $S^+(k)$. Following \cite{gelbartPSdistinguished}, one can extend the action to an action of $G^*$ on $S^{+}(k)$ by letting
\begin{equation*}
\widetilde{h_{\alpha_2}}(a^2)\phi(x) = \chi(a)|a|^{-1/2}\phi(a^{-1}x).
\end{equation*}
\begin{proposition}\label{Prop: extend to ast} The above action gives a well-defined representation of $G^*$ on $S^+(k)$. We denote the resulting representation as $\omega_{\psi,\chi}$.
\end{proposition}
\begin{proof} This is a direct check which we omit.\end{proof}

In \cite{gelbartBook} and \cite{gelbartPSdistinguished}, a different double cover of $\GL_2(k)$ is defined, which we now recall.  For $y \in k^\times$, define
\[
v(y,s) = \begin{cases}\quad 1&: \text{if $c \neq 0$},\\ \:(y,d)_2&:\text{ otherwise,}\end{cases}
\]where $s = \mm{a}{b}{c}{d}$.  Define $s^y = \diag(1,y)^{-1} s \diag(1,y)$. Now, for $\overline{s} = (s,\zeta)\in \widetilde{\SL}'_2(k)$ (defined as in \eqref{eqn: Gelbart's SL2}), let $\overline{s}^y = (s^y,v(y,s)\zeta)$. It is then proved that this gives an action of $k^\times$ on $\widetilde{\SL}'_2(k)$ and one defines $\widetilde{\GL}^{(0)}_2(k)$ to be the semidirect product  $\widetilde{\SL}'_2(k) \rtimes k^\times$.

We now compare the double cover $\widetilde{\GL}^{(0)}_2(k)$ and our $\widetilde{\GL}_2^{(1)}(k)$.  To do this, let $G^{(0)}$ be a group defined as follows.  As a set, it is $\widetilde{\GL}_2^{(1)}(k)$.  The multiplication in $G^{(0)}$ is defined as $$g * h = g\cdot h (\det(g),\det(h))_2,$$ where here $g \cdot h$ is the product in $\widetilde{\GL}_2^{(1)}(k)$.
\begin{proposition}\label{prop:twocovers} The group $G^{(0)}$ is isomorphic to $\widetilde{\GL}^{(0)}_2(k)$.
\end{proposition}

To prove the proposition, we require a few lemmas.  Let, temporarily, $G^{(0)}_1$ be the group generated by $\widetilde{\SL}_2(k)$ and $\widetilde{h_{\alpha_2}}(t)$ for $t \in k^\times$, subject to the relations \eqref{1here}, \eqref{2here}, and
\begin{enumerate}
	\setcounter{enumi}{2}
	\item $\widetilde{h_{\alpha_2}}(s) \widetilde{h_{\alpha_2}}(t) = \widetilde{h_{\alpha_2}}(st) $.
\end{enumerate}

\begin{lemma} The map $G^{(0)}_1 \rightarrow G^{(0)}$ that is the identity on generators is a well-defined isomorphism.
\end{lemma}
\begin{proof} It is clear that the map is a well-defined homomorphism, because the relations satisfied in $G^{(0)}_1$ are again satisfied in $G^{(0)}$.  Moreover, it is clear that the map is surjective, and covers the identity map on the linear group $\GL_2(k)$.  By another Bruhat decomposition argument, the kernel of $G^{(0)}_1 \rightarrow \GL_2(k)$ is at most $\mu_2(k)$.  It follows that the kernel is exactly $\mu_2(k)$ and $G^{(0)}_1 \rightarrow G^{(0)}$ is an isomorphism.\end{proof}

\begin{lemma} Fix $t \in k^\times$.  Define a map $\phi_t:\widetilde{\SL}_2(k) \rightarrow \widetilde{\SL}_2(k)$ on generators as $\zeta \mapsto \zeta$, $x_{\alpha_1}(u) \mapsto x_{\alpha_1}(t^{-1}u)$, $x_{-\alpha_1}(u) \mapsto x_{-\alpha_1}(tu)$.  Then this map is a well-defined isomorphism. \end{lemma}
\begin{proof} One checks that the relations in the first copy of $\widetilde{\SL}_2(k)$ are satisfied in the second copy.  Thus the map is a well-defined surjection.  Replacing $t$ by $t^{-1}$ gives a well-defined inverse.  Thus, $\phi_t$ is an isomorphism.
\end{proof}

\begin{lemma} The map $\widetilde{\SL}_2(k) \rtimes_{\phi_t} \langle \widetilde{h_{\alpha_2}}(t) \rangle \rightarrow G^{(0)}_1$ defined for $h \in \widetilde{\SL}_2(k)$ as
\[
(h, \widetilde{h_{\alpha_2}}(t)) \longmapsto h \widetilde{h_{\alpha_2}}(t)\]
is a well-defined isomorphism.\end{lemma}
\begin{proof}  Checking that it is well defined amounts to the relation that $\widetilde{h_{\alpha_2}}(t_1) h_2 \widetilde{h_{\alpha_2}}(t_1)^{-1} = \phi_{t_1}(h_2)$ in $\widetilde{\SL}_2(k)$, which is clear.
	
	The inverse map is $G^{(0)}_1 \rightarrow \widetilde{\SL}_2(k) \rtimes_{\phi_t} \langle \widetilde{h_{\alpha_2}}(t)\rangle$ given by the obvious map on generators.  The relations defining $G^{(0)}_1$ are again satisfied in the semi-direct product, so the map is well-defined.  It is clear that these maps are inverses to each other, giving the lemma.
\end{proof}

\begin{proof}[Proof of Proposition \ref{prop:twocovers}] Given the previous lemmas, we simply must check that the semi-direct product defining $\widetilde{\GL}^{(0)}_2(k)$ is the same as the one given by $\phi_t$, and one must map our $\widetilde{\SL}_2(k)$ to $\widetilde{\SL}'_2(k)$.  For this latter task, one checks that $\mm{1}{}{c}{1} \mapsto (\mm{1}{}{c}{1},1)$ is a splitting to $\widetilde{\SL}'_2(k)$.  (Use the identity on Hilbert symbols $(a,b)_2(-ab,a+b)_2=1$.)  This splitting pins down the isomorphism $\widetilde{\SL}_2(k) \rightarrow \widetilde{\SL}'_2(k)$. One finds that $\widetilde{w}_\alpha(t) \mapsto (\mm{}{t}{-t^{-1}}{},1)$ and that $\widetilde{h}_{\alpha_1}(t) \mapsto (\diag(t,t^{-1}),(t,t)_2)$.  We omit the rest of the proof.
\end{proof}

 Note that this shows that the subgroup $G^\ast\subset \widetilde{\GL}_2^{(1)}(k)$ naturally occurs as a subgroup of $\widetilde{\GL}^{(0)}_2(k)$, at least once we fix the above isomorphism $G^{(0)} \cong \widetilde{\GL}^{(0)}_2(k)$.

\subsection{The Weil representation for $\GL_2$}\label{Sec: Weil GL2}
The Weil representation of $\widetilde{\GL}_2^{(1)}(k)$ is defined as
\begin{equation}\label{eqn: Weil on GL2}
\Omega_{\psi,\chi}^{(1)}:=\Ind_{G^*}^{\widetilde{\GL}_2^{(1)}(k)}(\omega_{\psi,\chi}).
\end{equation}  In order to use results of \cite{gelbartPSdistinguished}, we will need to compare $ \Omega_{\psi,\chi}^{(1)}$ with the Weil representation studied in \emph{loc. cit.}, which is defined as
\begin{equation*}
\Omega_{\psi,\chi}^{(0)} := \Ind^{\widetilde{\GL}^{(0)}_2(k)}_{G^*}(\omega_{\psi,\chi}) \simeq \Ind^{G^{(0)}}_{G^*}(\omega_{\psi,\chi}).
\end{equation*}

To compare these representations, suppose $V^{(1)}$ is a representation of $\widetilde{\GL}_2^{(1)}(k)$.  Define a representation $V^{(0)}$ of $G^{(0)}$ by letting $V^{(0)} = V^{(1)}$ as vector spaces, with action $g * v = \frac{\gamma(\det(g) q)}{\gamma(q)} gv$.  Here $\ga(q)$ is as in \eqref{eqn: weil index}.

\begin{proposition} \label{prop:compare}Suppose $S$ is a representation of $G^*$, $V^{(1)} = \Ind_{G^*}^{\widetilde{\GL}_2^{(1)}(k)}(S)$, $V^{(0)}$ is as above, and let $V' = \Ind_{G^*}^{G^{(0)}}(S)$.  As representations of $\widetilde{\GL}_2^{(0)}(k)$, $V^{(0)}$ is isomorphic to $V'$ via the map $$f(g) \mapsto \frac{\gamma(\det(g) q)}{\gamma(q)} f(g).$$  In particular, the map $$\left(\Omega_{\psi,\chi}^{(1)}\right)^{(0)} \longrightarrow \Omega_{\psi,\chi}^{(0)}$$ given by $f(g) \mapsto \frac{\gamma(\det(g) q)}{\gamma(q)} f(g)$ is an isomorphism.
\end{proposition}
\begin{proof} This is a simple check. \end{proof}

We may now derive certain properties of $\Omega_{\psi,\chi}^{(1)}$ from the corresponding results of Gelbart--Piatetski-Shapiro \cite{gelbartPSdistinguished}.  Let temporarily $U(k) =\{x_{\al_1}(t): t \in k\}$ denote the unipotent radical of the upper triangular Borel subgroup of $\GL_2(k)$.  This subgroup splits uniquely into both $\widetilde{\GL}_2^{(1)}(k)$ and $\widetilde{\GL}^{(0)}_2(k)$, so let $U(k)$ also denote the image under the splitting.  If $V$ is a representation of either double cover, and $t \in k^\times$, a linear functional $L: V \rightarrow \C$ is said to be a $(U,\psi_t)$-functional if $L(x_\al(u) v) = \psi(t u) L(v)$ for all $u \in k$ and $v \in V$.
\begin{proposition}\label{prop:L0} The space of $(U,\psi_t)$-functionals on $\Omega_{\psi,\chi}^{(1)}$ is one-dimensional.  A basis of this space of functionals is given by
	\[f \in \Omega_{\psi,\chi}^{(1)} \mapsto f(h_{\alpha_2}(t^{-1}))(1).\]
\end{proposition}
\begin{proof} It is immediately checked that the map $f \mapsto f(h_{\alpha_2}(t^{-1}))(1)$ is a non-zero $(U,\psi_t)$-functional.  Thus, the key statement is the multiplicity-one claim.  For the representation $\Omega_{\psi,\chi}^{(0)}$, this is due to Gelbart--Piatetski-Shapiro \cite{gelbartPSdistinguished}.  Comparing $\Omega_{\psi,\chi}^{(1)}$ with $\Omega_{\psi,\chi}^{(0)}$ using Proposition \ref{prop:compare}, we see that
	\[
	\mathrm{Hom}_{U}(\Omega_{\psi,\chi}^{(1)},\psi_t)=\mathrm{Hom}_{U}(\Omega_{\psi,\chi}^{(0)},\psi_t);
	\]
	the multiplicity one for $\Omega_{\psi,\chi}^{(1)}$ follows.\end{proof}

We will also require some results on invariant vectors of $\Omega_{\psi,\chi}^{(1)}$. To state the first result, let $k=\Q_2$ and let $\Gamma_{1,\GL_2}(4)$ be the subgroup of $\GL_2(k)$ generated by $x_{\alpha}(u)$, $x_{-\alpha}(4u)$, $h_{\alpha_1}(t), h_{\alpha_2}(t)$ with $u \in \Z_2$ and $t \in 1 + 4\Z_2$. Using the generators and relations, an easy analogue of Theorem \ref{thm:iwahoriSplit} implies that $\Gamma_{1,\GL_2}(4)$ splits the cover $\widetilde{\GL}_2^{(1)}(\Q_2)$; we set $\Gamma^\ast_{1,\GL_2}(4)$ for the image of the splitting.  Similarly, we denote by $\Gamma_{1,\SL_2}^*(4)$ the subgroup of $\widetilde{\SL}_2(\Q_2)$ generated by $x_{\alpha}(u)$, $x_{-\alpha}(4u)$, $h_{\alpha_1}(t)$ with $u \in \Z_2$ and $t \in 1 + 4\Z_2$.
\begin{corollary}\label{cor:L02} Let $L_t$ denote the  non-zero $(U,\psi_t)$-functional given in the statement of Proposition \ref{prop:L0}.  If $t = 1$ or $t=-1$, there is a $\Gamma^\ast_{1,\GL_2}(4)$-invariant vector $f_t \in \Omega_{\psi,\chi}^{(1)}$ so that $L_t(f_t) = 1$.  In particular, if $t=1$ or $t=-1$ and a $(U,\psi_t)$-functional $L$ on $\Omega_{\psi,\chi}^{(1)}$ vanishes on the $\Gamma^\ast_{1,\GL_2}(4)$-invariant vectors, then $L = 0$.
\end{corollary}
\begin{proof} Let $\phi_0 \in S^{+}(\Q_2)$ be the characteristic function of $\Z_2$.  Define $f_1 \in \Omega_{\psi,\chi}^{(1)}$ via $f_1(1) = \phi_0$, $f_1(h_{\alpha_2}(5)) = \phi_0$ and if $g \notin G^* \cup G^* h_{\alpha_2}(5)$ then $f_1(g) = 0$.  Define $f_{-1} \in \Omega_{\psi,\chi}^{(1)}$ via $f_{-1}(h_{\alpha_2}(-1)) = \phi_0$, $f_{-1}(h_{\alpha_2}(-5)) = \phi_0$, and if $g \notin G^*h_{\alpha_2}(-1) \cup G^* h_{\alpha_2}(-5)$ then $f_{-1}(g) = 0$.
	
	By construction, $L_t(f_t)=1$ for $t=1,-1$.  One readily verifies that $f_1$ and $f_{-1}$ are $\Gamma^\ast_{1,\GL_2}(4)$-invariant: For this, one uses that $\phi_0$ is $\Gamma_{1,\SL_2}^*(4)$ invariant under the action of $\omega_{\psi}$, and that $h_{\alpha_2}(5)$, $h_{\alpha_2}(-1)$ normalize $\Gamma_{1,\SL_2}^*(4)$.  The corollary follows.
\end{proof}

We have an analogous statement at the odd primes.  Let $k=\Q_p$ with $p$ odd and let $\GL_2^*(\Z_p)$ be the subgroup of $\widetilde{\GL}_2^{(1)}(k)$ generated by $x_{\pm \alpha}(u)$, $\widetilde{h}_{\alpha_2}(t)$ with $u \in \Z_p$ and $t \in \Z_p^\times$; this is the image of a splitting of $\widetilde{\GL}_2^{(1)}(\Q_p)$ over $\GL_2(\Z_p)$.

\begin{lemma}\label{lem:f0p} Suppose $p$ is odd.  Let $\phi_0 \in S^{+}(\Q_p)$ be the characteristic function of $\Z_p$.  Let $\{1,\mu,p,\mu p\}$ with $\mu \in \Z_p^\times$ be representatives for $\Q_p^\times/(\Q_p^\times)^2$.  Define $f_0 \in \Ind_{G^*}^{\widetilde{\GL}_2^{(1)}(k)}(S^{+}(\Q_p))$ by $f_0(1) = \phi_0$, $f_0(\widetilde{h}_{\alpha_2}(\mu)) = \phi_0$, $f_0(\widetilde{h}_{\alpha_2}(p)) = 0$ and $f_0(\widetilde{h}_{\alpha_2}(p \mu)) = 0$. Then $f_0$ is $\GL_2^*(\Z_p)$-invariant.\end{lemma}
\begin{proof} This is a relatively direct check, which we omit. \end{proof}

It is proved in \cite{gelbartPSdistinguished} that $\Omega_{\psi,\chi}^{(0)}$, and thus $\Omega_{\psi,\chi}^{(1)}$, is irreducible.  We will see in Section \ref{sec:JFs} that $\Omega_{\psi,\chi}^{(1)}$ embeds in a certain principal series representation, from which it follows that the space of $\GL_2^*(\Z_p)$-invariant vectors of $\Omega_{\psi,\chi}^{(1)}$ is at most one-dimensional \cite[Section 9.2]{GanGao}, and thus exactly one-dimensional, spanned by the $f_0$ of Lemma \ref{lem:f0p}.  We obtain the following corollary.

\begin{corollary}\label{cor:L0p} Suppose $t=1$ or $t=-1$, $k=\Q_p$ with $p$ odd, and $L$ is $(U,\psi_t)$-functional that is $0$ on the unique line of $\GL_2^*(\Z_p)$-invariant vectors of $\Omega_{\psi,\chi}^{(1)}$.  Then $L = 0$.
\end{corollary}
\begin{proof}
	This follows from a similar argument to the $p=2$ case.
\end{proof}

\section{Jacquet functors}\label{sec:JFs}
For any finite prime $p$, let $V_{min}=\Pi_{min,p}$ denote the local component of $\Pi_{min}$ at $p$. Recall that $Q = L U_{Q}$ denotes the standard maximal parabolic of $F_4$ associated to the simple root $\alpha_2$.   In this subsection, we identify the Jacquet module $V_{min,U_Q}$ of $V_{min}$ with respect to $U_Q$ with the representation $\Omega_{\psi,\chi}^{(1)}$ of $ \widetilde{\GL}_2^{(1)}(\Q_p)$ considered in Section \ref{Sec: Weil GL2}. For this to make sense, we first explicate a map $\widetilde{L}(\Q_p) \rightarrow \widetilde{\GL}_2^{(1)}(\Q_p)$.

Recall the subgroup $\SL_3(\Q_p)$ of $F_4(\Q_p)$ as described before Lemma \ref{lem:inj1}.
\begin{proposition}\label{Prop: mod sl3} The group $\SL_3(\Q_p)$ splits into $\widetilde{F}_4(\Q_p)$, is normal in $\widetilde{L}(\Q_p)$, and one has $$\widetilde{L}(\Q_p)/\SL_3(\Q_p) \simeq \widetilde{\GL}_2^{(1)}(\Q_p).$$\end{proposition}
\begin{proof} We first note that $\SL_3(\Q_p)$ is a normal subgroup of $L(\Q_p)$ such that
	\[
	{L}(\Q_p)/\SL_3(\Q_p) \simeq \GL_2(\Q_p).
	\] That $\SL_3(\Q_p)$ splits into $\widetilde{F}_4(\Q_p)$ is Lemma \ref{Lem: SL split}.
	
	To see that this $\SL_3(\Q_p)$ is normal, let $s$ denote the splitting of $\SL_3(\Q_p)$ into $\widetilde{F}_4(\Q_p)$.  Because $\SL_3(\Q_p)$ is its own derived group, the splitting $s$ is unique.  Now, let $g' \in \widetilde{L}(\Q_p)$ with image $g \in L(\Q_p)$.  Define $s_g: \SL_3(\Q_p) \rightarrow \widetilde{F}_4(\Q_p)$ as $s_g(h) = g's(g^{-1}hg)(g')^{-1}$.  Since $\SL_3(\Q_p)$ is normal in $L(\Q_p)$, $s_g$ is another splitting; thus $s_g = s$ by uniqueness. This implies $(g')^{-1}s(h)g' = s(g^{-1}hg)$, proving $s\left(\SL_3(\Q_p)\right)$ is normal.
	
	Finally, we have a map $\widetilde{\GL}_2^{(1)}(\Q_p) \rightarrow \widetilde{L}(\Q_p)$, because we know that the relations defining $\widetilde{\GL}_2^{(1)}(\Q_p)$ are satisfied in $\widetilde{L}(\Q_p)$.  This induces $\widetilde{\GL}_2^{(1)}(\Q_p) \rightarrow \widetilde{L}(\Q_p)/\SL_3(\Q_p)$. The latter group is a non-split double cover of $\GL_2(\Q_p)$, as is $\widetilde{\GL}_2^{(1)}(\Q_p)$. Since the map $\widetilde{\GL}_2^{(1)}(\Q_p) \rightarrow \widetilde{L}(\Q_p)/\SL_3(\Q_p)$ is defined in terms of generators and relations, it fits into a commutative diagram
	\[
	\begin{tikzcd}
	1\ar[r]&\mu_2\ar[d,"="]\ar[r]&\widetilde{\GL}_2^{(1)}(\Q_p)\ar[d]\ar[r]&\GL_2(\Q_p)\ar[d,"="]\ar[r]&1
	\\1\ar[r]&\mu_2\ar[r]&\widetilde{L}(\Q_p)/\SL_3(\Q_p)\ar[r]&\GL_2(\Q_p)\ar[r]&1,
	\end{tikzcd}
	\]
	and is thus an isomorphism.
\end{proof}

Let $\chi_{exc}$ denote the unique exceptional character of $Z(\widetilde{T}(\Q_p))$; by an abuse of notation, we use the same symbol for the extension to $T_\ast(\Q_p)$ defined by setting
\begin{equation}\label{eqn: extend exceptional character}
\chi_{exc}(h_{\al_1}(t)) = |t|^{1/2} \frac{\gamma(q)}{\gamma(tq)}
\end{equation}
for $t\in \Q_p$; here $\ga(q)$ is defined in \eqref{eqn: weil index}. We set $B_L = L\cap B=TU_{B_L}$ the associated Borel subgroup of the Levi subgroup $L$ and set $B_{L,\ast}(\Q_p) = T_\ast(\Q_p)U_{B_L}(\Q_p)$.

It follows from \cite[Section 6]{lokeSavin} that there is an embedding $V_{min} \hookrightarrow \Ind_{{B}_\ast(\Q_p)}^{\widetilde{F}_4(\Q_p)}(\delta_{B}^{1/2} \chi_{exc}^{-1})$ and thus
\begin{equation}\label{eqn: into the induction}
V_{min,U_Q} \longrightarrow \\Ind_{{B}_\ast(\Q_p)}^{\widetilde{Q}(\Q_p)}(\delta_{B}^{1/2} \chi_{exc}^{-1})\cong \Ind_{B_{L,\ast}(\Q_p)}^{\widetilde{L}(\Q_p)}(\delta_{B}^{1/2} \chi_{exc}^{-1}).
\end{equation}
This latter map sends a function $f \in \Ind_{{B}_\ast(\Q_p)}^{\widetilde{F}_4(\Q_p)}(\delta_{B}^{1/2} \chi_{exc}^{-1})$ to its restriction $f|_{\widetilde{Q}}$.  It is clear that this factors through the Jacquet functor $V_{min,U_Q}$.  It is also clear that the map is non-zero.

\begin{proposition} The Jacquet functor $V_{min,U_Q}$ is irreducible as a representation of $\widetilde{L}(\Q_p)$.  Moreover, the representation $\Ind_{B_{L,\ast}(\Q_p)}^{\widetilde{L}(\Q_p)}(\delta_{B}^{1/2} \chi_{exc}^{-1})$ has a unique irreducible subrepresentation, which is thus identified with $V_{min,U_Q}$ under the above morphism.
\end{proposition}
\begin{proof} To prove the irreducibility of $V_{min,U_Q}$, we follow the argument of \cite[Theorem 2.2, 2.3]{BFGsmall}. This relies on the fact that the Jacquet functor of $\Ind_{{B}_\ast(\Q_p)}^{\widetilde{F}_4(\Q_p)}(\delta_B^{1/2}\chi_{exc})$ associated to any standard non-minimal parabolic subgroup has no supercuspidal subquotients \cite[Corollary 2.13(b)]{BZinduced}.
	
	Suppose $V_1 \subseteq V_{min,U_Q}$ is an $\widetilde{L}(\Q_p)$-invariant subspace, and ${V}_2$ the quotient of $V_{min,U_Q}$ by $V_1$, giving the short exact sequence of $\widetilde{L}(\Q_p)$-representations
	\[
	0\lra V_1\lra V_{min,U_Q}\lra {V}_2\lra0.
	\] By exactness of the Jacquet functor down to the unipotent radical $U_{B_L}$ of the Borel subgroup of $L$, we obtain
	\[
	0\lra V_{1,U_{B_L}}\lra \left(V_{min,U_Q}\right)_{U_{B_L}}\cong V_{min, U_B} \lra {V}_{2,U_{B_L}}\lra0.
	\]
	The Jacquet functor $V_{min, U_B}$ associated to the Borel subgroup of $F_4$ is irreducible \cite[Proposition 6.4]{lokeSavin}. In particular, either $V_{1,U_{B_L}}=0$ or $V_{2,U_{B_L}}=0$; suppose it is $V_{1,U_{B_L}}=0$.
	
	If $V_1$ has no non-zero Jacquet modules, we must have $V_1=0$ by \cite[Corollary 2.13(b)]{BZinduced}. Otherwise, let $P_L=M_LN_L\subset L$ be the standard parabolic subgroup that is minimal among those such that $V_{1,N_L}\neq0$. By assumption $P_L\neq B_L$, so that $V_{1,N_L}$ is a non-zero supercuspidal representation of $\widetilde{M}_L(\Q_p)$ and also a subquotient of the Jacquet module
	$
	\Ind_{{B}_\ast(\Q_p)}^{\widetilde{F}_4(\Q_p)}(\delta_B^{1/2}\chi_{exc})_{N_L},
	$
	which is a contradiction. An argument is identical if we assume $V_{2,U_{B_L}}=0$, completing the proof of the irreducibility of $V_{min,U_Q}$.
	
	The proof that $\Ind_{{B}_\ast(\Q_p)}^{\widetilde{Q}(\Q_p)}(\delta_{B}^{1/2} \chi_{exc}^{-1})$ has a unique irreducible subrepresentation is exactly the same as the semisimple case treated in \cite{lokeSavin}. Now recall that one has a non-zero map \eqref{eqn: into the induction}, giving the final claim.
\end{proof}

Pulling back along the quotient map from Proposition \ref{Prop: mod sl3}, we now analyze the representation $\Omega_{\psi,\chi}^{(1)}$ as a representation of $\widetilde{L}(\Q_p)$.
Define the multiplicative character $\chi(v) = |v|^{3/2}$, and recall that $\chi$ determines an extension of the representation on $S^{+}(\Q_p)$ from $\widetilde{\SL}_2(\Q_p)$ to the group $G^*$; see Proposition \ref{Prop: extend to ast}. Consider the corresponding Weil representation $\Omega_{\psi,\chi}^{(1)}=Ind_{G^*}^{\widetilde{\GL}_2^{(1)}(\Q_p)}(S^{+}(\Q_p))$ of $\widetilde{\GL}_2^{(1)}(\Q_p)$.

\begin{lemma} Consider the functional
	\begin{align*}
	B:\Omega_{\psi,\chi}^{(1)} &\longrightarrow \C\\
	B(f) &= f(1)(0).
	\end{align*}
	Then $B(t \cdot f) = (\delta_{B}^{1/2}\chi_{exc}^{-1})(t) B(f)$ for all $t \in T_{*}(\Q_p)$, where $\chi_{exc}$ is is the exceptional character $\chi_{exc}$ of $T_{*}(\Q_p)$ given by \eqref{eqn: extend exceptional character}.
\end{lemma}
\begin{proof} Using the formulas in Section \ref{Sec: Weil rep on SL2}, one has
	$$B(h_{\alpha_1}(t) \cdot f) = |t|^{1/2} \frac{\gamma(t q)}{\gamma(q)} B(f)$$ and $$B(h_{\alpha_2}(v^2) \cdot f) = \chi(v) |v|^{-1/2} B(f) = |v| B(f).$$
	
	Moreover, $B(h_{\alpha_3}(v) \cdot f) =B(h_{\alpha_4}(v) \cdot f)=B(f).$ Now observe that for each simple root $\delta_{B}^{1/2}(h_\al(t)) = |t|$.  The lemma now follows from the definition of $\chi_{exc}$.
\end{proof}

Because $\Omega_{\psi,\chi}^{(1)}$ is irreducible \cite{gelbartPSdistinguished}, Frobenius reciprocity provides an embedding of $\widetilde{L}(\Q_p)$-representations
\[\Omega_{\psi,\chi}^{(1)} \longrightarrow \Ind_{B_{L,\ast}(\Q_p)}^{\widetilde{L}(\Q_p)}(\delta_{B}^{1/2}\chi_{exc}^{-1}).\]

\begin{corollary}\label{Cor: Weil in Jacquet} The Jacquet module $V_{min,U_Q}$ is isomorphic to $\Omega_{\psi,\chi}^{(1)}$.
\end{corollary}

\section{The minimal modular form}\label{sec:minMF}
We return now to the global setting. Let $J=H_3(\Q)$ be the symmetric $3\times 3$ matrices with $\Q$ coefficients.  Fourier coefficients of modular forms on $F_4$ are parameterized by elements $\omega = (a,b,c,d) \in W_J(\Q)$ where
\[
W_J(\Q) = \Q\oplus J\oplus J^\vee \oplus \Q = \Q \oplus J \oplus J \oplus \Q
\]
as $J^\vee$ is identified with $J$ via the trace pairing.  In this subsection, we show that we may choose $v_f\in \Pi_{min,f}$ such that the modular form $\Theta_{F_4}:=\theta(v_f)$ satisfies that it has
\begin{enumerate}
	\item $\Ga_{F_4}(4)$ level and
	\item non-zero $(0,0,0,1)$-Fourier coefficient.
\end{enumerate}

This will rely on the following purely local result. Let $p$ be a finite prime. Denote by $K^\ast_p$ the compact open subgroup of $\widetilde{F}_4(\Q_p)$ at $p$ introduced in Section \ref{sec:splittings}, so that $K^\ast_2 = K_R'(4)$ and $K^\ast_p=F_4^\ast(\Z_p)$ for odd $p$.  Let $U_R = U_{\alpha_1} U_Q$ be the unipotent radical of the  parabolic subgroup $R\subset F_4$ associated to the simple roots $\al_1$ and $\al_2$; it splits canonically into $\widetilde{F}_4(\Q_p)$.  For $t=1$ or $t=-1$, define a character $\psi_{1,t}$ on $U_{R}(\Q_p)$ by using the fixed additive character $\psi_t$ on the root space $U_{\alpha_1}$.
\begin{theorem}\label{thm:localFC} Let $V_p$ denote the vector space underlying $\Pi_{min,p}$. Suppose $L$ is $(U_{R},\psi_{1,t})$-functional such that $L$ is $0$ on the $K^\ast_p$-fixed vectors of $V_p$.  Then $L = 0$. In particular, the twisted Jacquet functor associated to $(U_{R},\psi_{1,t})$ induces a surjection
	\[
	V_p^{K^\ast_p}\lra V_{p,(U,\psi_{1,t})},
	\]
	which is an isomorphism for $p\neq2$.
\end{theorem}
\begin{proof}
	There are two cases: $p=2$ and $p >2$.
	
	Let us first handle the case $p$ odd.  First observe that $V_p^{K^\ast_p} \rightarrow V_{U_Q}^{\widetilde{L} \cap K^\ast_p}$ is well-defined and non-zero.  Indeed, it is clear that the map is well-defined.  To see that it is non-zero, consider the further map to $V_{U_B}$ (recall $U_B$ denotes the unipotent radical of the Borel.)  Recalling the embedding of $V_p$ into $\Ind_{{B}_\ast(\Q_p)}^{\widetilde{F}_4(\Q_p)}(\delta_B^{1/2} \chi_{exc}^{-1})$, we may consider the linear functional on $V_p$ given by composing this map can with the evaluation-at-$1$ map: this gives a non-zero functional
	\[
	V_p\lra V_{U_Q}\lra V_{U_B}\lra \C.
	\] The spherical vector in this induced representation is non-zero at $t=1$, so that this functional is non-vanishing on $V_p^{K^\ast_p}$. In particular, the composition
	\begin{equation}\label{eqn:VpKp}
	V_p^{K^\ast_p} \longrightarrow V_{U_Q}^{\widetilde{L} \cap K^\ast_p}\longrightarrow V_{U_B}^{\widetilde{T} \cap K^\ast_p}
	\end{equation}
	is non-zero.

	Now observe that both $V_p^{K^\ast_p}$ and $V_{p,U_Q}^{\widetilde{L} \cap K^\ast_p}$ are at most one dimensional \cite[Section 9.2]{GanGao}.  In fact, each is exactly one-dimensional: in the case of $V_p$, this follows from the intertwining operator calculations of \cite{lokeSavin}. In the case of $V_{p,U_Q}$, it now follows from the non-vanishing of the map \eqref{eqn:VpKp}, and in any case, we constructed a spherical vector in Lemma \ref{lem:f0p}.
	The claim of the theorem now follows by Corollary \ref{cor:L0p} and the isomorphism
	\[
	V_{p,(U_R,\psi_{1,t})}\cong \left(V_{p,U_Q}\right)_{U_{\al_1},\psi_t}\cong \left(\Omega_{\psi,\chi}^{(1)}\right)_{U_{\al_1},\psi_t}.
	\]

	We now discuss the case of $p=2$.  First observe that $K^\ast_2=K'_R(4)$ has an Iwahori factorization with respect to $Q = L U_Q$, as proved in Corollary \ref{cor:R'IwahoriQ}.  Now, it follows by \cite[Theorem 3.3.3]{casselmanBook} that $V^{K_R^*(4)} \rightarrow V_{U_Q}^{\widetilde{L} \cap K_R^*(4)}$ is surjective.  In light of Corollary \ref{Cor: Weil in Jacquet}, the claim of the theorem thus follows as above by Corollary \ref{cor:L02}.
\end{proof}

\begin{remark}
    The $p$ odd case may also be handled in a similar fashion to the $p=2$ case by instead considering the subgroup $I^\ast_p\subset K_p^\ast$ associated to the Iwahori subgroup. The only non-trivial step is noting that
    \[
    V_p^{I_p^\ast}\cong V_p^{K_p^\ast}
    \]
    as both are one dimensional. This follows for $K_p^\ast$ as noted above and follows for $I_p^\ast$ as $V_p=\Pi_{min,p}$ corresponds to the trivial representation of the Iwahori--Hecke algebra under the Shimura correspondence proved in \cite[Section 9]{lokeSavin}. We thank Gordan Savin for pointing this out to us.
\end{remark}

Using Theorem \ref{thm:localFC}, we obtain the following corollary, completing the proof of Theorem \ref{thm: exists theta}.
\begin{corollary}\label{Cor: first coeff} There is a quaternionic modular form $\Theta_{F_4}$ of weight $\frac{1}{2}$ on $\widetilde{F}_4$ with $\Ga_{F_4}(4)$ level and non-zero $(0,0,0,1)$-Fourier coefficient.
\end{corollary}
\begin{proof} Let $\omega_1:=(0,0,0,1)\in W_J(\Q)$ and consider the $\omega_1$-Fourier coefficient
	\[
	\theta\longmapsto
	\int_{[N_J]} \theta(n) \psi^{-1}({\la\omega_1, \overline{n}\ra})dn,
	\]
	where $\theta$ is a vector in the space of automorphic forms $\Pi_{min}$. By \cite[Proposition 3]{ginzburgF4}, this gives a non-zero linear functional $L_{\omega_1}$ on  $\Pi_{min}$; that is, there are vectors in $\Pi_{min}$ with nonzero $\omega_1$-Fourier coefficient.  Moreover, such a vector can be chosen to be a quaternionic modular form (in other words, to lie in the minimal $\widetilde{K}_\infty$-type at the archimedean place) by the explicit formula for the generalized Whittaker function proved in Theorem \ref{thm:FE}.  Indeed, a corollary of the proof of the explicit formula is that there is a unique moderate growth $(N_{J}(\R),\psi(\langle \omega_1, - \rangle ))$-equivariant functional on $\Pi_{min,\infty}$ up to scalar multiple, and these functionals are nonvanishing on the minimal $\widetilde{K}_\infty$-type in $\Pi_{min,\infty}$.
	
	Now consider the linear map on $\Pi_{min,f}$ given by $v_f \mapsto L_{\omega_1}(\theta(v_f))$; see equation \eqref{eqn:theta} for the notation.  By what was just said, this map is non-zero on $\Pi_{min,f}$. Moreover, \cite[Proposition 4]{ginzburgF4} implies that for any $\theta$, we have
	\[
	\int_{[N_J]} \theta(n)\psi^{-1}({\la\omega_1, \overline{n}\ra})dn =\int_{[N_S]}  \left(\int_{[N_J]} \theta(nn') \psi^{-1}({\la\omega_1, \overline{n}\ra})dn\right)dn',
	\]
	where $N_S$ denote the unipotent radical of the Siegel parabolic subgroup of $H_J = \GSp_6$. But
	\[
	\int_{[N_S]}  \left(\int_{[N_J]} \theta(nn') \psi^{-1}({\la\omega_1, \overline{n}\ra})dn\right)dn' =\int_{[U_R]}\theta(u) \psi_{1,-1}^{-1}(u)du,
	\]
	where $U_R$ is the unipotent radical of the parabolic $R$ from Theorem \ref{thm:localFC} and $\psi_{1,-1} = \prod_v\psi_{1,-1,v}$ is the global analogue of the character considered locally. By that result, the non-zero linear map on $\Pi_{min,f}$ given by $v_f \mapsto L_{\omega_1}(\theta(v_f))$ does not vanish on the $\prod_p{K^\ast_p}$-invariant vectors.  The corollary follows by switching to semi-classical notation.
\end{proof}

\bibliography{nsfANT2020new}
\bibliographystyle{amsalpha}
\end{document}